\pdfoutput=1
\documentclass[11pt]{amsart}
\usepackage[paperheight=279mm,paperwidth=18cm,textheight=26cm,textwidth=14cm]{geometry}
\usepackage[T1]{fontenc}
\usepackage[utf8]{inputenc}
\usepackage[unicode]{hyperref}
\hypersetup{
bookmarks=true,
colorlinks=true,
citecolor=[rgb]{0,0,0.5},
linkcolor=[rgb]{0,0,0.5},
urlcolor=[rgb]{0,0,0.75},
pdfpagemode=UseNone,
pdfstartview=FitH,
pdfdisplaydoctitle=true,
pdftitle={p-variation estimates for martingale transforms},
pdfauthor={Peter Friz, Pavel Zorin-Kranich},
pdflang=en-US
}

\usepackage{amsmath}
\usepackage{amssymb}
\usepackage{amsthm}
\usepackage{xspace}
\usepackage{mathtools}

\usepackage{forest}
\forestset{
  default preamble={
    for tree={rectangle,minimum size=7pt,anchor=west,draw,grow'=north,l=0pt,l sep=4pt,font=\tiny,inner sep=1pt}
  }
}

\usepackage[style=alphabetic]{biblatex}
\numberwithin{equation}{section}
\theoremstyle{plain}
\newtheorem{theorem}{Theorem}[section]
\newtheorem{proposition}[theorem]{Proposition}
\newtheorem{lemma}[theorem]{Lemma}
\newtheorem{corollary}[theorem]{Corollary}
\newtheorem{definition}[theorem]{Definition}

\theoremstyle{remark}
\newtheorem{remark}[theorem]{Remark}
\newtheorem{claim}[theorem]{Claim}
\newtheorem{example}[theorem]{Example}

\DeclarePairedDelimiter\abs{\lvert}{\rvert}
\DeclarePairedDelimiter\norm{\lVert}{\rVert}
\DeclarePairedDelimiter\floor{\lfloor}{\rfloor}

\def\E{\mathbb{E}}
\DeclarePairedDelimiterXPP\EE[1]{\E}{\lparen}{\rparen}{}{\renewcommand\given{\SetSymbol[\delimsize]}#1} 
\providecommand\given{}
\newcommand\SetSymbol[1][]{%
\nonscript\:#1\vert
\allowbreak
\nonscript\:
\mathopen{}}
\DeclarePairedDelimiterX\Set[1]\{\}{%
\renewcommand\given{\SetSymbol[\delimsize]}
#1
}

\newcommand\cadlag{c\`{a}dl\`{a}g}
\renewcommand{\P}{\mathbb{P}}

\newcommand{\pred}{\mathrm{pred}}
\newcommand{\bv}{\mathrm{bv}}
\newcommand{\Pimax}{\Pi^*}
\newcommand{\one}{\mathbf{1}}
\newcommand{\bmax}{\vee} 
\newcommand{\bmin}{\wedge} 
\def\hPi{\Pi}
\def\cPi{\Pi}
\def\MYp{M Y'}

\newcommand{\mesh}{\operatorname{mesh}}
\newcommand{\dmesh}{\operatorname{d-mesh}}

\newcommand*{\Z}{\mathbb{Z}}

\newcommand*{\N}{\mathbb{N}}
\newcommand*{\R}{\mathbb{R}}
\newcommand*{\loc}{\mathrm{loc}}

\DeclareMathOperator{\Cut}{Cut}

\newcommand{\dif}{\mathop{}\!\mathrm{d}} 
\DeclareMathOperator*{\ucplim}{u.c.p.-lim}

\def\PZdefchar#1{
  \expandafter\def\csname frak#1\endcsname{\mathfrak{#1}}
  \expandafter\def\csname rm#1\endcsname{\mathrm{#1}}
  \expandafter\def\csname bb#1\endcsname{\mathbb{#1}}
  \expandafter\def\csname bf#1\endcsname{\mathbf{#1}}
  \expandafter\def\csname scr#1\endcsname{\mathcal{#1}}
  \expandafter\def\csname cal#1\endcsname{\mathcal{#1}}}
\def\PZdefloop#1{\ifx#1\PZdefloop\else\PZdefchar#1\expandafter\PZdefloop\fi}
\PZdefloop abcdefghijklmnopqrstuvwxyzABCDEFGHIJKLMNOPQRSTUVWXYZ\PZdefloop

\title[Rough semimartingales]{Rough semimartingales and $p$-variation estimates\\ for martingale transforms}
\author[P.~Friz]{Peter Friz}
\address[PF]{Institut für Mathematik\\ TU Berlin}
\address[PF]{Weierstraß--Institut für Angewandte Analysis und Stochastik}
\email{friz@math.tu-berlin.de}
\author[P.~Zorin-Kranich]{Pavel Zorin-Kranich}
\address[PZK]{Mathematical Institute\\ University of Bonn}
\email{pzorin@uni-bonn.de}

\makeatletter
\@namedef{subjclassname@2020}{%
  \textup{2020} Mathematics Subject Classification}
\makeatother

\begin{document}
\maketitle
\begin{abstract}
We establish a new scale of $p$-variation estimates for martingale paraproducts, martingale transforms, and It\^o integrals,
of relevance in rough paths theory, stochastic, and harmonic analysis.
As an application, we introduce rough semimartingales, a common generalization of classical semimartingales and (controlled) rough paths, and their integration theory.
\end{abstract}

\makeatletter
\providecommand\@dotsep{5}
\makeatother

\section{Statement of main results}

\subsection{Background}
Let $(\Omega, \calF, (\calF_t)_{t \geq 0}, \bfP)$ be a filtered probability space.
For a two-parameter process $\Pi=(\Pi_{t,t'})_{0 \leq t \leq t' < \infty}$ and $p \in (0,\infty)$, the \emph{$p$-variation} is defined by
\begin{equation}
\label{eq:def:Vp}
V^{p}\Pi :=
\sup_{l_{\max}, u_{0} \leq \dotsb \leq u_{l_{\max}}}
\Bigl( \sum_{l=1}^{l_{\max}} \abs{ \Pi_{u_{l-1},u_l} }^{p} \Bigr)^{1/p},
\end{equation}
with the $\ell^{p}$ norm replaced by the $\ell^{\infty}$ norm in the case $p=\infty$.
For a one-parameter process $f=(f_{t})_{t\geq 0}$, the $p$-variation is defined by
\[
V^{p}f := V^{p}(\delta f),
\quad
(\delta f)_{t,t'} := f_{t'} - f_{t}.
\]
The $p$-variation is a monotonically decreasing function of $p$.
A classical result about $p$-variation is \emph{L\'epingle's inequality} which tells that, for a \cadlag{} martingale $g=(g_{t})_{t\geq 0}$, we have
\begin{equation}
\label{eq:Lepingle}
\norm{V^{p}g}_{L^{q}(\Omega)} \lesssim \norm{V^{\infty}g}_{L^{q}(\Omega)},
\quad 2 < p \le \infty, \ 1 \le q < \infty .
\end{equation}
The notation $\lesssim$, along with some other conventions, is explained in Section~\ref{sec:notation}.
The estimate \eqref{eq:Lepingle} goes back to \cite{MR420837}.
The above version, which includes the endpoint case $q=1$, is more recent \cite[Remark 3.5]{MR4091110}, and is also the special case $F\equiv 1$, $p_{1}=\infty$ of Theorem~\ref{thm:main} below.

We note that $V^{\infty}g= \sup_{0 \le t < t'} \abs{\delta g_{t,t'}}$ is, essentially, the martingale maximal function of $(g_{t}-g_{0})_{t\geq 0}$.
For continuous martingales, the estimate \eqref{eq:Lepingle} holds for any $0 < q < \infty$, but this special case does not play a distinguished role in this article.
The estimate \eqref{eq:Lepingle} is false for $p=2$ already for the Brownian motion, see \cite[Theorem 1]{MR295434} for a precise lower bound in this case.

The notion of bounded $p$-variation is important in rough path theory, introduced in \cite{MR1654527}, which provides a pathwise meaning to some stochastic differential equations.
A systematic account of this theory for continuous paths can be found in \cite{MR2604669}, and a version for \cadlag{} paths can be found in \cite{MR3770049}.

In the range $p \in (2,3)$, which is the most interesting for martingales, rough path theory requires bounds on an area term as an input.
This area term is usually given by a stochastic integral, and it is our objective to prove suitable bounds for a wide class of integrands.
We approach this problem directly by keeping track of $p$-variation bounds in a construction of the It\^o integral.
We will now introduce the discrete approximations that we will use.

An \emph{adapted partition} $\pi$ is an increasing sequence of stopping times $(\pi_{n})_{n\in\N}$ such that $\pi_{0}=0$ and $\lim_{n\to\infty} \pi_{n} = \infty$.
For an adapted partition $\pi$, we write
\begin{equation}
\label{eq:def:floor}
\floor{t,\pi} := \max\Set{ s\in\pi \given s \leq t},
\quad 0 \leq t < \infty.
\end{equation}
For \cadlag{} adapted processes $F=(F_{s,t})_{0 \leq s \leq t}$, $g=(g_{t})_{t\geq 0}$ and an adapted partition $\pi$, we consider the following approximation to the It\^o integral:
\begin{equation}
\label{eq:hPi}
\hPi^{\pi}(F,g)_{t,t'}
:=
\sum_{\floor{t,\pi} \leq \pi_{j} < t'} F_{\floor{t,\pi},\pi_{j}} (g_{\pi_{j+1} \bmin t'} - g_{\pi_{j} \bmax t}),
\quad 0 \leq t \leq t' < \infty.
\end{equation}
The sum \eqref{eq:hPi} can be viewed as a Riemann--Stieltjes integral
\begin{equation}
\label{eq:hPi-integral}
\hPi^{\pi}(F,g)_{t,t'} =
\int_{(t,t']} F^{(\pi)}_{t,u-} \dif g_u,
\end{equation}
where $F^{(\pi)}$ is another adapted process, which is a discretized version of the process $F$, given by
\begin{equation}
\label{eq:F-discrete}
F^{(\pi)}_{s,t} := F_{\floor{s,\pi},\floor{t,\pi}}.
\end{equation}
An important special case arises when $F=\delta f$ are the increments of a one-parameter process $(f_{t})$, in which case we write
\[
\hPi^{\pi}(f,g) := \hPi^{\pi}(\delta f,g).
\]
Also, we have $(\delta f)^{(\pi)} = \delta(f^{(\pi)})$ with $f^{(\pi)}_{t} := f_{\floor{t,\pi}}$.

Another classical result about $p$-variation concerns the (deterministic, pointwise) existence of the Riemann--Stieltjes type integral
\begin{equation}
\label{eq:Young-integral}
\hPi(f,g) = \lim_{\pi} \hPi^{\pi}(f,g),
\end{equation}
called the \emph{Young integral}, provided $V^{p_1} f$, $V^p g$ are finite, $p_1>0,p>0$, and $1/p_1 + 1/p > 1$.
Although this result goes back to \cite[\textsection 10]{MR1555421}, the above version is only explicitly stated in \cite[Theorem 2.2]{MR3770049}.

If $g$ is a martingale, then $V^p g <\infty$ (locally in time) for any $2<p$ by L\'epingle's inequality~\eqref{eq:Lepingle}, and so Young's condition becomes $0 < p_1 < 2$.
Under this condition, for $1/r = 1/p_1 + 1/p$, we have
\begin{equation}
\label{eq:YoungMaxInequ}
V^r \hPi^{\pi}(f,g)
\lesssim (V^{p_1} f) (V^p g),
\end{equation}
and the same estimate holds for the limit $\hPi$ in \eqref{eq:Young-integral}.

\subsection{It\^o integral}
Our first main result extends the estimate \eqref{eq:YoungMaxInequ} to the case of It\^o integrals with integrands whose variation exponent is $p_{1} \geq 2$.
The pathwise estimate \eqref{eq:YoungMaxInequ} becomes false in this regime, and we have to substitute it with a moment estimate (which follows directly from \eqref{eq:YoungMaxInequ}, Hölder's, and L\'epingle's inequalities in the case $p_{1}<2$).
Moreover, we replace the increment process $\delta f$ by a general two-parameter process $F$; the motivation for doing so is explained below.
\begin{theorem}
\label{thm:main}
Let $0 < q_{1} \leq \infty$, $1\leq q_{0} < \infty$, $i_{\max}\in\N$, and $0 < r,p_{1},p_{i,0},p_{i,1} \leq \infty$ with $i\in\Set{1,\dotsc,i_{\max}}$.
Suppose
\begin{equation}
\label{eq:Vr-exponent-condition}
1/r < \min(1/p_{1} + 1/2, \min_{1\leq i \leq i_{\max}} 1/p_{i,1} + 1/p_{i,0}),
\quad
1/q = 1/q_{0} + 1/q_{1}.
\end{equation}
Let $(F_{s,t})_{s\leq t}$ be a \cadlag{} adapted process and $(g_{t})$ a \cadlag{} martingale.
Suppose that there exist \cadlag{} adapted processes $F^{i},\tilde{F}^{i}$, $i\in\Set{1,\dots,i_{\max}}$, such that
\begin{equation}
\label{eq:7}
F_{s,u}-F_{t,u} = \sum_{i=1}^{i_{\max}} F^{i}_{s,t} \tilde{F}^{i}_{t,u},
\quad
s \leq t \leq u.
\end{equation}
Then, the following holds.
\begin{enumerate}
\item\label{it:m1}
For every adapted partition $\pi$, we have the estimate
\begin{equation}
\label{eq:YoungBDG}
\begin{split}
\norm[\big]{ V^{r} \hPi^{\pi}(F,g) }_{L^{q}}
&\lesssim
\norm{ V^{p_1} F^{(\pi)} }_{L^{q_1}} \norm{ V^{\infty}g }_{L^{q_{0}}}
\\ &+
\sum_{i=1}^{i_{\max}} \norm{ V^{p_{i,1}} F^{i,(\pi)} \cdot V^{p_{i,0}} \Pi^{\pi}(\tilde{F}^{i},g) }_{L^{q}}.
\end{split}
\end{equation}
\item\label{it:m2}
For every $i \in \Set{1,\dotsc,i_{\max}}$, let $q_{i,0},q_{i,1} \in [q,\infty]$ with $1/q = 1/q_{i,0} + 1/q_{i,1}$, and suppose that
\begin{align}
\label{eq:hypothesis-Fi-discretization}
F^{i}
&= \lim_{\pi} F^{i,(\pi)}
&\text{in } L^{q_{i,1}}(V^{p_{i,1}}),
\\ \label{eq:hypothesis-Pi-Fi-discretization}
\Pi(\tilde{F}^{i},g)
&= \lim_{\pi} \Pi^{\pi}(\tilde{F}^{i},g)
&\text{exists in } L^{q_{i,0}}(V^{p_{i,0}}),
\end{align}
and $\tilde{F}^{i} \in L^{q_{1}}(V^{\infty})$.
Suppose that the right-hand side of \eqref{eq:YoungBDG-limit} is finite.
Then
\begin{equation}
\label{eq:paraprod-limit}
\hPi(F,g) := \lim_{\pi} \hPi^{\pi}(F,g)
\end{equation}
exists in $L^{q}(\Omega,V^{r})$, satisfies the bound
\begin{equation}
\label{eq:YoungBDG-limit}
\begin{split}
\norm[\big]{ V^{r} \hPi(F,g) }_{L^{q}}
&\lesssim
\norm{ V^{p_1} F }_{L^{q_1}} \norm{ V^{\infty}g }_{L^{q_{0}}}
\\ &+
\sum_{i=1}^{i_{\max}} \norm{ V^{p_{i,1}} F^{i} \cdot V^{p_{i,0}} \Pi(\tilde{F}^{i},g) }_{L^{q}},
\end{split}
\end{equation}
and, for any $0 \leq t \leq t' \leq t'' < \infty$, \emph{Chen's relation}
\begin{equation}
\label{eq:Chen}
\hPi(F,g)_{t,t''}
=
\hPi(F,g)_{t,t'} + \hPi(F,g)_{t',t''}
+ \sum_{i=1}^{i_{\max}} F^{i}_{t,t'} \hPi(\tilde{F}^{i},g)_{t',t''}.
\end{equation}
\end{enumerate}
\end{theorem}
The limit \eqref{eq:paraprod-limit} is the It\^o integral, which can also be denoted by
\begin{equation}
\label{eq:23}
\hPi(F,g)_{t,t'}
=
\int_{(t,t']} F_{t,u-} dg_u.
\end{equation}
The hypothesis~\eqref{eq:hypothesis-Fi-discretization} is easily verified if $F^{i}$ satisfies a structural hypothesis similar to \eqref{eq:7} for $F$, see Lemma~\ref{lem:F-pi-converges-to-F}.
The hypothesis~\eqref{eq:hypothesis-Pi-Fi-discretization} can typically be obtained by recursive application of Theorem~\ref{thm:main} with $\tilde{F}^{i}$ in place of $F$, if $\tilde{F}^{i}$ are in some sense of lower complexity than $F$.
Most prominently, if $F$ is some component of a rough path, then all $\tilde{F}^{i}$ can be taken to be lower level components of that path.

\subsubsection{Relation to previous works}
In the case $F\equiv 1$, we have $\Pi^{\pi}(F,g) = \delta g$ for any adapted partition $\pi$.
Moreover, the right-hand side of \eqref{eq:7} is an empty sum in this case, so that Theorem~\ref{thm:main} boils down to L\'epingle's inequality~\eqref{eq:Lepingle}.
Our argument has its roots in the approach to L\'epingle's inequality given in \cite{MR1019960,MR933985}; we also refer to \cite{MR4091110} for a short self-contained exposition of this case.

If $F = \delta f$ are the differences of a \cadlag{} process $f$, then
\[
F_{s,u} - F_{t,u} = (\delta f)_{s,t} \cdot 1 = F_{s,t} \cdot \tilde{F}_{t,u}
\]
with $\tilde{F}_{s,t}\equiv 1$.
The convergence hypotheses \eqref{eq:hypothesis-Fi-discretization} and \eqref{eq:hypothesis-Pi-Fi-discretization} are witnessed by the stopping construction in Lemma~\ref{lem:F-pi-converges-to-F}.
Since $\Pi(\tilde{F},g) = \delta g$ and by L\'epingle's inequality \eqref{eq:Lepingle} for $g$, the estimate \eqref{eq:YoungBDG-limit} becomes
\begin{equation}
\label{eq:main-est-delta-f}
\norm[\big]{ V^{r} \hPi(\delta f,g) }_{L^{q}}
\lesssim
\norm{ V^{p_1} (\delta f) }_{L^{q_1}} \norm{ V^{\infty}g }_{L^{q_{0}}}.
\end{equation}
In the special case $q_{1}=\infty$, the existence of the limit \eqref{eq:paraprod-limit} refines the $L^{q}$ convergence of discrete approximations to the It\^o integral \cite[Theorem 2.6]{MR606798}.

If $f$ is also a martingale, $1 \leq q_{1} < \infty$, and $r>1$, then, taking $p_{1}=2^{+}$ and using L\'epingle inequality \eqref{eq:Lepingle} for $f$, the estimate \eqref{eq:main-est-delta-f} implies
\begin{equation}
\label{eq:main-est-delta-f-martingale}
\norm[\big]{ V^{r} \hPi(\delta f,g) }_{L^{q}}
\lesssim
\norm{ V^{\infty} f }_{L^{q_1}} \norm{ V^{\infty}g }_{L^{q_{0}}}.
\end{equation}
In this case, the object $\Pi(\delta f,g)$ is analogous to so-called \emph{paraproducts} in harmonic analysis.
For paraproducts, an estimate of the form \eqref{eq:main-est-delta-f-martingale} was first proved in \cite{MR2949622}, motivated by an application of rough path theory in time-frequency analysis \cite[Corollary 1.2]{MR3596720}.

The estimate \eqref{eq:main-est-delta-f-martingale} is of interest because it shows that, for a (multidimensional) martingale $X$, the pair $(X,\Pi(X,X))$ is almost surely a rough path.
For continuous martingales, the estimate \eqref{eq:main-est-delta-f-martingale} was proved in \cite{MR2483743} (in the diagonal case $q_{0}=q_{1}$).
For \cadlag{} martingales, the estimate \eqref{eq:main-est-delta-f-martingale} was proved in \cite{MR3909973} (in the diagonal case $q_{0}=q_{1}$) and in \cite{MR4003122} (for general $q_{0},q_{1}>1$).

For non-martingale integrands $f$, the estimate \eqref{eq:main-est-delta-f} is new.
One of the motivations for considering this case is the construction of joint rough path lifts of rough paths and martingales, see Theorem~\ref{thm:dX} below, which underlies our notion of rough semimartingale.
Another motivation, see e.g.\ \cite{MR2164030} and \cite[Ch.14]{MR2604669}, is the analytic stability of It\^o integrals of the form $\int \varphi (f) \dif g$, with sufficiently regular $\varphi$, as a function of $f$.
A weaker version of the estimate \eqref{eq:main-est-delta-f}, which does not respect the Hölder scaling condition on $q$, was proved in the case $q_{0}=q_{1}=2$ in \cite[Proposition 3.13]{MR4255150} and used to establish invariance principles of random walks in random environments in rough path topology.

Although of no direct interest in rough paths, we note that the case $p_1 = \infty$, $r = 2^+$ of \eqref{eq:main-est-delta-f} is a consequence of L\'epingle's inequality applied to the martingales $(\int_{0}^{t} f_{u^-} \dif g_u)_{t}$ and $g$.
However, the approach via Theorem~\ref{thm:main} is still preferable in this case, since it provides a \emph{construction} of the It\^o integral $\int f_{u^-} \dif g_u$ that naturally comes with variation norm estimates.
We further elaborate on this point of view in Section~\ref{sec:Ito-mesh-convergence}, where we deduce the classical convergence results for discrete approximations to the It\^o integral with respect to \cadlag{} local martingales ($\calM_{\loc}$) from Theorem~\ref{thm:main}.
At this point, the ability to take $q_{0}=1$, missing in \cite{MR4003122}, is important, see Lemma~\ref{lem:L1-localizing-sequence}.

The estimate \eqref{eq:YoungBDG-limit} for processes $F$ that are not of the increment form is useful for the construction of It\^o \emph{branched rough paths}, see Section~\ref{sec:var-discrete-branched}.
For instance, if $f \in L^{q_{1}}(V^{p_{1}})$ with $p_{1}\geq 4$, then the information $\int\delta f^- \dif g$ is not sufficient for rough path theory, and more stochastic building blocks have to be included.
Theorem~\ref{thm:main} shows, for instance, that $ \int (\delta f^-)^2 \dif g $ has variational exponent $r = 1/(2/p_1 + 1/2)^-$.
Note that one can choose $r < 1$ iff $p_1 < 4$ which, in that case, reflects redundancy of $ \int (\delta f^-)^2 \dif g $ from a rough integration perspective.
In harmonic analysis, analogues of such integrals are known as \emph{multilinear paraproducts}, see e.g.\ \cite{MR1945289,MR3255002}.

Another setting in which two-parameter integrands $F$ are useful is that of controlled rough integration, introduced in \cite{MR2091358}.
The easiest situation is as follows.
Let $X,Y,Y'$ be \cadlag{} adapted processes and $g$ a \cadlag{} martingale.
We interpret $Y'$ as the Gubinelli derivative of $Y$ with respect to $X$, so that the remainder term is given by
\begin{equation}
\label{eq:Gubinelli-R}
R \equiv \delta Y - Y' \delta X :\iff
R_{s,t} \equiv \delta Y_{s,t} - Y'_{s} \delta X_{s,t}.
\end{equation}
Then
\begin{equation}
\label{eq:delta-R}
R_{s,u} - R_{t,u} = \delta Y'_{s,t} \delta X_{t,u} + R_{s,t} \cdot 1,
\end{equation}
and Theorem~\ref{thm:main} implies the estimate
\begin{equation*}
\label{eq:15}
\norm[\big]{ V^{r} \Pi(R,g) }_{q}
\lesssim
\norm[\big]{ V^{r_{2}} Y' \cdot V^{1/(1/r_{1}+1/2)} \Pi(\delta X,g) }_{q}
+
\norm[\big]{ V^{1/(1/r_{1}+1/r_{2})} R }_{q_{1}}
\norm{ V^{\infty}g }_{q_{0}}.
\end{equation*}
When the $\ell^{r}$ norm implicit in the left-hand side of this estimate is computed for a given partition $\pi$, this estimate can be interpreted as a bound for the error in a discrete approximation of the controlled integral $\int Y dg$.

Such integrands also appear in stochastic numerics, see e.g.\ \cite[Ch.5]{MR1214374}, \cite[Lem.4.2.]{MR1470933}, or \cite{MR2320830}.

\subsubsection{Further variants}
Theorem~\ref{thm:main} continues to hold with all processes being Hilbert spaces valued, upon replacing all products by tensor products, and the bounds do not depend on the dimensions of the Hilbert spaces.

The limiting variational estimate \eqref{eq:YoungBDG-limit} has a precise analogue in H\"older topology, given in Appendix~\ref{sec:Happ}, which extends and quantifies some previous constructions notably Diehl et al. \cite{MR3274695} and \cite[Ch.13]{FH2020} (with $g$ taken as Brownian motion). To wit, in these references the H\"older regularity is obtained by some variation of Kolmogorov's criterion (or Besov-H\"older embedding); the resulting
$(1/q)^+$-loss on the H\"older exponent (integrability parameter $q$) is avoided in Theorem~\ref{thm:H-Ito}.

\subsection{Rough integrators}
The second main result concerns integrals formally given by
\[
\cPi(g,\rmY)_{t,t'} \equiv \int_{(t,t']} (\delta g)_{t,u-} dY_u,
\]
where $g$ is a martingale and $\rmY$ is a \emph{suitable} (rough) \cadlag{} process.
When $V^{p_{1}}Y \in L^{q_{1}}(\Omega)$ for some $p_1 < 2$, using Young's inequality pathwise, with $p_0>2$ such that $1/p_0 + 1/p_1 > \max(1,1/r)$, followed by H\"older's inequality (with $q,q_{0},q_{1}$ as in Theorem~\ref{thm:main}) and L\'epingle's estimate (applied to $\norm{V^{p_0}g}_{L^{q_0}}$), we see
\begin{equation}
\label{equ:RIsimple}
\norm[\big]{ V^{r} \cPi(g,\rmY) }_{L^{q}(\Omega)}
\lesssim
\norm{ V^{p_1}Y }_{L^{q_{1}}(\Omega)}
\norm{V^{\infty}g}_{L^{q_{0}}(\Omega)}.
\end{equation}
When $p_1 \ge 2$, pathwise arguments fail.
Instead, we will define $\Pi(g,\rmY)$ using integration by parts.
We start with the summation by parts formula for the discretized paraproduct \eqref{eq:hPi} associated to an adapted partition $\pi$ of $[0,T]$:
\[
(Y_T - Y_0)(g_T - g_0) - \Pi^{\pi}(Y,g)_{0,T}
=
\Pi^{\pi}(g,Y)_{0,T} + \sum_{\pi_j < T } (Y_{T \bmin \pi_{j+1}} - Y_{\pi_{j}})(g_{T \bmin \pi_{j+1}} - g_{\pi_{j}}).
\]
Under the assumptions of Theorem~\ref{thm:main}, we can pass to the limit along $\pi$ on the left-hand side, and hence also on the right-hand side.
We would like to interpret the limits of the two summands on the right-hand side as as $\int g^- \dif \rmY = \cPi(g,\rmY) = \lim_{\pi} \cPi^{\pi}(g,\rmY)$ and a \emph{covariation bracket} $[Y,g] = \lim_{\pi} [Y,g]^\pi$, respectively.
However, these summands do not in general individually converge along $\pi$.
We give an example in which these two limits do not exist.
\begin{example}
Let $g = B$ be a standard Brownian motion and $Y_t = B^{H}_{t} := \int_0^t (t-s)^{H-1/2} \dif B$ a fractional Brownian motion (fBm) of Hurst parameter $H \in (0,1/2)$.
Then $Y_{t}$ has locally bounded $p_1$-variation for any $p_1 > 1/H$ (and no better).
For $T=1$ and a partition $\pi$ including $T$, by It\^o isometry, we have
\[
\E \sum_{\pi_{j} < 1} (B^H_{\pi_{j+1}-} - B^H_{\pi_j}) (B_{\pi_{j+1}}-B_{\pi_{j}})
\gtrsim
\E \sum_{j} \abs{\pi_{j}-\pi_{j-1}}^{H+1/2},
\]
which is divergent in the rough regime $H < 1/2$.
(In other words, the It\^o integral $\int B^H \dif B$ has infinite It\^o--Stratonovich correction; see \cite[Chapters 14 and 15]{FH2020} for a discussion of this example from a general renormalization perspective.)
As a consequence, $\lim_{\pi} \Pi^{\pi}(B,B^H)$ does not exist.
\end{example}
The problem in this example is \emph{correlation}.
One way of ruling out such situations is to take $Y= X$ deterministic (or independent of $g$), which is why independence of components is a common assumption for Gaussian rough paths \cite{MR2667703}.%
\footnote{For an independent Brownian $B^\perp$, existence of $\int B^\perp \dif B^H = \lim \Pi(B^\perp,B^H)^\pi$ holds in $L^2(\Omega)$.}
We use a more flexible structural assumption to overcome this problem, namely, we assume that the (adapted) process $Y$ is \emph{controlled} by a deterministic reference path $X$, in a sense based on \cite{MR2091358}.

\begin{theorem}
\label{thm:dX}
Let $q,q_{0},q_{1}$ be as in Theorem~\ref{thm:main}, $0 < r \leq \infty$, and $0 < \hat{p}_1 < 2 \leq p_1 \leq \infty$ with $1/r < 1/2+1/p_{1}$.
Let $X$ be a \emph{deterministic} \cadlag{} path, $\rmY=(Y,Y')$ a \cadlag{} adapted process, and $g$ a \cadlag{} martingale.
Assume that
\[
V^{\infty} g \in L^{q_{0}},
\quad
\MYp := \sup_{t} \abs{Y'_{t}} \in L^{q_{1}},
\quad
X \in V^{p_{1}},
\quad
V^{\hat{p}_1} R^{\rmY} \in L^{q_{1}},
\]
where
\begin{equation}
\label{eq:Gubinelli-remainder}
R^{\rmY}_{s,t} := R^{\rmY,X}_{s,t} :=
Y_{t}-Y_{s}-Y_{s}'(X_{t}-X_{s}),
\quad 0 \leq s \leq t < \infty.
\end{equation}
Then, there exists a process $(\cPi(g,\rmY)_{t,t'})_{0 \leq t \leq t' < \infty}$ with the following properties.
\begin{enumerate}
\item It is a u.c.p.\ limit along deterministic partitions of discretized paraproducts:
\begin{equation}
\label{eq:9}
\cPi(g,\rmY)_{0,T} = \ucplim_{\dmesh(\pi) \to 0} \hPi^{\pi}(g,Y)_{0,T} =: \int_0^T (\delta g)_{0,t-} \dif \rmY_t.
\end{equation}
\item We have Chen's relation
\begin{equation}
\label{eq:Chen-g-dY}
\cPi(g,\rmY)_{t,t''}
=
\cPi(g,\rmY)_{t,t'} + \cPi(g,\rmY)_{t',t''} + (g_{t'}-g_{t})(Y_{t''}-Y_{t'}).
\end{equation}
\item We have the bound
\begin{equation}
\label{eq:controlled-integrator-BDG-limit}
\norm[\big]{ V^{r} \cPi(g,\rmY) }_{L^{q}(\Omega)}
\lesssim
\Bigl( V^{p_{1}}X \norm{ \MYp }_{L^{q_{1}}(\Omega)} + \norm{ V^{\hat{p}_1}R^{\rmY} }_{L^{q_{1}}(\Omega)} \Bigr) \norm{V^{\infty}g}_{L^{q_{0}}(\Omega)}.
\end{equation}
\end{enumerate}
\end{theorem}
Theorem~\ref{thm:dX} is proved in Section~\ref{sec:int-by-parts}.
The construction of $\cPi(g,\rmY)$ is based on the aforementioned integration by parts identity in combination with constructing
quadratic covariation, given as (u.c.p.) limit of $[Y, g]^\pi$ (see Definition~\ref{def:discrete-bracket-2}), for every local martingale $g$, identified explicitly in Theorem~\ref{thm:[Y,g]-discrete-approx} as
\begin{equation} \label{equ:Ch1Cov}
\sum_{s \leq t} \Delta X_{s} Y'_{s-} \Delta g_{s} + \sum_{s \leq t} \Delta R^{\rmY}_{s} \Delta g_{s}
=: [\rmY,g]_{t}.
\end{equation}
Note that $[\rmY,g]$ implicitly depends on $X$.
Moreover, in general, $[Y, Y]^\pi$ does not converge.
Again, several remarks are in order.
\begin{itemize}
\item The exponent $p_1$ quantifies the variational regularity of both $X$ and $Y$.
The assumption $p_1 \ge 2$ is not essential. Indeed, as noted above, when $p_1 < 2$ one can use (pathwise) Young, H\"older, and L\'epingle to get the estimate \eqref{equ:RIsimple}, from which \eqref{eq:controlled-integrator-BDG-limit}, if so desired, is an easy consequence.

\item The assumption $\hat{p}_1 < 2$ reflects the ``length'' of the expansion $Y_t \approx Y_s + Y'_s (X_t - X_s)$, familiar from controlled rough path theory (think: $\hat{p}_1 = p_{1}/2$) although we do not need to control any variation norm of $Y'$ here: Theorem~\ref{thm:dX} is a stochastic result, and not based on pathwise (sewing) arguments.
It is then clear that the condition on $\hat{p}_1$ could be relaxed by suitable higher order ``controllness'' assumptions, but we have not pursued this further.
\item The special case of deterministic $Y=X$ corresponds to $(Y,Y') = (X,1), R^{\rmY} = 0$. Take $q_1 = \infty$ and $1\leq q_{0} = q < \infty$, so that \eqref{eq:controlled-integrator-BDG-limit} simplifies to
\begin{equation}
\norm[\big]{ V^{r} \cPi(g,X) }_{L^{q_0}(\Omega)}
\lesssim
( V^{p_{1}}X) \norm{V^{\infty}g}_{L^{q_{0}}(\Omega)}.
\end{equation}
In case of random $X$, but independent of $g$, this estimate can be used upon conditioning on $X$, and immediately gives
\[
\norm[\big]{ V^{r} \cPi(g,X) }_{L^{q_0}(\Omega)}
\lesssim
\norm{V^{p_{1}}X}_{L^{q_{0}}(\Omega)} %
\norm{V^{\infty}g}_{L^{q_{0}}(\Omega)}.
\]
The better integrability of the left-hand side, compared to \eqref{eq:controlled-integrator-BDG-limit}, is a consequence of independence.
\item U.c.p.\ convergence as $\mesh(\pi)\to 0$ (with non-deterministic partitions $\pi$) in \eqref{eq:9} fails in general for the two-parameter processes $\cPi^{\pi}(g,\rmY)_{t,t'}$.
In fact, it already fails in the simpler situation of Corollary~\ref{cor:Ito-mesh-convergence}, which deals with mesh convergence of discrete approximations to It\^o integrals.
\end{itemize}

\subsection{Rough semimartingales}

Recall that a classical {\em semimartingale} $Z=g+Y$, possibly vector valued, is the sum of a \cadlag{} local martingale $g$ and \cadlag{} adapted $Y \in V^1_\loc$. This was generalized, at least in the continuous setting,
to {\em Dirichlet processes} \cite{MR621001}, where the finite variation condition on $Y$ is replaced by vanishing quadratic variation.
In a similar spirit, we can define {\em Young semimartingales} (YSM) as processes $Z=g+Y$, as above, but now with $Y \in V^{2-}_\loc$, meaning $V^{p}_\loc$ for $p \in [1,2)$. Although this decomposition need not be unique, for any two Young semimartingales $Z,\bar{Z}$,
the paraproduct $\Pi (Z, \bar{Z})_{t,t'} = \int (\delta Z)_{t,u-} d \bar Z_u$ is easily seen to be well-defined, essentially as consequence of It\^o and Young integration, with pathwise estimates obtained by combining Young and L\'epingle, exactly as was done for \eqref{equ:RIsimple}. Examples of suitable $V^{2-}_\loc$ processes include fractional Brownian motion with Hurst parameter $H>1/2$ and $\alpha$-stable L\'evy processes, $\alpha < 2$, see \cite{JM83,MR2073185} for some general results.

Both Dirichlet processes and Young semimartingales face a seemingly fundamental barrier at $p=2$. Yet,
Theorems~\ref{thm:main} and~\ref{thm:dX} provide us with a way of going beyond - the key idea is to postulate a deterministic reference path $X$.
(This assumption appears naturally, e.g.\ under partial conditioning of driving noise, cf.\ Corollary~\ref{cor:final}.)

\begin{definition} \label{def:controlled}
Let $p\in [2,3)$.
Let $X$ be a \cadlag{} adapted
process, with values in some Hilbert space $\tilde H$ and $X \in V^{p}_{\loc}$ almost surely.
We call a pair of \cadlag{} adapted processes $\rmY=(Y,Y')$ with values in some Hilbert space $H$ and in the operator space $L(\tilde H,H)$, respectively, an \emph{$X$-controlled $p$-rough process} if $Y,Y' \in V^{p}_{\loc}$ and $R^{\rmY,X} \in V^{p/2}_{\loc}$, almost surely.
\end{definition}
Recall that $R^{\rmY,X}$ was defined in \eqref{eq:Gubinelli-remainder}.
\begin{definition} \label{def:RSM}
Let $p\in [2,3)$ and $X\in V^{p}_{\loc}$ be a \cadlag{} \emph{deterministic} path.
We define an \emph{$X$-controlled $p$-rough semimartingale (RSM)} to be a \cadlag{} adapted process of the form
\[
(g+Y,Y') : \Omega \times [0,\infty) \to H \oplus L(\tilde H,H),
\]
where $g$ is a \cadlag{} local martingale and $\rmY=(Y,Y')$ is an $X$-controlled $p$-rough \cadlag{} adapted process.
\end{definition}

A trivial example of $X$-controlled $p$-RSM is given by $(g + X, \mathrm{Id})$ for some deterministic \cadlag{} path $X \in V^{p}_{\loc}, p < 3$, as may be supplied by a typical realization of another martingale.
The following can be seen as RSM version of the Doob--Meyer decomposition for special semimartingales.
\begin{theorem}
Let $(g_{i}+Y_{i},Y_{i}')$ be $X_{i}$-controlled RSMs, $i=1,2$, with $g_{1}+Y_{1}=g_{2}+Y_{2}$.
Assume $Y_{i}(\omega,t)$ is previsible for $i=1,2$ and $Y_{1}(\omega,0) = Y_{2}(\omega,0)$.
Then $Y_{1}=Y_{2}$.
\end{theorem}
\begin{proof} From \eqref{equ:Ch1Cov}, using crucially the existence of the reference path $X_{i}$, the quadratic covariation
\[
[Y_{i}, \bar{g}] = \ucplim_{\dmesh(\pi)\to 0} [Y_{i}, \bar{g}]^\pi
\]
exists \emph{and vanishes} for every \emph{continuous} local martingale $\bar{g}$.
(This shows that $g_{i}+Y_{i}$ is a \emph{weak Dirichlet process} in the sense of \cite{MR1961622, MR2276891}).
The difference $Y_1 - Y_2 =: \bar{g}$ is a previsible local martingale, hence a continuous local martingale.
But then
\[
\ucplim_{\dmesh(\pi)\to 0} [Y_1 - Y_2,Y_1 - Y_2]^\pi
=
\ucplim_{\dmesh(\pi)\to 0} [Y_1,\bar{g}]^\pi - [Y_2,\bar{g}]^\pi
=
0.
\]
This shows that $Y_1 - Y_2$ is a continuous martingale with vanishing quadratic variation (cf.\ \eqref{eq:martingale-bracket-discretized}), starting at zero, hence identically equal to zero.
\end{proof}

Similar to controlled rough paths, the notion of RSM is most fruitful when paired with {\em rough paths}.
Recall \cite{MR1654527, MR3693973}, see also \cite{MR1891200} and \cite{MR3968507} for a recent review (with applications to homogenization), that a \cadlag{} $p$-rough path with $p\in (2,3)$ can be viewed as a pair of \cadlag{} processes $\bfX=(X,\bbX)=((X_t),(\bbX_{s,t}))$ with values in a Banach space $B$ and a tensor product space $B \otimes B$, with $V^p X, V^{p/2} \bbX$ (locally in time) finite and subject to Chen relation $\bbX_{t,t''}= \bbX_{t,t'}+ \bbX_{t',t''}+ (\delta X)_{t,t'}(\delta X)_{t',t''}$.
Recall further that \cadlag{} $X$-controlled $p$-rough paths can be integrated against $\bfX$ and, more generally, other \cadlag{} $X$-controlled $p$-rough paths,
\begin{equation}
\label{eq:int-controlled-controlled-discrete-sums}
\begin{split}
\int_{(0,T]} \delta\rmY \dif \bar \rmY
&=
\lim_{\mesh(\pi)\to 0} \Pi^{\pi}(\rmY, \bar \rmY)_{0,T},\\
\Pi^{\pi}(\rmY, \bar \rmY)_{T,T'}
&=
\sum_{\pi_{j} \leq T} \delta Y_{0,\pi_{j}} \delta \bar Y_{\pi_{j},\pi_{j+1} \bmin T} + Y'_{\pi_{j}} \bar Y'_{\pi_{j}} \bbX_{\pi_{j},\pi_{j+1} \bmin T}.
\end{split}
\end{equation}
The statement with mesh convergence above is from \cite[Proposition 2.6]{MR3770049}; the proof in fact also shows that the convergence is locally uniform in $T$.
Convergence of \cadlag{} rough integrals in the net sense was proved in \cite[Theorem 34]{MR3693973} (with $\bar{Y}=X$, $\bar{Y}'=1$; see \cite[Remark 4.12]{FH2020} for the general case), extending the Hölder continuous case in \cite{MR2091358}.

\begin{theorem} \label{thm:RSMlift}
Let $p \in [2,3)$, $\bfX=(X,\bbX)$ be a \cadlag{} $p$-rough path.
For any two rough semimartingales
$\mathrm{W} = (g + Y, Y'), \mathrm{\bar W} = (\bar g + \bar Y, \bar Y')$,
the following holds.
\begin{enumerate}
\item \label{thm:RSMlift:Pi-well-def}
The paraproduct
\begin{equation}
\label{eq:Pi-RSM-RSM}
\Pi (\mathrm{W},\mathrm{\bar W})_{t,t'} :=
\int_{(t,t']} \delta (g+Y)_{t,u-} \dif \bar g_u + \int_{(t,t']} (\delta g)_{t,u-} \dif \bar \rmY_u
+ \int_{(t,t']} (\delta \rmY)_{t,u-} \dif \bar \rmY_u
\end{equation}
is well-defined, in the sense that it does not depend on the decomposition of $\mathrm{W}$.
The summands on the right hand side of \eqref{eq:Pi-RSM-RSM} are defined as follows: the first one is an It\^o integral, the second is $\int \delta g \dif \bar \rmY := \Pi (g, \bar \rmY)$, and the third is a rough integral.
\item \label{thm:RSMlift:Pi-is-RSM}
The enhanced paraproduct
\[
(\Pi (\mathrm{W},\mathrm{\bar W})_{0,t}, \delta( g+ Y)_{0,t} \bar{Y}'_t)
\]
defines another rough semimartingale, with local martingale component given by the It\^o integral $\int_{(0,t]} \delta( g+ Y)_{0,u-} \dif \bar g_u$.
\item \label{thm:RSMlift:lift}
The following process is almost surely a \cadlag{} $p$-rough path:
\[
\bfW := (g+Y, \Pi (\mathrm{W},\mathrm{W})),
\]
where $\Pi(\rmW,\rmW)$ is given by Part~\ref{thm:RSMlift:Pi-well-def} with $\bar\rmW=\rmW$.
\end{enumerate}
\end{theorem}
Theorem~\ref{thm:RSMlift} will be proved in Section~\ref{sec:discretization-covariation}.
Quantitative estimates for the terms on the right-hand side of \eqref{eq:Pi-RSM-RSM} are provided by Theorem~\ref{thm:main}, Theorem~\ref{thm:dX}, and (\cadlag{}) rough integration theory, respectively.

The extra structure (in form of $Y'$) of RSM is crucial for validity of Theorem~\ref{thm:RSMlift}, for the simple reason that there is no (sensible) construction of $\int Y^- \dif \bar Y$ for $Y,\bar Y$ of finite $p$-variation, $p \ge 2$, even in
case of vanishing $2$-variation paths. (This remark also shows that there does not exist a paraproduct for general Dirichlet processes, leave alone weak Dirichlet processes).

With notation as in Definition~\ref{def:RSM}, a pair $(Y,g) \in H_1 \oplus H_2 =: H$ becomes a RSM upon writing
\[
\begin{pmatrix}
Y \\
g
\end{pmatrix}
=
\left(
\begin{pmatrix}
0 \\
g
\end{pmatrix}
+
\begin{pmatrix}
Y \\
0
\end{pmatrix}
,
\begin{pmatrix}
Y' \\
0
\end{pmatrix}
\right),
\]
write $\calJ(\rmY, g)$ for the resulting $p$-rough path.
A simple yet important special case is $(Y,Y')=(X,\mathrm{Id})$.
As a special case, a pair $(\bfX,g)$ then automatically gives rise to a $p$-rough path $\calJ
(\bfX, g)$, as (It\^o) joint rough path lift of $(g,X)$.
See also Theorem~\ref{thm:controlled-integral=RSM-integral} for a consistency results between pathwise rough and rough semimartingale integration.
We spell out some estimates for the (It\^o) joint rough path, straight forward consequences of Theorems~\ref{thm:main} and~\ref{thm:dX}.

\begin{theorem} \label{thm:ItoLyonsLift}
Let $\bfX=(X,\bbX)$ be a \cadlag{} $p$-rough path over $\R^m$, $p \in (2,3)$, and $g$ an $\R^n$-valued martingale with $V^{\infty} g \in L^{q_{0}}$, for some $1\leq q_{0} < \infty$.
Then, a.s., the map
\begin{equation}
\label{eq:joint-rough-lift}
\calJ:
(\bfX, g (\omega)) \mapsto
\left(
\begin{pmatrix}
X \\
g
\end{pmatrix},
\begin{pmatrix}
\bbX & \cPi(g,X) \\
\hPi(X,g) & \hPi(g,g)
\end{pmatrix}
\right) = (X^g (\omega), \bbX^g (\omega)).
\end{equation}
takes values in the space of \cadlag{} $p$-rough paths over $\R^{m+n}$, with $q_0$-integrable homogeneous rough path norm, given by
\[
V^p_{\mathrm{hom}} \bfX^g := V^p X^g + (V^{p/2} \bbX^g)^{1/2} \in L^{q_0}.
\]
Moreover, $\calJ$ is locally Lipschitz continuous in the sense that
\[
\norm[\big]{ V^p(X_1^{g_1} - X_2^{g_2}) }_{L^{q_0}}
\lesssim
V^p(X_1 - X_2) + \norm{ V^\infty(g_1 - g_2) }_{L^{q_0}},
\]
and
\begin{align*}
& \norm{ V^{p/2} (\hPi(X_1,g_1) - \hPi(X_2,g_2)) }_{L^{q_0}} + \norm{ V^{p/2} (\cPi(g_1,X_1) - \cPi(g_2,X_2)) }_{L^{q_0}} \\
& \qquad \qquad \qquad \lesssim (V^pX_1) \norm{ V^\infty(g_1 - g_2) }_{L^{q_0}} + V^p(X_1 - X_2) \norm{ V^\infty g_2 }_{L^{q_0}} , \\
& \norm{ V^{p/2} (\hPi(g_1,g_1) - \hPi(g_2,g_2)) }_{L^{q_0/2}} \\
& \qquad \qquad \qquad \lesssim ( \norm{ V^\infty g_1 }_{L^{q_0}}+ \norm{ V^\infty g_2 }_{L^{q_0}}) \norm{ V^\infty(g_1 - g_2) }_{L^{q_0}}
\end{align*}
In particular, the map $(\bfX,g) \mapsto \calJ(\bfX,g)=: \bar \bfX = (\bar X, \bar \bbX)$ is continuous (and uniformly so on bounded sets), with respect to homogeneous $L^{q_0}$ rough paths metric
\[
\norm{ V^p_{\mathrm{hom}} (\bar \bfX_1 - \bar \bfX_2) }_{L^{q_0}}
\sim
\norm{ V^p (\bar X_1 - \bar X_2)}_{L^{q_0}} + \norm{(V^{p/2} (\bar \bbX_1 - \bar \bbX_2))^{1/2}}_{L^{q_0}} .
\]
\end{theorem}

\subsection{Differential equations}
In Theorem~\ref{thm:RSMlift}, we gave a canonical construction of a (random) $p$-rough path $\bfW$ associated to any rough semimartingale $\mathrm{W} = (g + Y, Y')$
in sense of Definition~\ref{def:RSM}. The parameter $p \in (2,3)$ and the reference path $X$ are kept fixed.
In particular, rough semimartingales can drive differential equations,
\begin{equation}
dZ = \sigma (Z^-) \dif \mathrm{W}
:\iff
dZ = \sigma (Z^-) \dif \bfW \label{equ:RDEW},
\end{equation}
understood for a.e. realization of $\bfW =\bfW (\omega)$ as rough differential equation (by nature, multidimensional). This should be contrasted with
SDEs driven by weak Dirichlet processes \cite{MR2303950}, essentially restricted to scalar drivers.%
\footnote{This restriction is easy to understand since every deterministic continuous path is a weak Dirichlet process. In general, this is not sufficient to drive a differential equation in a unique way, which is the raison d'\^etre of rough path theory.}
Results from (deterministic) rough path theory \cite[Theorem 3.8]{MR3770049} provide a unique solution $Z=Z(\bfW,Z_0)$ of the initial value problem for \eqref{equ:RDEW} provided that $\sigma \in \mathrm{Lip}^{3}$, although a look at the proof reveals that $\sigma \in \mathrm{Lip}^{p+}$ is sufficient, as is the classical case for continuous rough paths, see e.g.\ \cite{MR1654527, Davie, MR2091358,MR2604669}.
The construction assures that $Z_t =Z_t(\bfW(\omega),Z_0(\omega))$ defines an adapted (\cadlag{}) process provided that the initial datum $Z_{0}$ is $\calF_{0}$-measurable.
When $(Y,Y') = 0$, $\bfW$ is nothing but the
It\^o rough path lift of the \cadlag{} local martingale $g$, as previously constructed in \cite{MR3909973}, and yields (a robust version of) the classical
It\^o solution, as found in textbooks, such as e.g. \cite{MR2273672}, on stochastic differential equations. 
It  convenes to replace $\sigma$ by $(\sigma,\mu)$ and consider
\begin{equation} \label{equ:Hybrid}
dZ = \sigma(Z^-) \dif \bfX + \mu(Z^-) \dif g
:\iff
dZ = (\sigma, \mu) (Z^-) \dif \calJ (\bfX, g).
\end{equation}
Several authors have studied, later used, such ``mixed'' differential equations, often with $g=B$, a multidimensional Brownian motion, and $\bfX$ replaced by an independent fractional Brownian $B^H$ motion with $H>1/2$; see  \cite{MR2440915}, \cite{MR4124526} and references therein.  
In this case, the left-hand side of \eqref{equ:Hybrid} makes sense in mixed Young It\^o sense (and could accordingly be phrased in terms of Young semimartingales).
From the perspective of \cite{MR2667703}, it suffices to construct $(B^H,B)$ jointly as Gaussian rough paths, which is possible for $H>1/4$.
Equation \eqref{equ:Hybrid}, in
case when $g$ is a Brownian motion $B$ and $\bfX$ a geometric H\"older rough path, was treated in \cite{MR3134732} as flow transformed It\^o SDE, in
\cite{MR3274695,MR3746646}, in the right-hand side sense of \eqref{equ:Hybrid}.
(In absence of jumps, the situation is much simplified in
that $(\bfX, B)$ is constructed by a Kolmogorov type criterion for rough paths; see \cite[Ch.12]{FH2020} for a review.)
Last but not least, we mention the work \cite{arxiv:2106.10340} that takes a different perspective on the problem of mixed differential equations, 
$$
dZ = \sigma(Z) \dif \bfX + \mu(Z) \dif B,
$$
but with Brownian noise $B$. The conceptual main point in this work is the introduction of stochastic controlled rough paths inspired by Khoa L\^e's stochastic sewing \cite{MR4089788}. This allows for a direct strong solution theory, under a mere Lipschitz condition on $\mu$. 
In turn, the stochastic sewing lemma is somewhat rigidly tied to H\"older regularity (to wit, the Young argument of $p$-variation sewing amounts
to cherry-pick the right to-be-removed point of a partition, a procedure not compatible with the adaptedness structure essential to stochastic sewing).
In particular, such arguments are currently ill-suited%
\footnote{Should there by a major advance in $p$-variation stochastic sewing, it would be very fruitful to combine it with the ideas and estimates of this work, but at this stage this is pure speculation.}
to understand the case of general \cadlag{} $g \in \calM_\loc$ in \eqref{equ:Hybrid}.

A remark on the subtlety of \eqref{equ:Hybrid} is in order: the formal expression on the left 
suggests that $Z$ is a rough semimartingale with local martingale component given by the (well-defined) It\^o integral $\int \mu (Z^-) \dif g$.
However, from a rough path perspective, $Z$ is constructed as an $(X,g)$-controlled rough path.
Knowing only $\bfX$, this is insufficient to define $\int \sigma(Z^-) \dif \bfX$
by (pathwise) rough integration.

The next theorem, which is proved in Section~\ref{sec:consistency}, shows that the left-hand side of \eqref{equ:Hybrid} has, thanks to stochastic cancellations, a bona-fide integral meaning after all.

\begin{theorem}
\label{thm:RDE-sol=>RSDE-sol}
Let $\sigma, \mu \in \mathrm{Lip}^{p+}$, so that \eqref{equ:Hybrid} admits a unique solution process in RDE sense, given by
\begin{equation} \label{equ:ZZZ}
Z_t (\omega) := Z_t(\bfX, Z_0(\omega); \omega) := Z_t(\calJ (\bfX, g)(\omega),Z_0(\omega)),
\end{equation}
adapted for $\calF_0$-measurable $Z_0$.
Then $(Z, \sigma(Z))$ is a rough semimartingale with decomposition $Z = M + Y$ with local martingale component $M = \int_0^\cdot \mu(Z^-) \dif g$ and
$Y$ given by
\[
Y_t
= \ucplim_{\dmesh(\pi) \to 0} \sum_{j : \pi_{j} < t} \Bigl( \sigma (Z_{\pi_{j}}) X_{\pi_{j},\pi_{j+1}\bmin t} + ((D\sigma) \sigma) (Z_{\pi_{j}})\bbX_{\pi_{j},\pi_{j+1}\bmin t} \Bigr)
= \int_0^t \sigma (Z^-_s) \dif \bfX_s .
\]
\end{theorem}

The next result asserts, loosely speaking, that an It\^o SDE solution, conditioned on (an independent) part of the driving noise, is a.s. a rough semimartingale.
(This can be seen as major extension of the rather trivial fact $B(\omega) + X$ is a rough semimartingale (in $\omega$) for a.e. typical realization of $X = B^\perp (\omega')$, for independent Brownian motions $B,B^\perp$.)

\begin{corollary} \label{cor:final} Assume $g = g(\omega)$ and $X = X(\omega')$ are independent local martingales, defined on some filtered product space $(\bar \Omega, \bar\calF) = (\Omega,\calF) \times (\Omega',\calF')$.
Let $\sigma,\mu$ be as in Theorem~\ref{thm:RDE-sol=>RSDE-sol}
and write
$\tilde Z(Z_{0};\omega,\omega')$ for the unique $\bar\calF$-adapted solution of the It\^o SDE
\begin{equation}
\label{eq:Ito-eq}
d\tilde Z = \sigma(\tilde Z^-) \dif X + \mu(\tilde Z^-) \dif g
\end{equation}
with $\bar\calF_{0}$-measurable initial data $Z_{0}=Z_0 (\omega, \omega')$. With the It\^o rough path lift of $X$,
\[\bfX(\omega') = (X,\bbX)(\omega')= (X(\omega'), \Pi(X,X)(\omega')\]
and rough semimartingale $Z$ as in \eqref{equ:ZZZ} we have, for a.e. $\omega$ and a.e. $\omega'$,
\begin{equation}
\label{eq:Ito-solution=RSDE-solution}
\tilde Z(Z_{0};\omega, \omega') = Z(\bfX(\omega'),Z_{0}(\omega,\omega'); \omega).
\end{equation}
\end{corollary}
\begin{proof}
In view of uniqueness of the It\^o solution, it suffices to show that the right-hand side of \eqref{eq:Ito-solution=RSDE-solution} is an It\^o solution of \eqref{eq:Ito-eq}.
By Theorem~\ref{thm:RDE-sol=>RSDE-sol}, it suffices to show that
\[
\ucplim_{\dmesh(\pi) \to 0} \sum_{j : \pi_{j} < t} ((D\sigma) \sigma) (\tilde{Z}_{\pi_{j}})\bbX_{\pi_{j},\pi_{j+1}\bmin t} = 0
\]
on $\bar\Omega$, where $\tilde{Z}$ denotes the right-hand side of \eqref{eq:Ito-solution=RSDE-solution}.
This follows from Lemma~\ref{lem:consistency-PiYg}.
\end{proof}
We note that $\mu_t (\omega',A) := \mathbb{P} ( Z_t(\bfX(\omega'),Z_{0}(\cdot ,\omega') \in A )$ gives a regular conditional distribution (r.c.p.) of $\tilde Z_t$ given $X$. This is of interest in filtering theory \cite{MR2454694,MR2884603,MR3134732,MR3752669} where $X$ (resp. $\tilde Z$) are viewed as observation (resp. signal) process.\footnote{With extra notational effort, but no use of abstract results, the r.c.p.\ of $(\tilde Z, \sigma (\tilde Z))$ given $X$ is expressed terms of the distribution of the rough semimartingale $(Z,\sigma(Z))$.}

\medskip

It is not difficult to envision future uses of rough semimartingales. With surely non-exhaustive pointers to the literature, \cite{MR2978136,MR1659958,MR3513594,MR3752669} we can mention specifically rough BSDEs, McKean--Vlasov mean field -, controlled stochastic differential equations, mean field game modeling in presence of common (a.k.a. environmental) noise, modeled by $\sigma(Z^-) \dif \bfX$, as in (\ref{equ:Hybrid}), whereas the martingale component $\mu(Z^-,...) \dif g$ therein can now include all the extra structure not, or not easily, treatable by (rough)pathwise methods.

\subsection{Notation and conventions}
\label{sec:notation}
We write $A \lesssim B$ if there is a constant $C<\infty$, depending only on Lebesgue exponents and variational exponents, denoted by letters $q$ and $p,r$, respectively, such that $A \leq C B$.
This dependence is sometimes emphasized by subscripts such as $\lesssim_{q}$.
In particular, the constant $C$ never depends on the filtered probability space $\Omega$, the processes $F$, or the martingales $g$.
We write $A\sim B$ if $A \lesssim B$ and $B \lesssim A$.
We denote equivalence by definition by ``$:\iff$''.

We call a two-parameter process $(F_{s,t})_{s \leq t}$ \emph{adapted} if $F_{s,t}$ is $\calF_{t}$-measurable for every $s \leq t$.
We call a two-parameter process $(F_{s,t})_{s \leq t}$ \emph{\cadlag{}} if the limits
\begin{align*}
F_{s,t} &= \lim_{\substack{s'\to s, s'\geq s,\\ t'\to t, t' \geq t}} F_{s',t'},
&
F_{s-,t} &:= \lim_{\substack{s'\to s, s'< s,\\ t'\to t, t' \geq t}} F_{s',t'},
\\
F_{s,t-} &:= \lim_{\substack{s'\to s, s'\geq s,\\ t'\to t, t' < t}} F_{s',t'},
&
F_{s-,t-} &:= \lim_{\substack{s'\to s, s'< s,\\ t'\to t, t' < t}} F_{s',t'}
\end{align*}
exist.
The existence of joint limits is important in Lemma~\ref{lem:F-pi-converges-to-F}.

Now we define the convergence modes that we consider.
The set of adapted partitions is a directed set with respect to the inclusion relation $\pi' \subseteq \pi :\iff \Set{\pi'_{n} \given n \in\N} \subseteq \Set{\pi_{n} \given n \in\N} \text{ a.s.}$.
If $(x_{\pi})$ is a net in some metric space indexed by this directed set, we write $\lim_{\pi}x_{\pi}$ for its net limit (if it exists), that is,
\begin{equation}
\label{eq:lim-net}
\lim_{\pi} x_{\pi} = x
\text{ iff }
(\forall \epsilon>0) (\exists \pi_{0}) (\forall \pi \supseteq \pi_{0}) d(x_{\pi},x) < \epsilon.
\end{equation}
The \emph{mesh} of an adapted partition $\pi$ is defined by $\mesh(\pi) := \sup_{j} \norm{\pi_{j}-\pi_{j-1}}_{\infty}$.
We write
\begin{equation}
\label{eq:lim-mesh}
\lim_{\mesh(\pi)\to 0} x_{\pi} = x
\text{ iff }
(\forall \epsilon>0) (\exists \delta>0) (\forall \pi : \mesh(\pi)<\delta) d(x_{\pi},x) < \epsilon.
\end{equation}
Finally, $\lim_{\dmesh(\pi)\to 0}$ is defined as in \eqref{eq:lim-mesh}, but with all partitions $\pi$ being deterministic.

On the space of random processes indexed by $\R_{\geq 0}$ with values in some metric space, the topology of \emph{uniform convergence in probability (u.c.p.)} can be defined by the neighborhood base of a process $f$, indexed by $T,\epsilon>0$, consisting of the sets
\begin{equation}
\label{eq:ucp-topology}
\Set[\Big]{\tilde{f} \given \bbP \Set[\big]{ \sup_{0 \leq t' \leq T} d(\tilde{f}_{t'},f_{t'}) > \epsilon } < \epsilon }.
\end{equation}
The u.c.p.\ topology is metrizable, for example, it is induced by the metric
\begin{equation}
\label{eq:ucp-metric}
\bfd (f,\tilde{f})
:=
\sum_{T\geq 0} 2^{-T} \E \bigl( 1 - 1/(1+\sup_{0 \leq t' \leq T} d(\tilde{f}_{t'},f_{t'})) \bigr).
\end{equation}
If the limit in \eqref{eq:lim-mesh} is taken with respect to the u.c.p.\ topology, we indicate this by writing $\ucplim$ in place of $\lim$.

\section{Vector-valued estimates in discrete time}
\label{sec:vv}
The main result of this section, Theorem~\ref{thm:vv-pprod}, is a bound for discrete time versions of the It\^{o} integral.
Its main advantage over the previous result \cite[Proposition 3.1]{MR4003122} is that the integrands $F^{(k)}$ are allowed to be arbitrary two-parameter processes, rather than martingale differences.
The connection of Theorem~\ref{thm:vv-pprod} with variation norm estimates will be established in Corollary~\ref{cor:main}.
All processes in this section are in discrete time, that is, the time variables are in $\N$.

We begin this section by recalling several known results.
We abbreviate $\norm{\cdot}_{q} := \norm{\cdot}_{L^{q}(\Omega)}$.

\subsection{Davis decomposition}
For a scalar-valued process $(f_{n})$, we denote the martingale maximal function and its stopped version by
\[
Mf := \sup_{n} \abs{f_{n}},
\quad
M_{t}f := \sup_{n \leq t} \abs{f_{n}},
\]
and the martingale square function and its stopped version by
\[
Sf := \ell^{2}_{n} \abs{df_{n}},
\quad
S f_{t} := \ell^{2}_{n} \abs{df_{n}} \one_{n\leq t}.
\]
Here and later,
\[
dg_{j} := g_{j}-g_{j-1}.
\]
We denote $\ell^{p}$ norms by
\[
\ell^{p}_{k} a_{k} := \bigl( \sum_{k \in \N} \abs{a_{k}}^{p} \bigr)^{1/p}.
\]
In order to simplify notation, we only consider martingales $g$ with $g_{0}=0$.

\begin{theorem}[{Davis decomposition, cf.\ \cite{MR0268966}}]
\label{thm:davis-dec}
Let $(f_{n})_{n=0}^{\infty}$ be a martingale with values in a Banach space $X$.
Suppose that $f_{0}=0$ and $f_{n} \in L^{1}(\Omega\to X,\calF_{n})$ for all $n$.
Then there is a decomposition $f_{n} = f^{\pred}_{n} + f^{\bv}_{n}$ into martingales adapted to the same filtration with $f^{\pred}_{0}=0$ such that the differences of $f^{\pred}$ have predictable majorants:
\begin{equation}
\label{eq:davis-predictable}
\norm{d f^{\pred}_{n}}_{X} \leq 2 \sup_{n'<n} \norm{d f_{n'}}_{X}
\end{equation}
and $f^{\bv}$ has bounded variation, in an integral sense for every $q\in [1,\infty)$:
\begin{equation}
\label{eq:davis:bv}
\norm[\Big]{ \sum_{n'\leq n} \norm{d f^{\bv}_{n'}}_{X} }_{L^{q}}
\leq (q+1) \norm[\big]{ \sup_{n' \leq n} \norm{d f_{n'}}_{X} }_{L^{q}}.
\end{equation}
\end{theorem}
We include a proof that gives slightly better constants than the usual one.
\begin{proof}
Abbreviate $M d f_{n} := \sup_{n' \leq n} \norm{df_{n'}}_{X}$.
For $n\geq 1$, let
\[
g_n
:=
\min(1,\frac{Mdf_{n-1}}{\norm{df_n}_{X}}) df_n,
\quad
h_n
:= df_n - g_n
= \max(0,1-\frac{Mdf_{n-1}}{\norm{df_n}_{X}}) df_n.
\]
Then, by definition,
\[
\norm{g_n}_{X} \leq Mdf_{n-1},
\]
and, by positivity of conditional expectation, also
\[
\norm{\EE{ g_n \given \calF_{n-1}} }_{X}
\leq
\EE{ \norm{g_n}_{X} \given \calF_{n-1}}
\leq
\EE{ Mdf_{n-1} \given \calF_{n-1}}
=
Mdf_{n-1}.
\]
This implies \eqref{eq:davis-predictable} for
\[
f^{\pred}_{n} := \sum_{n'=1}^{n} \bigl(g_{n'} - \EE{g_{n'} \given \calF_{n'-1}} \bigr).
\]
Furthermore, we have the telescoping bound
\[
\norm{h_n}_{X}
=
\max(0,\norm{df_n}_{X} - Mdf_{n-1})
=
Mdf_{n} - Mdf_{n-1}.
\]
This implies
\[
\sum_{n'\leq n} \norm{h_{n'}}_{X}
\leq
Mdf_{n}.
\]
By the dual Doob's inequality \cite[Proposition 3.2.8]{MR3617205}, this implies
\begin{multline*}
\norm[\Big]{\sum_{n'\leq n} \norm{\EE{h_{n'} \given \calF_{n'-1}} }_{X} }_{L^{q}}
\leq
\norm[\Big]{\sum_{n'\leq n} \EE{\norm{h_{n'}}_{X} \given \calF_{n'-1}} }_{L^{q}}
\\ \leq
q \norm[\Big]{\sum_{n'\leq n} \norm{h_{n'}}_{X} }_{L^{q}}
\leq
q \norm{Mdf_{n}}_{L^{q}}.
\end{multline*}
The last two estimates imply \eqref{eq:davis:bv} for
\[
f^{\bv}_{n} := \sum_{n'=1}^{n} \bigl( h_{n'} - \EE{h_{n'} \given \calF_{n'-1}} \bigr).
\qedhere
\]
\end{proof}

\begin{lemma}
\label{lem:davis-dec:Spred}
Let $1 \leq q < \infty$, $X$ be a Banach function space, elements of which are $\R$-valued maps $x(\cdot)$, and $(f_{n})$ a martingale with values in $X$.
Then for $f^{\pred}$ given by Theorem~\ref{thm:davis-dec} we have
\[
\norm{ \norm{Sf^{\pred}}_{X} }_{L^{q}}
\leq (q+2) \norm{ \norm{Sf}_{X} }_{L^{q}},
\]
where the square function is given by
\[
\norm{Sf}_{X} := \norm{ \ell^2_{n} ( df_n (\cdot)) }_X
\]
\end{lemma}
\begin{remark}
We will apply Lemma~\ref{lem:davis-dec:Spred} this with $X=\ell^r$, i.e. $r$-summable series, viewed as maps from $\N \to \R$, with the usual Banach structure.
\end{remark}
\begin{proof}[Proof of Lemma~\ref{lem:davis-dec:Spred}]
Using \eqref{eq:davis:bv} we estimate
\begin{align*}
\norm[\big]{ \norm{ S f^{\pred} }_{X} }_{L^{q}}
&\leq
\norm[\big]{ \norm{ S f }_{X} }_{L^{q}}
+
\norm[\big]{ \norm{ S f^{\bv} }_{X} }_{L^{q}}
\\ &\leq
\norm[\big]{ \norm{ S f }_{X} }_{L^{q}}
+
\norm[\Big]{ \norm[\big]{ \sum_{n} \abs{d f^{\bv}_{n}} }_{X} }_{L^{q}}
\\ &\leq
\norm[\big]{ \norm{ S f }_{X} }_{L^{q}}
+
\norm[\Big]{ \sum_{n} \norm{ d f^{\bv}_{n} }_{X} }_{L^{q}}
\\ &\leq
\norm[\big]{ \norm{ S f }_{X} }_{L^{q}}
+
(q+1) \norm[\big]{ \sup_{n} \norm{ d f_{n} }_{X} }_{L^{q}}
\\ &\leq (q+2)
\norm[\big]{ \norm{ S f }_{X} }_{L^{q}}.
\qedhere
\end{align*}
\end{proof}

\subsection{Vector-valued BDG inequality}
We recall the weighted Burkholder--Davis--Gundy inequality.
\begin{lemma}[{\cite{MR3688518}}]
\label{lem:weighted-BDG}
Let $(f_{n})$ be a martingale with respect to a filtration $(\calF_{n})$ and $w$ a positive random variable.
Then
\[
\EE{Mf \cdot w} \leq 16(\sqrt{2}+1) \EE{Sf \cdot Mw},
\]
where $Mw = \sup_{n} \EE{w \given \calF_{n}}$.
\end{lemma}
\begin{remark}
The proof of Lemma~\ref{lem:weighted-BDG} given in \cite{MR3688518} also works for martingales with values in a real Hilbert space.
\end{remark}
\begin{lemma}\label{lem:vv-BDG}
Let $h^{(k)} = (h^{(k)}_{n})$, $k\in\N$, be martingales with respect to some fixed filtration.
Let $1 \leq q<\infty$ and $1 \leq r < \infty$.
Then we have
\begin{equation}\label{eq:l1BDG}
\norm[\big]{ M h^{(k)} }_{L^q(\ell^r_k)}
\lesssim_{q,r}
\norm[\big]{ S h^{(k)} }_{L^q(\ell^r_k)}.
\end{equation}
\end{lemma}
\begin{proof}
First we consider the case $1<q<\infty$.

Take positive functions with $\norm{w^{(k)}}_{L^{q'}(\ell^{r'}_{k})} = 1$.
Then, by Lemma~\ref{lem:weighted-BDG}, we have
\begin{align*}
\EE[\Big]{ \sum_{k} (M h^{(k)}) w^{(k)} }
&\lesssim
\sum_{k} \EE[\big]{ S h^{(k)} M w^{(k)} } \\
&\leq
\norm[\big]{ S h^{(k)} }_{L^q(\ell^r_{k})} \norm[\big]{ M w^{(k)}}_{L^{q'}(\ell^{r'}_{k})},
\end{align*}
where $q',r'$ are Hölder conjugates, that is, $1=1/q+1/q'=1/r+1/r'$.

By the vector-valued Doob's inequality \cite[Theorem 3.2.7]{MR3617205}, we have
\[
\norm[\big]{ M w^{(k)}}_{L^{q'}(\ell^{r'}_{k})}
\lesssim
\norm[\big]{ w^{(k)}}_{L^{q'}(\ell^{r'}_{k})}
=
1.
\]
By duality between $L^q(\ell^r_{k})$ and $L^{q'}(\ell^{r'}_{k})$, this implies the claim \eqref{eq:l1BDG}.

Now we consider $q=1$.
The case $r=1$ follows from the usual BDG inequality, so we may assume $1<r<\infty$.

Decompose $h = h^{\pred} + h^{\bv}$ as in Theorem~\ref{thm:davis-dec} with $X = \ell^{r}$.
The contribution of $h^{\bv}$ to \eqref{eq:l1BDG} is easy to estimate.
In order to estimate the contribution of $h^{\pred}$, for $\lambda > 0$, define the stopping time
\[
\tau := \inf \Set[\big]{ t \given \norm{S h^{\pred}_{t}}_{\ell^{r}}>\lambda \text{ or } \norm{S h_{t}}_{\ell^{r}}>\lambda }.
\]
We claim that
\begin{equation}
\label{eq:stopped-square}
\norm{S h^{\pred}_{\tau}}_{\ell^{r}} \leq \norm{S h^{\pred}}_{\ell^{r}} \wedge 5 \lambda.
\end{equation}
Indeed, the first bound is trivial, and the second bound is only non-void if $0<\tau<\infty$.
In the latter case, by \eqref{eq:davis-predictable}, we have
\[
\norm{S h^{\pred}_{\tau}}_{\ell^{r}}
\leq
\norm{S h^{\pred}_{\tau-1}}_{\ell^{r}}
+
\norm{h^{\pred}_{\tau}-h^{\pred}_{\tau-1}}_{\ell^{r}}
\leq
\lambda
+
4 \sup_{n'<\tau} \norm{h_{n'}-h_{n'-1}}_{\ell^{r}}
\leq
5 \lambda.
\]
Also,
\begin{align*}
\Set{ \norm{ M h^{\pred} }_{\ell^{r}} > \lambda }
&\subseteq
\Set{ \norm{ M h^{\pred}_{\tau} }_{\ell^{r}} > \lambda }
\cup
\Set{ \tau < \infty }
\\ &\subseteq
\Set{ \norm{ M h^{\pred}_{\tau} }_{\ell^{r}} > \lambda }
\cup
\Set{ \norm{ S h }_{\ell^{r}} > \lambda }
\cup
\Set{ \norm{ S h^{\pred} }_{\ell^{r}} > \lambda }
\end{align*}
By the layer cake formula,
\begin{align*}
\norm{Mh^{\pred}}_{L^{1}(\ell^{r})}
&=
\int_{0}^{\infty} \bbP \Set{ \norm{ M h^{\pred} }_{\ell^{r}} > \lambda } \dif\lambda
\\ &\leq
\int_{0}^{\infty} \bbP \Set{ \norm{ M h^{\pred}_{\tau} }_{\ell^{r}} > \lambda } \dif\lambda
+
\int_{0}^{\infty} \bbP \Set{ \norm{ S h^{\pred} }_{\ell^{r}} > \lambda } \dif\lambda
\\ & \quad +
\int_{0}^{\infty} \bbP \Set{ \norm{ S h }_{\ell^{r}} > \lambda } \dif\lambda
=: I + II + III.
\end{align*}
The term $III$ is the claimed right-hand side of the estimate~\eqref{eq:l1BDG}, again by the layer cake formula.
By Lemma~\ref{lem:davis-dec:Spred}, we have
\[
II =
\norm{ \norm{ S h^{\pred} }_{\ell^{r}} }_{L^{1}}
\lesssim
\norm{ \norm{ S h }_{\ell^{r}} }_{L^{1}}.
\]
Using the already known $L^{r}(\ell^{r})$ case of Lemma~\ref{lem:vv-BDG} and \eqref{eq:stopped-square}, we bound the first term by
\begin{align*}
I &\lesssim
\int_{0}^{\infty} \lambda^{-r} \norm{ S h^{\pred}_{\tau} }_{L^{r}(\ell^{r})}^{r} \dif\lambda
\\ &\leq
\int_{0}^{\infty} \lambda^{-r} \norm{ \norm{ S h^{\pred} }_{\ell^{r}} \wedge 5\lambda }_{L^{r}}^{r} \dif\lambda
\\ &=
\bbE \int_{0}^{\infty} \min\bigl(\lambda^{-r} \norm{ S h^{\pred} }_{\ell^{r}}^{r}, 5^{r} \bigr) \dif\lambda
\\ &\lesssim
\bbE \norm{S h^{\pred}}_{\ell^{r}}
= II,
\end{align*}
and we reuse the previously established estimate for $II$.
\end{proof}

\begin{remark}
L\'epingle's inequality \eqref{eq:Lepingle} can be obtained from Lemma~\ref{lem:vv-BDG} and Corollary~\ref{cor:paraprod-stopping}.
In fact, Corollary~\ref{cor:paraprod-stopping} simplifies for processes $\Pi$ that are of difference form, see \cite[Corollary 2.4]{MR4091110}, so that the vector-valued bound \eqref{eq:l1BDG} is not necessary to show \eqref{eq:Lepingle}.
\end{remark}

\subsection{Vector-valued maximal paraproduct estimate}
For an adapted process $(F_{s,t})$ and a martingale $(g_{n})$, we define
\begin{equation}
\label{eq:paraprod}
\Pi_{s,t}(F,g) := \sum_{s < j \leq t} F_{s,j-1} dg_{j}
=
\sum_{s \leq j < t} F_{s,j} (g_{j+1}-g_{j}).
\end{equation}
Note that $\Pi(F,g)_{s,\cdot}$ only depends on $(F_{s,\cdot})$.

\begin{proposition}
\label{prop:vv-pprod}
Let $0 < q,q_{1} \leq \infty$, $1 \leq q_{0},r,r_{0} < \infty$, $1 \leq r_{1} \leq \infty$.
Assume $1/q = 1/q_{0} + 1/q_{1}$ and $1/r = 1/r_{0} + 1/r_{1}$.
Then, for any martingales $(g^{(k)}_{n})_{n}$, any adapted sequences $(F^{(k)}_{s,t})_{s\leq t}$, and any stopping times $\tau_{k}' \leq \tau_{k}$ with $k\in\Z$, we have
\begin{equation}
\label{eq:vv-pprod}
\norm[\big]{ \ell^{r}_{k} \sup_{\tau_{k}' \leq t \leq \tau_{k}} \abs{\Pi(F^{(k)},g^{(k)})_{\tau_{k}',t}} }_{q}
\leq C_{q_{0},q_{1},r_{0},r_{1}}
\norm[\big]{ \ell^{r_{1}}_{k} \sup_{\tau_{k}' \leq t < \tau_{k}} \abs{F^{(k)}_{\tau_{k}',t}} }_{q_{1}}
\norm{ \ell^{r_{0}}_{k} Sg^{(k)}_{\tau_{k}',\tau_{k}} }_{q_{0}},
\end{equation}
where $Sg_{s,t} := \bigl( \sum_{j=s+1}^{t} \abs{dg_{j}}^{2} \bigr)^{1/2}$.
\end{proposition}

\begin{proof}[Proof of Proposition~\ref{prop:vv-pprod}]
We may replace each $g^{(k)}$ by the martingale
\begin{equation}
\label{eq:stop-g}
\tilde{g}^{(k)}_{n}
:=
g^{(k)}_{n \wedge \tau_{k}} - g^{(k)}_{n \wedge \tau_{k}'}
\end{equation}
without changing the value of either side of \eqref{eq:vv-pprod}.

Consider first $q \geq 1$.
For each $k$, the sequence
\[
h^{(k)}_{t}
:=
\begin{cases}
0, & t < \tau_{k}',\\
\Pi(F^{(k)},g^{(k)})_{\tau_{k}',t}, & t \geq \tau_{k}',
\end{cases}
\]
is a martingale.
We may also assume $F_{\tau_{k}',t} = 0$ if $t \not\in [\tau_{k}',\tau_{k})$.
By Lemma~\ref{lem:vv-BDG}, we can estimate
\begin{align*}
LHS~\eqref{eq:vv-pprod}
& \lesssim
\norm[\big]{ \ell^{r}_{k} \abs{S h^{(k)}} }_{q}
\\ &=
\norm[\big]{ \ell^{r}_{k} \ell^{2}_{j} \abs{F_{\tau_{k}',j-1}^{(k)} dg^{(k)}_{j}} }_{q}
\\ &\leq
\norm[\big]{ \ell^{r}_{k} MF^{(k)} \ell^{2}_{j} \abs{dg^{(k)}_{j}} }_{q}
\\ &\leq
\norm{ \ell^{r_{1}}_{k} MF^{(k)} }_{q_{1}}
\norm[\big]{ \ell^{r_{0}}_{k} Sg^{(k)} }_{q_{0}}.
\end{align*}
Here and later, we abbreviate $MF^{(k)} := \sup_{j} \abs{F^{(k)}_{\tau_{k}',j}}$.

Consider now $q<1$.
Multiplying $F=(F^{(k)})_{k\in\Z}$ by some scalar and $g=(g^{(k)})_{k\in\Z}$ by some other scalar, we may assume
\begin{equation}
\label{eq:vv-paraprod-scaling-assumption}
\norm[\big]{ \ell^{r_{1}}_{k} MF^{(k)} }_{q_{1}}
= \norm[\big]{ \ell^{r_{0}}_{k} Sg^{(k)} }_{q_{0}}
= 1,
\end{equation}
and we have to show
\[
\norm[\big]{\ell^{r}_{k} \sup_{\tau_{k}' \leq t \leq \tau_{k}} \abs{\Pi(F^{(k)},g^{(k)})_{\tau_{k}',t}} }_{q}
\lesssim 1.
\]

We use the Davis decomposition $g=g^{\pred}+g^{\bv}$ (Theorem~\ref{thm:davis-dec} with $X=\ell^{r_{0}}$).
The contribution of the bounded variation part is estimated as follows:
\begin{align*}
\MoveEqLeft
\norm[\big]{ \ell^{r}_{k} \sup_{\tau_{k}' \leq t \leq \tau_{k}} \abs{\Pi(F^{(k)},g^{(k),\bv})_{\tau_{k}',t}} }_{q}
\\ &\leq
\norm[\Big]{ \ell^{r}_{k} \sum_{j} \abs{F^{(k)}_{\tau_{k}',j-1}} \cdot \abs{dg^{(k),\bv}_{j}} }_{q}
\\ &\leq
\norm{ \ell^{r_{1}}_{k} MF^{(k)} }_{q_{1}}
\norm[\Big]{ \ell^{r_{0}}_{k} \Bigl( \sum_{j} \abs{dg^{(k),\bv}_{j}} \Bigr) }_{q_{0}}
\\ &\leq
\norm{ \ell^{r_{1}}_{k} MF^{(k)} }_{q_{1}}
\norm[\Big]{ \sum_{j} \ell^{r_{0}}_{k} \abs{dg^{(k),\bv}_{j}} }_{q_{0}}
\\ &\lesssim
\norm{ \ell^{r_{1}}_{k} MF^{(k)} }_{q_{1}}
\norm{ \sup_{j} \ell^{r_{0}}_{k} \abs{dg^{(k)}_{j}} }_{q_{0}}
\\ &\leq
\norm{ \ell^{r_{1}}_{k} MF^{(k)} }_{q_{1}}
\norm{ \ell^{r_{0}}_{k} Sg^{(k)} }_{q_{0}},
\end{align*}
where we used \eqref{eq:davis:bv} in the penultimate step.

It remains to consider the part $g^{\pred}$ with predictable bounds for jumps.
By the layer cake formula, we have
\begin{multline}
\label{eq:vv-paraprod-layer-cake}
\norm[\big]{ \ell^{r}_{k} \sup_{\tau_{k}' \leq t \leq \tau_{k}} \abs{\Pi(F^{(k)},g^{(k),\pred})_{\tau_{k}',t}} }_{q}^{q}
\\ =
\int_{0}^{\infty} \P \Set[\big]{ \ell^{r}_{k} \sup_{\tau_{k}' \leq t \leq \tau_{k}} \abs{\Pi(F^{(k)},g^{(k),\pred})_{\tau_{k}',t}} > \lambda^{1/q} } \dif \lambda.
\end{multline}
Fix some $\lambda>0$ and define a stopping time
\begin{equation}
\label{eq:vv-paraprod:stopping}
\tau := \inf \Set[\Big]{ t \given
\ell^{r_{0}}_{k} S g^{(k)}_{t} > \lambda^{1/q_{0}} \text{ or }
\ell^{r_{0}}_{k} S g^{(k),\pred}_{t} > \lambda^{1/q_{0}} \text{ or }
\ell^{r_{1}}_{k} \sup_{0 < j \leq t} \abs{F_{\tau_{k}',j}^{(k)}} > \lambda^{1/q_{1}} }.
\end{equation}
Define stopped martingales $\tilde{g}^{(k)}_{t} := g^{(k),\pred}_{t \wedge \tau}$ and adapted processes
\[
\tilde{F}^{(k)}_{t,t'} :=
F^{(k)}_{t, t' \bmin \tau-1}.
\]
Then, on the set $\Set{\tau=\infty}$, we have
\[
\Pi(F^{(k)},g^{(k),\pred})_{\tau_{k}',t}
=
\Pi(\tilde{F}^{(k)},\tilde{g}^{(k)})_{\tau_{k}',t}
\quad\text{for all } k,t.
\]
Hence,
\begin{equation}
\label{eq:10}
\begin{split}
\MoveEqLeft
\Set[\big]{ \ell^{r}_{k} \sup_{\tau_{k}' \leq t \leq \tau_{k}} \abs{\Pi(F^{(k)},g^{(k),\pred})_{\tau_{k}',t}} > \lambda^{1/q} }
\\ \subset&
\Set[\big]{ \ell^{r}_{k} \sup_{\tau_{k}' \leq t \leq \tau_{k}} \abs{\Pi(\tilde{F}^{(k)},\tilde{g}^{(k)})_{\tau_{k}',t}} > \lambda^{1/q} }
\\ &\cup
\Set{ \ell^{r_{0}}_{k} S g^{(k)} > \lambda^{1/q_{0}} }
\cup \Set{ \ell^{r_{0}}_{k} S g^{(k),\pred} > \lambda^{1/q_{0}} }
\\ &\cup
\Set{ \ell^{r_{1}}_{k} MF^{(k)} > \lambda^{1/q_{1}} }
\end{split}
\end{equation}
The contributions of the latter three terms to \eqref{eq:vv-paraprod-layer-cake} are $\lesssim 1$ by \eqref{eq:vv-paraprod-scaling-assumption} and Lemma~\ref{lem:davis-dec:Spred}.
It remains to handle the first term.

By construction, we have $\ell^{r_{1}}_{k} M\tilde{F}^{(k)} \leq \lambda^{1/q_{1}}$, and due to \eqref{eq:davis-predictable} we also have $\ell^{r_{0}}_{k} S \tilde{g}^{(k)} \leq 3\lambda^{1/q_{0}}$.
Choose an arbitrary exponent $\tilde{q}$ with $q_{0} < \tilde{q} < \infty$.
By the already known case of the Proposition with $(q_{0},q_{1})$ replaced by $(\tilde{q},\infty)$, we obtain
\begin{equation}
\label{eq:12}
\begin{split}
\MoveEqLeft
\P \Set[\big]{ \ell^{r}_{k} \sup_{\tau_{k}' \leq t \leq \tau_{k}} \abs{\Pi(\tilde{F}^{(k)},\tilde{g}^{(k)})_{\tau_{k}',t}} > \lambda^{1/q} }
\\ &\leq
\lambda^{-\tilde{q}/q} \norm{ \ell^{r}_{k} \sup_{\tau_{k}' \leq t \leq \tau_{k}} \abs{\Pi(\tilde{F}^{(k)},\tilde{g}^{(k)})_{\tau_{k}',t}} }_{\tilde{q}}^{\tilde{q}}
\\ &\lesssim_{\tilde{q}}
\lambda^{-\tilde{q}/q} \norm{\ell^{r_{1}}_{k} M\tilde{F}^{(k)}}_{\infty}^{\tilde{q}}
\norm{\ell^{r_{0}}_{k} S \tilde{g}^{(k)} }_{\tilde{q}}^{\tilde{q}}
\\ &\leq
\lambda^{-\tilde{q}/q_{0}} \norm{\ell^{r_{0}}_{k} S g^{(k),\pred} \wedge 3\lambda^{1/q_{0}}}_{\tilde{q}}^{\tilde{q}}.
\end{split}
\end{equation}
This estimate no longer depends on the stopping time $\tau$.
Integrating the right-hand side of \eqref{eq:12} in $\lambda$, we obtain
\begin{align*}
\int_{0}^{\infty} \lambda^{-\tilde{q}/q_{0}} \norm{\ell^{r_{0}}_{k} S g^{(k),\pred} \wedge 3\lambda^{1/q_{0}}}_{\tilde{q}}^{\tilde{q}}
\dif \lambda
&=
\mathbb{E} \int_{0}^{\infty}
\bigl( \lambda^{-\tilde{q}/q_{0}} (\ell^{r_{0}}_{k} S g^{(k),\pred})^{\tilde{q}} \wedge 3^{\tilde{q}} \bigr)
\dif \lambda
\\ &\sim
\mathbb{E} (\ell^{r_{0}}_{k} S g^{(k),\pred})^{q_{0}}
\\ &\sim 1,
\end{align*}
where we used $\tilde{q}>q_{0}$, Lemma~\ref{lem:davis-dec:Spred} with $X=\ell^{r_{0}}$, and the assumption \eqref{eq:vv-paraprod-scaling-assumption}.
\end{proof}

Next, we deduce a version of Proposition~\ref{prop:vv-pprod} that involves a two-parameter supremum of the kind that appears in Corollary~\ref{cor:paraprod-stopping}.
Recall the definition of second order increments of a two-parameter process $(F_{s,t})$:
\begin{equation}
\label{eq:delta2}
(\delta F)_{s,t,u} := F_{s,u} - F_{s,t} - F_{t,u},
\quad
s < t < u.
\end{equation}
For a fixed $s$, we define
\begin{equation}
\label{eq:delta2-first-fixed}
(\delta_{s} F)_{t,u} := F_{s,u} - F_{t,u},
\quad
s < t < u.
\end{equation}

\begin{theorem}
\label{thm:vv-pprod}
In the situation of Proposition~\ref{prop:vv-pprod}, we have
\begin{equation}
\label{eq:vv-pprod-double-sup}
\begin{split}
\MoveEqLeft
\norm[\big]{ \ell^{r}_{k} \sup_{\tau_{k}' \leq s<t \leq \tau_{k}} \abs{\Pi_{s,t}(F^{(k)},g^{(k)})} }_{q}
\leq
\norm[\big]{ \ell^{r}_{k} \sup_{\tau_{k}' \leq s<t \leq \tau_{k}} \abs{\Pi_{s,t}(\delta_{\tau_{k}'} F^{(k)},g^{(k)})} }_{q}
\\ &+
C_{q_{0},q_{1},r_{0},r_{1}}
\norm[\big]{ \ell^{r_{1}}_{k} \sup_{\tau_{k}' \leq t < \tau_{k}} \abs{F^{(k)}_{\tau_{k}',t}} }_{q_{1}}
\norm{ \ell^{r_{0}}_{k} Sg^{(k)}_{\tau_{k}',\tau_{k}} }_{q_{0}},
\end{split}
\end{equation}
where $(Sg_{s,t})^{2} = \sum_{s < j \leq t} \abs{dg_{j}}^{2}$.
\end{theorem}
\begin{proof}
For $s\leq t\leq u$, the sums \eqref{eq:paraprod} satisfy the relation
\begin{equation}
\label{eq:Pi-Chen-relation}
\begin{split}
\delta \Pi(F,g)_{s,t,u}
&=
\Pi_{s,u}(F,g) - \Pi_{s,t}(F,g) - \Pi_{t,u}(F,g)
\\ &=
\sum_{t < j \leq u} (F_{s,j-1} - F_{t,j-1} ) dg_{j}
\\ &=
\Pi(\delta_{s} F,g)_{t,u}.
\end{split}
\end{equation}
Therefore, we can estimate
\[
\abs{\Pi_{s,t}(F,g)}
\leq
\abs{\Pi_{\tau_{k}',t}(F,g)} + \abs{\Pi_{\tau_{k}',s}(F,g)} + \abs{\Pi(\delta_{\tau_{k}'} F,g)_{s,t}}.
\]
The contribution of the first two terms is bounded by Proposition~\ref{prop:vv-pprod}.
The contribution of the last term to the left-hand side of \eqref{eq:vv-pprod-double-sup} is that left-hand side with $F$ replaced by $\delta_{\tau_{k}'} F$.
\end{proof}

We will use Theorem~\ref{thm:vv-pprod} with $\tau_{k}'=\tau_{k-1}$, where $(\tau_{k})$ is an adapted partition, $g^{(k)}=g$, and $F^{(k)}=F$.
It is most useful in presence of a structural hypothesis on $\delta F$ of the kind introduced in \cite[Lemma 3.1]{MR2578445}.
\begin{corollary}
\label{cor:vv-pprod-delta-F}
Let $q,q_{0},q_{1},r,r_{1}$ be as in Proposition~\ref{prop:vv-pprod} with $r_{0}=2$.
Let $(F_{s,t})$ be an adapted process such that
\begin{equation}
\label{eq:delta-F-structure}
\delta_{s} F_{t,u} = \sum_{i=1}^{i_{\max}} F^{i}_{s,t} \tilde{F}^{i}_{t,u}
\end{equation}
with adapted processes $F^{i},\tilde{F}^{i}$, $g$ a martingale, and $(\tau_{k})$ an adapted partition.
Then, we have
\begin{equation}
\label{eq:vv-pprod-structured-F}
\begin{split}
\MoveEqLeft
\norm[\big]{ \ell^{r}_{k} \sup_{\tau_{k-1} \leq s < t \leq \tau_{k}} \abs{\Pi(F,g)_{s,t}} }_{q}
\lesssim
\sum_{i=1}^{i_{\max}} \norm[\big]{ \ell^{r}_{k} \bigl(\sup_{\tau_{k-1}\leq s < t \leq \tau_{k}} \abs{F^{i}_{\tau_{k-1},s}} \cdot \abs{\Pi(\tilde{F}^{i},g)_{s,t}} \bigr) }_{q}
\\ &+
\norm[\big]{ \ell^{r_{1}}_{k} \sup_{\tau_{k-1} \leq t < \tau_{k}} \abs{F_{\tau_{k-1},t}} }_{q_{1}}
\norm{ Sg }_{q_{0}}.
\end{split}
\end{equation}
\end{corollary}

\subsection{Branched rough paths} \label{sec:discrete-branched}
In this section, we iterate Corollary~\ref{cor:vv-pprod-delta-F} by applying it recursively to each term $\Pi(\tilde{F}^{i},g)$ on the right-hand side of \eqref{eq:vv-pprod-structured-F}.
The algebraic framework for this iteration is provided by the theory of branched rough paths introduced in \cite{MR2578445}, see also \cite{MR3300969}.
We recall the relevant notation from \cite{MR2578445}.
We fix a finite set of labels $\calL$.
The set of (finite) trees with vertices labeled by the elements of $\calL$ is denoted by $\mathcal{T_{L}}$.
A \emph{forest} is a finite unordered tuple of trees in $\mathcal{T_{L}}$, in which repetition is allowed.
The set of all forests is denoted by $\mathcal{F_{L}}$.
The free commutative $\R$-algebra generated by the trees $\mathcal{T_{L}}$ is denoted by $\mathcal{AT_{L}}$.
It can be identified with the free $\R$-vector space generated by $\mathcal{F_{L}}$.

A \emph{branched rough path} is an algebra homomorphism\footnote{In discrete time, we do not need a regularity assumption. Suitable bounded $p$-variation assumptions are of course needed to transfer our results to continuous time.}
\[
F : \mathcal{AT_{L}} \to \calC_{2},
\]
where $\calC_{2}$ is the algebra of \cadlag{} functions on the simplex $\Set{ (s,t) \given s<t }$, that satisfies the generalized Chen relation
\begin{equation}
\label{eq:Chen-branched}
\delta F^{\frakf} = F^{\Delta(\frakf) - 1 \otimes \frakf - \frakf \otimes 1},
\quad \frakf \in \mathcal{AT_{L}}.
\end{equation}
On the right-hand side, we use the extension of $F$ to an algebra homomorphism $\mathcal{AT_{L}} \otimes \mathcal{AT_{L}} \to \calC_{3}$ defined by $F^{\frakf \otimes \frakf'} = F^{\frakf} F^{\frakf'}$, where we use the product $\calC_{2} \times \calC_{2} \to \calC_{3}$ given by $(FG)_{stu}=F_{st}G_{tu}$.
The coproduct $\Delta : \mathcal{AT_{L}} \to \mathcal{AT_{L}} \otimes \mathcal{AT_{L}}$ is an algebra homomorphism acting on forests by
\begin{equation}
\label{eq:coproduct}
\Delta(\frakf) = \sum_{(\frakb,\frakr) \in \Cut \frakf} \frakb \otimes \frakr,
\end{equation}
where the sum goes over the multiset of all \emph{admissible cuts}, that is, partitions of trees in the forest $\frakf$ into (possibly empty) initial trees collected in the forest $\frakr$ (for ``roots'') and final trees collected in the forest $\frakb$ (for ``branches'').
Our convention for cuts is different from \cite[eq. (3)]{MR2578445}, in that we allow roots and branches to be empty.

\begin{theorem}
\label{thm:vv-pprod-branched}
Let $q \in (0,\infty)$, $q_{0} \in [1,\infty)$, and, for each tree $\frakt \in \mathcal{T_{L}}$, let $q_{\frakt} \in (0,\infty]$.
Let $r \in [1,\infty)$ and, for each tree $\frakt \in \mathcal{T_{L}}$, let $r_{\frakt} \in [1,\infty]$.
Let $\frakf \in \mathcal{F_{L}}$ be a forest and let $\frakF$ be the set of all forests $\frakf'$ that are the disjoint unions of arbitrary partitions of trees in $\frakf$ into subtrees.
Assume that, for each $\frakf' \in \frakF$, we have
\[
1/q = 1/q_{0} + \sum_{\frakt \in \frakf'} 1/q_{\frakt},
\quad
1/r = 1/2 + \sum_{\frakt \in \frakf'} 1/r_{\frakt}.
\]
Let $F$ be an adapted family of branched rough paths, $g$ a martingale, and $\tau$ an adapted partition.
Then, we have
\begin{equation}
\label{eq:vv-pprod-branched}
\norm[\big]{ \ell^{r}_{k} \sup_{\tau_{k-1} \leq s<t \leq \tau_{k}} \abs{\Pi_{s,t}(F^{\frakf},g)} }_{q}
\lesssim
\sum_{\frakf' \in \frakF}
\Bigl( \prod_{\frakt \in \frakf'} \norm[\big]{ \ell^{r_{\frakt}}_{k} \sup_{\tau_{k-1} \leq t < \tau_{k}} \abs{F^{\frakt}_{\tau_{k-1},t}} }_{q_{\frakt}} \Bigr)
\norm{ Sg }_{q_{0}}.
\end{equation}
\end{theorem}
\begin{proof}
We use strong induction on the degree of the forest $\frakf$, that is, the total number of vertices in its trees.
Let $\frakf$ be given and suppose that the claim is known for all forests with strictly smaller degree.
By the generalized Chen relation \eqref{eq:Chen-branched} and the definition of the coproduct \eqref{eq:coproduct}, we have
\begin{equation}
\label{eq:Chen0-branched}
\delta_{s} F^{\frakf}_{t,u} =
\sum_{(\frakb,\frakr) \in \Cut(\frakf), \frakb \neq 0} F^{\frakb}_{s,t} F^{\frakr}_{t,u}.
\end{equation}

We apply Corollary~\ref{cor:vv-pprod-delta-F} with $r_{1} = r_{\frakf}$, $q_{1} = q_{\frakf}$, where $1/r_{\frakf} = \sum_{\frakt\in\frakf} 1/r_{\frakt}$ and $1/q_{\frakf} = \sum_{\frakt\in\frakf} 1/q_{\frakt}$.
Then the second term on the right-hand side of \eqref{eq:vv-pprod-structured-F} corresponds to the summand $\frakf'=\frakf$ in \eqref{eq:vv-pprod-branched}.

It remains to estimate the first term on the right-hand side of \eqref{eq:vv-pprod-structured-F}.
For a fixed cut $(\frakb,\frakr)$, we have
\begin{align*}
\MoveEqLeft
\norm[\big]{ \ell^{r}_{k} \sup_{\tau_{k-1} \leq s<t \leq \tau_{k}} \abs{F^{\frakb}_{\tau_{k-1},s}} \abs{\Pi(F^{\frakr},g)_{s,t}} }_{q}
\\ &\leq
\prod_{\frakt' \in \frakb}
\norm{\ell^{r_{\frakt'}}_{k} \sup_{\tau_{k-1} \leq s < \tau_{k}} \abs{F^{\frakt'}_{\tau_{k-1},s}}}_{q_{\frakt'}}
\cdot
\norm[\big]{ \ell^{\tilde{r}}_{k} \sup_{\tau_{k-1} \leq s<t \leq \tau_{k}} \abs{\Pi(F^{\frakr},g)_{s,t}} }_{\tilde{q}},
\end{align*}
where
\[
1/\tilde{q} = 1/q - \sum_{\frakt' \in \frakb} 1/q_{\frakt'},
\quad
1/\tilde{r} = 1/r - \sum_{\frakt' \in \frakb} 1/r_{\frakt'}.
\]
The latter norm can be estimated by the inductive hypothesis, since $\deg \frakr < \deg \frakf$.
\end{proof}

\begin{example}[Vector-valued BDG inequality]
The vector-valued BDG inequality~\ref{lem:vv-BDG} is the case of the empty forest $\frakf$ in Theorem~\ref{thm:vv-pprod-branched}.
In this case, we have $F^{\frakf} \equiv 1$, so that
\[
\Pi(F^{\frakf},g) = \delta g.
\]
Therefore, the estimate~\eqref{eq:vv-pprod-branched} becomes \eqref{eq:l1BDG}.
\end{example}

\begin{example}[Differences]
Suppose that $F=\delta f$.
This corresponds to the forest $\frakf$ consisting of the single tree
\begin{forest}
[a]
\end{forest}.
In this case, $\frakF = \Set{\frakf}$, and Theorem~\ref{thm:vv-pprod-branched} gives
\[
\norm[\big]{ \ell^{r}_{k} \sup_{\tau_{k-1} \leq s<t \leq \tau_{k}} \abs{\Pi(F,g)_{s,t}} }_{q}
\leq
C_{q_{0},q_{1},r_{1}}
\norm[\big]{ \ell^{r_{1}}_{k} \sup_{\tau_{k-1} \leq t < \tau_{k}} \abs{\delta f_{\tau_{k-1},t}} }_{q_{1}}
\norm{ Sg }_{q_{0}}.
\]
\end{example}

\begin{example}[Product of differences]
More generally, suppose that
\begin{equation}
\label{eq:prod-of-delta}
F = \prod_{j} \delta f^{(j)}.
\end{equation}
This corresponds to the forest $\frakf$ being
\begin{forest}
phantom,
[[1] [\dots] ]
\end{forest}.
In this case, we also have $\frakF = \Set{\frakf}$, and \eqref{eq:vv-pprod-branched} with $1/r=1/2+\sum_{j}1/r_{j}$, $1/q=1/q_{0}+\sum_{j} 1/q_{j}$ becomes
\begin{multline}
\label{eq:vv-pprod-bushy}
\norm[\big]{ \ell^{r}_{k} \sup_{\tau_{k-1} \leq s < t \leq \tau_{k}} \abs{\Pi_{s,t}( \prod_{j} \delta f^{(j)},g)} }_{q}
\lesssim
\prod_{j} \norm[\big]{ \ell^{r_{j}}_{k} \sup_{\tau_{k-1} \leq t < \tau_{k}} \abs{\delta f^{(j)}_{\tau_{k-1},t}} }_{q_{j}}
\cdot
\norm{ Sg }_{q_{0}}.
\end{multline}
\end{example}

\begin{example}[Second level of a rough path]
Suppose that $F=\bbX$, where $\bbX$ is the second level of a rough path.
This corresponds to the forest $\frakf$ consisting of the single tree
\begin{forest}
[a [b]]
\end{forest}.
The family $\frakF$ then consists of the two forests
\[
\begin{forest}
[a [b]]
\end{forest}
\quad\text{and}\quad
\begin{forest}
phantom,
[[a] [b]]
\end{forest}.
\]
Suppose also, for simplicity, $r_{a}=r_{b}=2r_{1}$ and $q_{a}=q_{b}=2q_{1}$.
In this case, Theorem~\ref{thm:vv-pprod-branched} gives the estimate
\begin{multline*}
\norm[\big]{ \ell^{r}_{k} \sup_{\tau_{k-1} \leq s<t \leq \tau_{k}} \abs{\Pi_{s,t}(\bbX,g)} }_{q}
\lesssim
\norm[\big]{ \ell^{r_{1}}_{k} \sup_{\tau_{k-1} \leq t < \tau_{k}} \abs{\bbX^{(k)}_{\tau_{k-1},t}} }_{q_{1}}
\norm{ Sg }_{q_{0}}.
\\ +
\norm[\big]{ \ell^{2r_{1}}_{k} \sup_{\tau_{k-1} \leq t < \tau_{k}} \abs{\delta X_{\tau_{k-1},t}} }_{2q_{1}}^{2}
\norm{ Sg }_{q_{0}}.
\end{multline*}
\end{example}

\begin{example}[A bushy tree]
Suppose that forest $\frakf$ consisting of the single tree
\begin{forest}
[a [b] [c]]
\end{forest}.
The family $\frakF$ then consists of the four forests
\[
\begin{forest}
[a [b] [c]]
\end{forest},
\quad
\begin{forest}
phantom,
[[a [c]] [b]]
\end{forest},
\quad
\begin{forest}
phantom,
[[a [b]] [c]]
\end{forest},
\quad\text{and}\quad
\begin{forest}
phantom,
[[a] [b] [c]]
\end{forest}.
\]
\end{example}

\section{Variational estimates in discrete time}
In this section, we will estimate $V^{r} \Pi(F,g)$ in open ranges $r>\rho$.
There is a dichotomy depending on the value of the threshold $\rho$.
For $\rho<1$, we will use the sewing lemma, see Section~\ref{sec:sewing}.
The main new results of this article are in the range $\rho \geq 1$.
In this range, pathwise estimates are insufficient, and we have to rely on the cancellation provided by the martingale $g$.
By the construction in Section~\ref{sec:var-stopping}, variation norm estimates in this range follow directly from the vector-valued estimates in Section~\ref{sec:vv}.
All processes in this section are in discrete time, that is, the time variables are in $\N$.

\subsection{Stopping time construction}
\label{sec:var-stopping}
In this section, we will bound $r$-variation by square function-like objects.
For L\'epingle's inequality, this idea was introduced in \cite{MR1019960,MR933985}.
It was first applied to a (real variable) paraproduct in \cite{MR2949622}.
The stopping time argument in \cite{MR1019960,MR933985} involves a real interpolation step that was made increasingly more explicit in \cite{MR2434308,MR4055178}.
We use different stopping times, which better capture the structure of the process at hand and avoid the real interpolation step.
For L\'epingle's inequality, similar stopping times were introduced in \cite{MR4091110}.
One of the advantages of the present construction is that it allows us to remove a restriction on the integrability parameters ($q_0 > 1$) from \cite{MR4003122}.

For an adapted process $(\Pi_{s,t})_{s \leq t}$, let
\[
\Pimax_{n''} := \sup_{ 0 \leq n < n' \leq n''} \abs{\Pi_{n,n'}},
\quad
\Pimax := \Pimax_{\infty}.
\]
\begin{lemma}
\label{lem:paraprod-stopping}
For any discrete time adapted process $(\Pi_{s,t})_{s<t}$, there exist sequences of stopping times $\tau^{(m)}_{j}$, increasing in $j\geq 0$ for each $m\in\N$, such that for every $0<\rho<r<\infty$ we have
\begin{multline}
\label{eq:5}
\sup_{\substack{l_{\max},\\ u_{0} < \dotsb < u_{l_{\max}}}}
\sum_{l=1}^{l_{\max}} \abs{\Pi_{u_{l-1},u_{l}}}^{r}
\leq
\frac{(\Pimax)^{r}}{1-2^{-r}}
+ 2^{\rho} \sum_{m=0}^{\infty} (2^{-m}\Pimax)^{r-\rho} \sum_{j=1}^{\infty} \Bigl( \sup_{\tau^{(m)}_{j-1} \leq t < \tau^{(m)}_{j}} \abs{\Pi_{t, \tau^{(m)}_{j}}} \Bigr)^{\rho}.
\end{multline}
\end{lemma}

\begin{proof}[Proof of Lemma~\ref{lem:paraprod-stopping}]
For $m\in\N$, define stopping times
\[
\tau^{(m)}_{0} := 0,
\]
and then, for $j \ge 0$, allowing values in $\N \cup \Set{\infty}$,
\begin{equation}
\label{eq:stopping-time}
\tau^{(m)}_{j+1} := \inf \Set[\Big]{ t > \tau^{(m)}_{j} \given \sup_{\tau^{(m)}_{j} \leq t' < t} \abs{ \Pi_{t', t} } > 2^{-m-1} \Pimax_{t} }.
\end{equation}
Fix $\omega\in\Omega$ and let $(u_{l})_{l=0}^{l_{\max}}$ be a finite strictly increasing sequence.
Consider $0<\rho<r<\infty$ and split
\begin{equation}
\label{eq:4}
\sum_{l=1}^{l_{\max}} \abs{\Pi_{u_{l-1},u_{l}}}^{r}
=
\sum_{m=0}^{\infty} \sum_{l \in L(m)} \abs{\Pi_{u_{l-1},u_{l}}}^{r},
\end{equation}
where
\begin{equation}
\label{eq:3}
L(m) := \Set[\big]{ l \in \Set{1,\dotsc,l_{\max}} \given 2^{-m-1}\Pimax_{u_{l}} < \abs{\Pi_{u_{l-1},u_{l}}} \leq 2^{-m}\Pimax_{u_{l}} }.
\end{equation}
In \eqref{eq:4}, we only omitted vanishing summands, since $\abs{\Pi_{u_{l-1},u_l}} \leq \Pimax_{u_l}$.
Let also $L'(m) := L(m) \setminus \Set{\sup L(m)}$.
Using \eqref{eq:3}, we obtain
\begin{equation}
\label{eq:1}
\sum_{l=1}^{l_{\max}} \abs{\Pi_{u_{l-1},u_{l}}}^{r}
\leq
\sum_{m=0}^{\infty} (2^{-m}\Pimax)^{r-\rho} \sum_{l \in L'(m)} \abs{\Pi_{u_{l-1},u_{l}}}^{\rho}
+ \sum_{m=0}^{\infty} (2^{-m} \Pimax)^{r}.
\end{equation}

\begin{claim}
For every $l \in L(m)$, there exists $j$ s.t. $\tau_j^{(m)} \in (u_{l-1},u_{l}]$.
\end{claim}
\begin{proof}[Proof of the claim]
Let $j$ be maximal with $\tau^{(m)}_{j} \leq u_{l-1}$.
Since $l\in L(m)$, by definition \eqref{eq:3}, we have
\[
\abs{\Pi_{u_{l-1},u_{l}}}
>
2^{-m-1} \Pimax_{u_{l}}.
\]
By the definition of stopping times \eqref{eq:stopping-time}, we obtain $\tau^{(m)}_{j+1} \leq u_{l}$.
\end{proof}

Fix $m$.
For each $l \in L'(m)$, let $j(l)$ be the largest $j$ such that $\tau^{(m)}_{j} \in (u_{l-1}, u_{l}]$.
Then all $j(l)$ are distinct, and, since $l \neq \max L(m)$, the claim shows that $\tau^{(m)}_{j(l)+1} < \infty$.
Furthermore, by \eqref{eq:3}, the monotonicity of $t \mapsto \Pimax_{t}$, and the definition \eqref{eq:stopping-time} of stopping times, we have
\begin{equation}
\label{eq:2}
\abs{\Pi_{u_{l-1},u_{l}}}
\leq
2^{-m} \Pimax_{u_{l}}
\leq
2^{-m} \Pimax_{\tau^{(m)}_{j(l)+1}}
\leq
2 \sup_{\tau^{(m)}_{j(l)} \leq t' < \tau^{(m)}_{j(l)+1}} \abs{ \Pi_{t', \tau^{(m)}_{j(l)+1}}}
\end{equation}
by the definition of $\tau^{(m)}_{j(l)}$.
Since all $j(l)$ are distinct, this implies
\[
\sum_{l \in L'(m)} \abs{\Pi_{u_{l-1},u_{l}}}^{\rho}
\leq 2^{\rho}
\sum_{j=1}^{\infty} \sup_{\tau^{(m)}_{j-1} \leq t' < \tau^{(m)}_{j}} \abs{ \Pi_{t', \tau^{(m)}_{j}}}^{\rho}.
\]
Substituting this into \eqref{eq:1}, we conclude the proof of Lemma~\ref{lem:paraprod-stopping}.
\end{proof}

\begin{corollary}
\label{cor:paraprod-stopping}
Let $(\Pi_{s,t})_{s \leq t}$ be an adapted process with $\Pi_{t,t}=0$ for all $t$.
Then, for every $0 < \rho < r < \infty$ and $q \in (0,\infty]$, we have
\begin{equation}
\label{eq:Vr-stopping}
\norm{ V^{r} \Pi }_{L^{q}}
\lesssim
\sup_{\tau}
\norm[\Big]{ \Bigl( \sum_{j=1}^{\infty} \bigl( \sup_{\tau_{j-1} \leq t < t' \leq \tau_{j}} \abs{\Pi_{t, t'}} \bigr)^{\rho} \Bigr)^{1/\rho} }_{L^{q}},
\end{equation}
where the supremum is taken over all adapted partitions $\tau$.
\end{corollary}
\begin{proof}
By the monotone convergence theorem, we can restrict the times in the definition of $V^{r}$ to a finite set, and then apply Lemma~\ref{lem:paraprod-stopping}.

The term $\Pimax$ is of the form on the right-hand side of \eqref{eq:Vr-stopping} with $\tau_{1}=\infty$.
Therefore, the claim follows from the triangle inequality in $L^{q}$ (if $q\geq 1$), $q$-convexity of $L^{q}$ (if $q<1$), and H\"older's inequality.
\end{proof}

\subsection{Sewing lemma}
\label{sec:sewing}
In this section, we apply the sewing lemma to the processes $\Pi(F,g)$.

\begin{lemma}
\label{lem:sewing}
Let $F,F^{i},\tilde{F}^{i}$, $i\in\Set{i,\dotsc,i_{\max}}$, be two-parameter processes such that $F_{s,s}=0$ and \eqref{eq:delta-F-structure} holds.
Let $g_{t}$ be a one-parameter process.
Let $\rho < 1$ and $1/\rho = 1/p_{i,0} + 1/p_{i,1}$ for every $i$.
Then, we have
\begin{equation}
\label{eq:6}
V^{\rho} \Pi(F,g)
\lesssim
\sum_{i=1}^{i_{\max}} V^{p_{i,1}} F^{i} \cdot V^{p_{i,0}} \Pi (\tilde{F}^{i}, g).
\end{equation}
\end{lemma}
\begin{proof}
We will use the sewing lemma \cite[Theorem 2.5]{MR3770049} with
\[
\Xi_{s,t} := \Pi(F, g)_{s,t}.
\]
By definition \eqref{eq:paraprod} and the hypothesis $F_{s,s}=0$, we have $\Xi_{j,j+1}=0$, so that
\[
\Pi(F, g)_{s,t} = \Xi_{s,t} - \sum_{j=s}^{t-1} \Xi_{j,j+1}.
\]
Moreover, from Chen's relation \eqref{eq:Pi-Chen-relation}, we obtain
\[
(\delta \Xi)_{s,t,u}
=
\sum_{t \leq j < u} (\delta_{s} F_{t,j}) \delta g_{j,j+1}
=
\sum_{i=1}^{i_{\max}} F^{i}_{s,t} \Pi (\tilde{F}^{i}, g)_{t,u}.
\]
We may assume that none of the summands in \eqref{eq:delta-F-structure} vanish identically and that all norms on the right-hand side of \eqref{eq:6} are finite.
In this case, the functions
\begin{align*}
\omega_{i}(s,t) &:=
\sup_{l_{\max}, s \leq u_{0} \leq \dotsb \leq u_{l_{\max}} \leq t}
\sum_{l=1}^{l_{\max}} \abs{ F^{i}_{u_{l-1},u_l} }^{p_{i,1}},
\\
\tilde{\omega}_{i}(t,u) &:=
\sup_{l_{\max}, s \leq u_{0} \leq \dotsb \leq u_{l_{\max}} \leq t}
\sum_{l=1}^{l_{\max}} \abs{ \Pi (\tilde{F}^{i}, g)_{u_{l-1},u_l} }^{p_{i,0}}
\end{align*}
are controls (i.e., superadditive functions mapping ordered pairs of times to positive real numbers), and we have
\[
\abs{(\delta \Xi)_{s,t,u}}
\leq
\sum_{i=1}^{i_{\max}} \omega_{i}(s,t)^{1/p_{i,1}} \tilde{\omega}_{i}(t,u)^{1/p_{i,0}},
\]
which is exactly the hypothesis of the sewing lemma \cite[Theorem 2.5]{MR3770049}.
The sewing lemma implies
\[
\abs{\Pi(F, g)_{s,t}}
=
\abs{\Xi_{s,t} - \sum_{j=s}^{t-1} \Xi_{j,j+1}}
\lesssim
\sum_{i=1}^{i_{\max}} \omega_{i}(s,t)^{1/p_{i,1}} \tilde{\omega}_{i}(s,t)^{1/p_{i,0}}.
\]
This implies the claim \eqref{eq:6}.
\end{proof}

\subsection{Discrete sums corresponding to It\^{o} integrals}
\label{sec:var-discrete-Ito}

Here, we combine the results in Sections~\ref{sec:var-stopping} and~\ref{sec:sewing} into a statement that holds for arbitrary variational exponents $r$.

\begin{corollary} \label{cor:main}
Let $0 < q_{1} \leq \infty$, $1\leq q_{0} < \infty$, and $0 < r,p_{1} \leq \infty$.
Let $1/q = 1/q_{0}+1/q_{1}$ and assume $1/r < 1/p_{1} + 1/2$.
Let $(F_{s,t})$ be an adapted process such that \eqref{eq:delta-F-structure} holds, $g$ a martingale, and $(\tau_{k})$ an adapted partition.
Assume that $1/r < 1/p_{i,0} + 1/p_{i,1}$ for every $i$.
Then, we have
\begin{equation}
\label{eq:Vr-Pi-discrete}
\norm[\big]{ V^{r} \Pi(F,g) }_{q}
\lesssim
\norm[\big]{ V^{p_{1}} F }_{q_{1}}
\norm{ Sg }_{q_{0}}
+
\sum_{i=1}^{i_{\max}} \norm[\big]{ V^{p_{i,1}} F^{i} \cdot V^{p_{i,0}} \Pi(\tilde{F}^{i},g) }_{q}.
\end{equation}
\end{corollary}
\begin{proof}
Define $\rho$ by $1/\rho = 1/p_{1} + 1/2$.
Consider first the case $\rho \geq 1$.
By Corollary~\ref{cor:paraprod-stopping} with $1 \leq \rho < r < \infty$, it suffices to estimate the terms
\[
\norm{ \ell^{\rho}_{j} \sup_{\tau_{j-1} \leq t < t' \leq \tau_{j}} \abs{\Pi(F,g)_{t,t'}} }_{L^{q}(\Omega)},
\]
uniformly in the adapted partition $\tau$.
They are bounded by Corollary~\ref{cor:vv-pprod-delta-F}.

Consider now the case $\rho < 1$.
Note that $p_{1} < \infty$, so that $F_{s,s}=0$ for all $s$ by definition \eqref{eq:def:Vp}.
The claim now follows from Lemma~\ref{lem:sewing}, even without the last term in \eqref{eq:Vr-Pi-discrete}.
\end{proof}

\subsection{Discrete sums arising in It\^{o} integration of branched rough paths}
\label{sec:var-discrete-branched}
One can obtain estimates for $\Pi(F,g)$, with $F$ being a component of a branched rough path, by iterating Corollary~\ref{cor:main}.
However, this would involve potentially applying Corollary~\ref{cor:paraprod-stopping} at every step of the iteration, resulting in unnecessary losses.
It is in fact more efficient to iterate vector-valued, rather than variational, estimates, which we have already done in Theorem~\ref{thm:vv-pprod-branched}.
Here, we indicate the consequences that Theorem~\ref{thm:vv-pprod-branched} has for variation norm estimates.

\begin{corollary}
\label{cor:var-pprod-branched}
Let $q \in (0,\infty)$, $q_{0} \in [1,\infty)$, and, for each tree $\frakt \in \mathcal{T_{L}}$, let $q_{\frakt} \in (0,\infty]$.
Let $\rho \in (0,\infty)$ and, for each tree $\frakt \in \mathcal{T_{L}}$, let $r_{\frakt} \in [1,\infty]$.
Let $\frakf \in \mathcal{F_{L}}$ be a forest and let $\frakF$ be the set of all forests $\frakf'$ that are the disjoint unions of arbitrary partitions of trees in $\frakf$ into subtrees.
Assume that, for each $\frakf' \in \frakF$, we have
\[
1/q = 1/q_{0} + \sum_{\frakt \in \frakf'} 1/q_{\frakt},
\quad
1/\rho = 1/2 + \sum_{\frakt \in \frakf'} 1/r_{\frakt}.
\]
Let $F$ be an adapted family of branched rough paths and $g$ a martingale.
Then, for every $r>\rho$, we have
\begin{equation}
\norm[\big]{ V^{r} \Pi(F^{\frakf},g) }_{q}
\lesssim
\sum_{\frakf' \in \frakF}
\Bigl( \prod_{\frakt \in \frakf'} \norm[\big]{ V^{r_{\frakt}} F^{\frakt} }_{q_{\frakt}} \Bigr)
\norm{ Sg }_{q_{0}}.
\end{equation}
\end{corollary}
\begin{proof}
Consider first the case $\rho \geq 1$.
By Corollary~\ref{cor:paraprod-stopping}, it suffices to estimate
\begin{equation}
\label{eq:21}
\norm{ \ell^{\rho}_{k} \sup_{\tau_{k-1} \leq t < t' \leq \tau_{k}} \abs{\Pi(F^{\frakf},g)_{t, t'}} }_{q},
\end{equation}
uniformly in the adapted partition $\tau$.
This is the content of Theorem~\ref{thm:vv-pprod-branched}.

In the case $\rho < 1$, we may also assume $r<1$, and we induct on $\deg \frakf$.
Since $\rho < 1$, the forest $\frakf$ cannot be empty, and it follows from the definition of a branched rough path that $F^{\frakf}_{s,s}=0$.
Suppose that the result is known for all forests with smaller degree.

By Lemma~\ref{lem:sewing}, the generalized Chen relation \eqref{eq:Chen0-branched}, and Hölder's inequality, we obtain the pointwise estimate
\begin{equation}
\label{eq:11}
V^{r} \Pi(F^{\frakf},g)
\lesssim
\sum_{(\frakb,\frakr) \in \Cut(\frakf), \frakb\neq 0} V^{r(\frakb)} F^{\frakb} \cdot V^{\tilde{r}(\frakr)} \Pi(F^{\frakr},g)_{t,u},
\end{equation}
where for every cut $(\frakb,\frakr)$ of $\frakf$ we set
\[
1/r(\frakb) = \sum_{\frakt \in \frakb} 1/r_{\frakt},
\quad
1/r = 1/r(\frakb) + 1/\tilde{r}(\frakr).
\]
By Hölder's inequality, we estimate the $L^{q}$ norm of the $(\frakb,\frakr)$-summand on the right-hand side of \eqref{eq:11} by
\begin{equation}
\label{eq:17}
\norm{ V^{r(\frakb)} F^{\frakb}}_{q(\frakb)}
\norm{ V^{\tilde{r}(\frakr)} \Pi(F^{\frakr},g)_{t,u} }_{\tilde{q}(\frakr)},
\end{equation}
where
\[
1/q(\frakb) = \sum_{\frakt \in \frakb} 1/q_{\frakt},
\quad
1/\tilde{q}(\frakr) = 1/q-1/q(\frakb) = 1/q_{0} + \sum_{\frakt \in \frakr} 1/q_{\frakt}.
\]
In the first term in \eqref{eq:17}, we use $F^{\frakb} = \prod_{\frakt\in\frakb} F^{\frakt}$, so that
\[
\norm{ V^{r(\frakb)} F^{\frakb}}_{q(\frakb)}
\leq
\prod_{\frakt\in\frakb} \norm{ V^{r_{\frakt}} F^{\frakt}}_{q_{\frakt}}.
\]
In the second term in \eqref{eq:17}, we can use the inductive hypothesis because $\deg \frakr < \deg \frakf$.
\end{proof}

\section{Estimates for the It\^o integral}

\subsection{It\^o integral}
\label{sec:Ito}
\begin{proof}[Proof of Theorem~\ref{thm:main}, part~\ref{it:m1}]
Since $\hPi^{\pi}(F,g)_{t,t'}$ is \cadlag{} in both $t$ and $t'$, we have
\[
V^{r} \hPi^{\pi}(F,g) =
\lim_{n\to\infty} \sup_{l_{\max}, u_{0} < \dotsb < u_{l_{\max}}, u_{l} \in \pi^{(n)}}
\Bigl( \sum_{l=1}^{l_{\max}} \abs{ \hPi^{\pi}(F,g)_{u_{l-1},u_l} }^{r} \Bigr)^{1/r},
\]
where $\pi^{(n)} = \pi \cup 2^{-n}\N$.
By the monotone convergence theorem, it suffices to consider a fixed $\pi^{(n)}$, as long as the bound does not depend on $n$.

For any adapted partitions $\pi \subseteq \pi'$, we have
\begin{equation}
\label{eq:Pi-pi'-F-pi}
\begin{split}
\Pi^{\pi}(F,g)_{t,t'}
&=
\sum_{k : \floor{t,\pi} \leq \pi_{k} < t'} F_{\floor{t,\pi},\pi_{k}} (g_{\pi_{k+1} \bmin t'} - g_{\pi_{k} \bmax t})
\\ &=
\sum_{k : \floor{t,\pi} \leq \pi_{k} < t'} F_{\floor{t,\pi},\pi_{k}}
\sum_{l : \pi_{k} \bmax \floor{t,\pi'} \leq \pi'_{l} < \pi_{k+1} \bmin t'} (g_{\pi'_{l+1} \bmin t'} - g_{\pi'_{l} \bmax t})
\\ &=
\sum_{k : \floor{t,\pi} \leq \pi_{k} < t'}
\sum_{l : \pi_{k} \bmax \floor{t,\pi'} \leq \pi'_{l} < \pi_{k+1} \bmin t'}
F_{\floor{t,\pi},\floor{\pi'_{l},\pi}}
(g_{\pi'_{l+1} \bmin t'} - g_{\pi'_{l} \bmax t})
\\ &=
\sum_{l : \floor{t,\pi'} \leq \pi'_{l} < t'}
F^{(\pi)}_{\floor{t,\pi'},\pi'_{l}}
(g_{\pi'_{l+1} \bmin t'} - g_{\pi'_{l} \bmax t})
\\ &=
\Pi^{\pi'}(F^{(\pi)},g)_{t,t'},
\end{split}
\end{equation}
where $F^{(\pi)}$ is given by \eqref{eq:F-discrete}.
Define discrete time processes $F^{(\pi)}_{\pi'},g_{\pi'}$ by
\[
(F^{(\pi)}_{\pi'})_{j,j'} = F^{(\pi)}_{\pi'_{j},\pi'_{j'}},
\quad
(g_{\pi'})_{j} = g_{\pi'_{j}}.
\]
Then, we have
\begin{align*}
\Pi^{\pi}(F,g)_{\pi'_{j},\pi'_{j'}}
&=
\Pi^{\pi'}(F^{(\pi)},g)_{\pi'_{j},\pi'_{j'}}
\\ &=
\sum_{l : \floor{\pi'_{j},\pi'} \leq \pi'_{l} < \pi'_{j'}}
F^{(\pi)}_{\floor{\pi'_{j},\pi'},\pi'_{l}}
(g_{\pi'_{l+1} \bmin \pi'_{j'}} - g_{\pi'_{l} \bmax \pi'_{j}})
\\ &=
\sum_{l : j \leq l < j'}
F^{(\pi)}_{\pi'_{j},\pi'_{l}}
(g_{\pi'_{l+1}} - g_{\pi'_{l}})
\\ &=
\Pi(F^{(\pi)}_{\pi'},g_{\pi'})_{j,j'},
\end{align*}
where the last line is the discrete time paraproduct defined in \eqref{eq:paraprod}.
Therefore, the required bound follows from Corollary~\ref{cor:main}, since it follows from \eqref{eq:7} that
\begin{equation}
\label{eq:F-pi-diff}
F^{(\pi)}_{s,u}-F^{(\pi)}_{t,u} = \sum_{i=1}^{i_{\max}} F^{i,(\pi)}_{s,t} \tilde{F}^{i,(\pi)}_{t,u}.
\qedhere
\end{equation}
\end{proof}

\begin{lemma}
\label{lem:F-pi-converges-to-F}
Let $F,F^{i},\tilde{F}^{i}$ be \cadlag{} adapted processes such that \eqref{eq:7} holds and $F^{i}_{t,t}=0$ for all $i,t$.
Suppose that $V^{p_{1}}F \in L^{q_{1}}$ for some $p_{1},q_{1} \in (0,\infty]$ and $V^{\infty}\tilde{F}^{i} \in L^{q_{1}}$ for every $i$.
Then, for every $\tilde{p}_{1} \in (p_{1},\infty) \cup \Set{\infty}$, we have
\[
\lim_{\pi} \norm{V^{\tilde{p}_{1}} (F-F^{(\pi)})}_{L^{q_{1}}} = 0.
\]
\end{lemma}
\begin{proof}
We have $V^{p_{1}}F^{(\pi)} \leq V^{p_{1}}F$ and, by Hölder's inequality,
\[
V^{\tilde{p}_{1}}(F-F^{(\pi)})
\leq
V^{p_{1}}(F-F^{(\pi)})^{1-\theta}
V^{\infty}(F-F^{(\pi)})^{\theta}
\]
with some $\theta \in (0,1]$, so it suffices to consider $\tilde{p}_{1}=\infty$.

Let $\epsilon > 0$ and define a sequence of stopping times recursively, starting with $\pi_{0} := 0$, by
\begin{equation}
\label{eq:22}
\begin{split}
\pi_{j+1} &:= \pi_{j+1}^{(F)} \bmin \min_{i \in \Set{1,\dotsc,i_{\max}}} \pi_{j+1}^{(i)},
\\
\pi_{j+1}^{(F)} &:= \inf \calT_{j+1}^{(F)},
\quad \calT_{j+1}^{(F)} = \Set[\Big]{ t>\pi_{j} \given \sup_{s \leq \pi_{j}} \abs{F_{s,t}-F_{s,\pi_{j}}} \geq \epsilon },
\\
\pi_{j+1}^{(i)} &:= \inf \calT_{j+1}^{(i)},
\quad \calT_{j+1}^{(i)} = \Set[\Big]{ t>\pi_{j} \given \sup_{\pi_{j} \leq s \leq t} \abs{F^{i}_{s,t}} \geq \epsilon}.
\end{split}
\end{equation}
We now verify that this indeed defines an adapted partition.
In order to verify that $\pi_{j+1}$ is a stopping time, we show that $\pi_{j+1}^{(F)}$ and $\pi_{j+1}^{(i)}$ are hitting times, that is, the above infima are either $+\infty$ or minima.
Suppose first $T:=\pi_{j+1}^{(F)} < \infty$.
Then, there exist sequences $(s_{k})$, $(t_{k})$ such that $s_{k} \leq \pi_{j}$ and $t_{k}\geq T$ with $t_{k} \to T$ and $\abs{F_{s_{k},t_{k}}-F_{s_{k},\pi_{j}}} \geq \epsilon-1/k$.
Passing to a subsequence, we may assume that the sequence $(s_{k})$ is monotonic and converges to some $s \leq \pi_{j}$.
By the \cadlag{} hypothesis, this implies
$\abs{F_{s\pm,T}-F_{s\pm,\pi_{j}}} \geq \epsilon$,
where the sign $\pm$ depends on whether $(s_{k})$ is monotonically increasing or decreasing.
Using the \cadlag{} hypothesis again, this implies $T \in \calT_{j+1}^{(F)}$.
Hence, $\pi_{j+1}^{(F)}$ is a hitting time.

Suppose next $T:=\pi_{j+1}^{(i)} < \infty$.
Then, there exist sequences $(s_{k})$, $(t_{k})$ such that $\pi_{j} \leq s_{k} \leq t_{k}$ and $t_{k}\geq T$ with $t_{k} \to T$ and $\abs{F^{i}_{s_{k},t_{k}}} \geq \epsilon - 1/k$.
Passing to a subsequence, we may assume that the sequence $(s_{k})$ is monotonic and converges to some $s \in [\pi_{j},T]$.
By the \cadlag{} hypothesis, this implies
$\abs{F^{i}_{s\pm,T}} \geq \epsilon$,
where the sign $\pm$ depends on whether $(s_{k})$ is monotonically increasing or decreasing.
Using the \cadlag{} hypothesis again, this implies $T \in \calT_{j+1}^{(i)}$ (here we use $F^{i}_{\pi_{j},\pi_{j}}=0$ to conclude $T>\pi_{j}$).
Hence, $\pi_{j+1}^{(i)}$ is a hitting time.

The above discussion shows in particular that $\pi_{j+1}>\pi_{j}$.
To see that $\lim_{j\to\infty}\pi_{j}=\infty$, suppose for a contradiction that $\lim_{j\to\infty} \pi_{j} = T < \infty$.
Let $t_{j} := \pi_{j+1}$.
Then, either there exists a subsequence $\calJ\subseteq\N$ with $t_{j} = \pi_{j+1}^{(F)} \to T$ for $j\in\calJ$, or an $i\in\Set{1,\dotsc,i_{\max}}$ and a subsequence $\calJ\subseteq\N$ with $t_{j} = \pi_{j+1}^{(i)} \to T$ for $j\in\calJ$.

Consider first the case $t_{j} = \pi_{j+1}^{(F)}$ for $j\in\calJ$.
Then, for $j\in\calJ$, there exist $s_{j} \leq \pi_{j}$ such that
$\abs{F_{s_{j},t_{j}} - F_{s_{j},\pi_{j}}} \geq \epsilon - 1/j$.
Passing to a subsequence, we may assume that $(s_{j})_{j\in\calJ}$ is monotonic and converges to some $s \leq T$.
By the \cadlag{} hypothesis, this implies
$\abs{F_{s\pm,T-} - F_{s\pm,T-}} \geq \epsilon$,
where the sign $\pm$ depends on whether $(s_{k})$ is monotonically increasing or decreasing, a contradiction.

Consider next the case $t_{j} = \pi_{j+1}^{(i)}$ for $j\in\calJ$.
Then, for $j\in\calJ$, there exist $s_{j} \in [\pi_{j},t_{j})$ such that $\abs{F^{i}_{s_{j},t_{j}}} \geq \epsilon - 1/j$.
Since the sequence $(\pi_{j})_{j\in\N}$ is strictly monotonically increasing and converges to $T$, we have $s_{j} \to T-$ and $t_{j} \to T-$ for $j\in\calJ$.
By the \cadlag{} hypothesis, this implies $\abs{F^{(i)}_{T-,T-}} \geq \epsilon$.
On the other hand, by the hypothesis $F^{i}_{t,t}=0$ and the \cadlag{} hypothesis, we have $F^{i}_{T-,T-}=0$, a contradiction.

Thus we have shown that $\pi$ is indeed an adapted partition.
By \eqref{eq:7}, for any adapted partition $\pi' \supseteq \pi$ and $s\leq t$, we have
\begin{align*}
\abs{ F_{s,t} - F^{(\pi')}_{s,t} }
&\leq
\abs{ F_{s,t} - F_{\floor{s,\pi'},t} }
+
\abs{ F_{\floor{s,\pi'},t} - F_{\floor{s,\pi'},\floor{t,\pi'}} }
\\ &\leq
\sum_{i=1}^{i_{\max}} \abs{F^{i}_{\floor{s,\pi'},s} } \abs{\tilde{F}^{i}_{s,t}}
+
\abs{ F_{\floor{s,\pi'},t} - F_{\floor{s,\pi'},\floor{t,\pi}} }
+
\abs{ F_{\floor{s,\pi'},\floor{t,\pi'}} - F_{\floor{s,\pi'},\floor{t,\pi}} }
\\ &\leq
\sum_{i=1}^{i_{\max}} \epsilon \cdot V^{\infty} \tilde{F}^{i}
+
2\epsilon.
\qedhere
\end{align*}
\end{proof}

\begin{remark}
Some structural condition (such as \eqref{eq:7}) on the two-parameter process $F$ is necessary in Lemma~\ref{lem:F-pi-converges-to-F}.
Even if $F$ is deterministic, continuous, and vanishes on the diagonal, $F^{(\pi)}$ does not necessarily converge to $F$ uniformly.
To see this, let $\phi : \R \to [0,1]$ be a smooth function such that $\phi=0$ on $(-\infty,0]$ and $\phi=1$ on $[1,\infty)$.
Let $F(s,t) := \phi(st) \phi(t-s)$.
Then, for any partition $\pi$ with $\pi_{j}\to\infty$, for $s \in (0,\pi_{1})$, we have
\[
F(s,\pi_{j}) - F^{(\pi)}(s,\pi_{j})
=
F(s,\pi_{j})-F(0,\pi_{j})\to 1
\text{ as } j\to\infty.
\]
In the above example, $F$ is not uniformly continuous.
Convergence can also fail for uniformly continuous in time processes if their samples are not equicontinuous.
To see this, let $\Omega=(0,1)$ with the Lebesgue measure, $\calF_{t}$ the trivial $\sigma$-algebra for $t<1/3$ and the Lebesgue $\sigma$-algebra for $t\geq 1/3$.
Let $F(s,t) := \phi(2s \phi(3t-1)/\omega) \phi(3(t-s))$, where $\omega\in\Omega$ and $0\leq s\leq t\leq 1$.
For any $0 \leq s \leq t\leq 1/3$, we have $F(s,t)=0$, so this process is indeed measurable with respect to the given filtration.
For any adapted partition $\pi$, there is an $0<s_{0}\leq 1/3$ such that $s_{0}\leq\pi_{1}(\omega)$ for a.e.\ $\omega\in\Omega$.
Let $0<s<s_{0}$ and $t\geq 2/3$.
Then
\[
F(s,t) - F(0,t) = \phi(2s/\omega) - \phi(0) = 1
\quad\text{for } \omega < 2s,
\]
so that $\norm{V^{\infty}(F-F^{(\pi)})}_{L^{\infty}} = 1$.
\end{remark}

\begin{proof}[Proof of Theorem~\ref{thm:main}, part~\ref{it:m2}]
By the Cauchy criterion for net convergence, the existence of the limit \eqref{eq:paraprod-limit} will follow if we can show that
\begin{equation}
\label{eq:14}
\lim_{\pi} \sup_{\pi' \supseteq \pi} \norm[\big]{ V^{r} (\hPi^{\pi}(F,g)-\hPi^{\pi'}(F,g)) }_{L^{q}} = 0.
\end{equation}
To this end, we use that, by \eqref{eq:Pi-pi'-F-pi}, we have
\[
\hPi^{\pi}(F,g)-\hPi^{\pi'}(F,g)
=
\hPi^{\pi'}(F^{(\pi)}-F^{(\pi')},g).
\]
It follows from \eqref{eq:F-pi-diff} that
\begin{align*}
\MoveEqLeft
(F^{(\pi)}_{s,u}-F^{(\pi')}_{s,u})-(F^{(\pi)}_{t,u}-F^{(\pi')}_{t,u})
\\ &=
\sum_{i=1}^{i_{\max}} F^{i,(\pi)}_{s,t} \tilde{F}^{i,(\pi)}_{t,u}
-
\sum_{i=1}^{i_{\max}} F^{i,(\pi')}_{s,t} \tilde{F}^{i,(\pi')}_{t,u}
\\ &=
\sum_{i=1}^{i_{\max}} (F^{i,(\pi)}_{s,t} - F^{i,(\pi')}_{s,t}) \tilde{F}^{i,(\pi)}_{t,u}
+
\sum_{i=1}^{i_{\max}} F^{i,(\pi')}_{s,t} (\tilde{F}^{i,(\pi)}_{t,u} - \tilde{F}^{i,(\pi')}_{t,u}).
\end{align*}
Let $\tilde{p}_{1} \in (p_{1},\infty] \cup \Set{\infty}$ be such that $1/r < 1/\tilde{p}_{1}+1/2$.
By Part~\ref{it:m1} of Theorem~\ref{thm:main} with $F$ replaced by $F^{(\pi)}-F^{(\pi')}$, we obtain
\begin{align*}
\MoveEqLeft
\norm[\big]{ \hPi^{\pi'}(F^{(\pi)}-F^{(\pi')},g) }_{L^{q}}
\\ &\lesssim
\norm{ V^{\tilde{p}_1} (F^{(\pi)}-F^{(\pi')})^{(\pi')} }_{L^{q_1}} \norm{ V^{\infty}g }_{L^{q_{0}}}
\\ &+
\sum_{i=1}^{i_{\max}} \norm{ V^{p_{i,1}} (F^{i,(\pi)} - F^{i,(\pi')})^{(\pi')} \cdot V^{p_{i,0}} \Pi^{\pi'}(\tilde{F}^{i,(\pi)},g) }_{L^{q}}
\\ &+
\sum_{i=1}^{i_{\max}} \norm{ V^{p_{i,1}} (F^{i,(\pi')})^{(\pi')} \cdot V^{p_{i,0}} \Pi^{\pi'}(\tilde{F}^{i,(\pi)}-\tilde{F}^{i,(\pi')},g) }_{L^{q}}
\end{align*}
The first line converges to $0$ by Lemma~\ref{lem:F-pi-converges-to-F}.
The second and third line converge to $0$ by the hypotheses \eqref{eq:hypothesis-Fi-discretization} and \eqref{eq:hypothesis-Pi-Fi-discretization}, respectively.

In order to show the Chen relation \eqref{eq:Chen}, we first show that the corresponding relation holds pointwise for the discretized paraproducts $\Pi^{\pi}$.
Indeed, by definition \eqref{eq:hPi}, we have
\begin{equation}
\label{eq:delta-Pi-pi}
\hPi^{\pi}(F,g)_{t,t''}
- \hPi^{\pi}(F,g)_{t,t'} - \hPi^{\pi}(F,g)_{t',t''}
\end{equation}
\begin{multline*}
= \sum_{\floor{t,\pi} \leq \pi_{j} < t''} F_{\floor{t,\pi},\pi_{j}} (g_{\pi_{j+1} \bmin t''} - g_{\pi_{j} \bmax t})
- \sum_{\floor{t,\pi} \leq \pi_{j} < t'} F_{\floor{t,\pi},\pi_{j}} (g_{\pi_{j+1} \bmin t'} - g_{\pi_{j} \bmax t})
\\- \sum_{\floor{t',\pi} \leq \pi_{j} < t''} F_{\floor{t',\pi},\pi_{j}} (g_{\pi_{j+1} \bmin t''} - g_{\pi_{j} \bmax t'})
\end{multline*}
\begin{multline*}
= \sum_{\floor{t,\pi} \leq \pi_{j} < t'} F_{\floor{t,\pi},\pi_{j}} (g_{\pi_{j+1} \bmin t''} - g_{\pi_{j+1} \bmin t'})
+
\sum_{t' \leq \pi_{j} < t''} (F_{\floor{t,\pi},\pi_{j}}-F_{\floor{t',\pi},\pi_{j}}) (g_{\pi_{j+1} \bmin t''} - g_{\pi_{j}})
\\- \sum_{\pi_{j} < t' < \pi_{j+1}}  F_{\floor{t',\pi},\pi_{j}} (g_{\pi_{j+1} \bmin t''} - g_{\pi_{j} \bmax t'}).
\end{multline*}
All summands except possibly the one with $\pi_{j}<t'<\pi_{j+1}$ in the first sum vanish, and it follows that
\begin{align*}
\eqref{eq:delta-Pi-pi}
&=
\sum_{\floor{t',\pi} \leq \pi_{j} < t''} (F_{\floor{t,\pi},\pi_{j}}-F_{\floor{t',\pi},\pi_{j}}) (g_{\pi_{j+1} \bmin t''} - g_{\pi_{j} \bmax t'})
\\ \text{(by \eqref{eq:7})} &=
\sum_{\floor{t',\pi} \leq \pi_{j} < t''} \sum_{i=1}^{i_{\max}}
F^{i}_{\floor{t,\pi},\floor{t',\pi}} \tilde{F}^{i}_{\floor{t',\pi},\pi_{j}}
(g_{\pi_{j+1} \bmin t''} - g_{\pi_{j} \bmax t'})
\\ \text{(by \eqref{eq:hPi})} &=
\sum_{i=1}^{i_{\max}} F^{i,(\pi)}_{t,t'} \Pi^{\pi}(\tilde{F}^{i},g)_{t',t''}.
\end{align*}
By the hypotheses \eqref{eq:hypothesis-Fi-discretization} and \eqref{eq:hypothesis-Pi-Fi-discretization} and the already known conclusion \eqref{eq:paraprod-limit}, we can take net limits along adapted partitions $\pi$ on both sides of this equality.
This yields \eqref{eq:Chen}.
\end{proof}

\subsection{Mesh convergence}
\label{sec:Ito-mesh-convergence}
Theorem~\ref{thm:main} can be used to recover the classical results about uniform convergence of probability of discrete approximations to the It\^o integral.
We begin with the simpler case of continuous integrands.
\begin{corollary}
\label{cor:Ito-continuous-mesh-convergence}
In the situation of part~\ref{it:m2} of Theorem~\ref{thm:main}, suppose that $F=\delta f$, $q_{0},q_{1}<\infty$, and the process $f$ has a.s.\ continuous paths.
Then convergence in \eqref{eq:paraprod-limit} holds in the stronger sense that
\begin{equation}
\label{eq:paraprod-mesh-limit}
\hPi(\delta f,g) = \lim_{\mesh(\pi)\to 0} \hPi^{\pi}(\delta f,g)
\end{equation}
in $L^{q}(V^{p})$, where $\pi$ ranges over adapted partitions.
\end{corollary}
\begin{proof}
In view of the uniform bound in part~\ref{it:m1} of Theorem~\ref{thm:main}, it suffices to consider a bounded time interval.
On such an interval, the paths of $f$ are uniformly continuous.
Therefore, $F^{(\pi)} \to F$ uniformly as $\mesh(\pi)\to 0$.
Since $F^{(\pi)}$ are also uniformly bounded in $L^{q_{1}}(V^{p_{1}})$, we have $F^{(\pi)} \to F$ in $L^{q_{1}}(V^{\tilde{p}_{1}})$ for any $\tilde{p}_{1} \in (p_{1},\infty) \cup \Set{\infty}$.
We can choose $\tilde{p}_{1}$ such that $1/r < 1/\tilde{p}_{1} + 1/2$.
It remains to apply the estimate \eqref{eq:YoungBDG-limit} with $p_{1}$ replaced by $\tilde{p}_{1}$ to
\[
\hPi(F,g) - \hPi^{\pi}(F,g) = \hPi(F-F^{(\pi)},g).
\qedhere
\]
\end{proof}

Next, we recover the convergence result for discrete approximations to the It\^o integral in the presence of jumps.
Recall that a \emph{local martingale} is a stochastic process $g=(g_{t})_{t\in\R_{\geq 0}}$ such that there exists an adapted partition $\tau$ such that, for every $k\in\N$, we have
\begin{enumerate}
\item for every $t\in\R_{\geq 0}$, $g_{t \bmin \tau_{k}} \in L^{1}(\Omega)$, and
\item the stopped process $g^{\tau_{k}} := (g_{t \bmin \tau_{k}})_{t}$ is a martingale.
\end{enumerate}
Any adapted partition as above is called a \emph{localizing sequence} for $g$.

\begin{lemma}
\label{lem:L1-localizing-sequence}
Let $g$ be a \cadlag{} local martingale.
Then there exists a localizing sequence $(\tau_{k})$ for $g$ such that, for every $k$, we have $g^{\tau_{k}} \in L^{1}(V^{\infty})$.
\end{lemma}
\begin{proof}
Let $(\tilde\tau_{k})$ be a localizing sequence for $g$.
Define
\[
\tau_{k} := \tilde\tau_{k} \bmin k \bmin \inf\Set{t \given \abs{g_{t}} \geq k}.
\]
Then
\[
V^{\infty} g^{\tau_{k}} \leq k + \abs{g_{\tau_{k}}}.
\]
The first summand is in $L^{\infty} \subset L^{1}$.
For the second summand, we have
\[
\E \abs{g_{\tau_{k}}}
=
\E \abs{g^{\tilde\tau_{k}}_{\tau_{k}}}
\leq
\E \abs{g^{\tilde\tau_{k}}_{k}}
<
\infty.
\qedhere
\]
\end{proof}

Now, we can recover the existence of It\^o integrals.
\begin{corollary}
\label{cor:Ito-mesh-convergence}
Let $f$ be a \cadlag{} adapted process and $g$ a \cadlag{} local martingale.
Then, there exists the limit
\begin{equation}
\label{eq:Ito-mesh-convergence}
\Pi(f,g)_{0,\cdot} = \ucplim_{\mesh(\pi) \to 0} \Pi^{\pi}(f,g)_{0,\cdot}.
\end{equation}
\end{corollary}

Note that the two-parameter supremum
\[
\sup_{0 \leq t \leq t' \leq T} \abs{\hPi^{\pi}(f,g)_{t,t'}-\hPi(f,g)_{t,t'}}
\]
does not converge to $0$ if $f$ has jumps.
Indeed by Chen's relation, it is bounded below by a multiple of
\[
\sup_{0 \leq t \leq T}
\abs{ \delta(f-f^{(\pi)})_{0,t} \delta g_{t,T} }
=
\sup_{0 \leq t \leq T}
\abs{ (f_{t}-f_{\floor{t,\pi}}) \delta g_{t,T} },
\]
and the difference $(f_{t}-f_{\floor{t,\pi}})$ does not converge to $0$ if $f$ has jumps.

\begin{proof}[Proof of Corollary~\ref{cor:Ito-mesh-convergence}]
We may assume without loss of generality that $f_{0}=0$ and $g_{0}=0$.
Let $(\tilde\tau_{k})$ be a localizing sequence for $g$ given by Lemma~\ref{lem:L1-localizing-sequence}.
Then
\[
\tau_{k} := \tilde\tau_{k} \bmin \inf \Set{ t \given \abs{f_{t}}>k}
\]
is also a localizing sequence.
Fix $T>0$ and $\epsilon>0$.
For a sufficiently large $k$, we will have
\[
\bbP \Set{ \tau_{k} \leq T } < \epsilon/10.
\]
Replacing $g$ by $g^{\tau_{k}}$ and $f$ by $(f_{t\bmin \tau_{k}-})_{t}$, we may assume that $g \in L^{1}(V^{\infty})$ and $f\in L^{\infty}(V^{\infty})$.

By part~\ref{it:m2} of Theorem~\ref{thm:main} with $q=1$ and any $r>2$, there exists an adapted partition $\pi^{\circ}$ such that, for every adapted partition $\pi' \supseteq \pi^{\circ}$, we have
\[
\norm[\Big]{ V^{r} (\hPi^{\pi'}(f,g) - \hPi(f,g)) }_{L^{q}(\Omega)} < (\epsilon/10)^{1+1/q}.
\]
In particular, for every adapted partition $\pi' \supseteq \pi^{\circ}$, we have
\[
\P \Omega_{\pi'} < \epsilon/10,
\quad
\Omega_{\pi'} := \Set{ \sup_{0 \leq t \leq T} \abs{\hPi^{\pi'}(f,g)_{0,t} - \hPi(f,g)_{0,t}} > \epsilon/10 }.
\]
Since $V^{\infty}f$ is finite a.s., there exists $A < \infty$ such that
\[
\P \Omega_{2} < \epsilon/10,
\quad
\Omega_{2} := \Set{ \sup_{t \leq T} \abs{f_{t}} > A } < \epsilon/10.
\]
Since $\lim_{j\to\infty} \pi^{\circ}_{j} = \infty$ a.s., there exists $J \in \N$ such that
\[
\P \Omega_{3} < \epsilon/10,
\quad
\Omega_{3} := \Set{ \pi^{\circ}_{J} < T}.
\]
Since $g_{t}$ is right continuous in $t$ and measurable on $\Omega$, there exists $\delta > 0$ such that
\[
\bbP \Omega_{4} < \epsilon/10,
\quad
\Omega_{4} := \Set{ \sup_{j\leq J} \sup_{0 \leq s \leq 2\delta} \abs{g_{\pi^{\circ}_{j}+s} - g_{\pi^{\circ}_{j}}} > \epsilon/(10 A J) }
\]
and
\[
\bbP \Omega_{5} < \epsilon/10,
\quad
\Omega_{5} := \Set{ \min_{j\leq J} \abs{\pi^{\circ}_{j+1} - \pi^{\circ}_{j}} \leq \delta }.
\]
We will show that this $\delta$ works for \eqref{eq:Ito-mesh-convergence}.

Let $\pi$ be an adapted partition with $\mesh(\pi) < \delta$.
Let $\pi' := \pi \cup \pi^{\circ}$, this is another adapted partition.
For every $\pi'_{l} \in \pi^{\circ} \setminus \pi$ and $\pi'_{l} < t'$, we will use the identity
\begin{multline}
\label{eq:8}
f_{\pi'_{l-1}} (g_{\pi'_{l} \bmin t'} - g_{\pi'_{l-1}})
+
f_{\pi'_{l}} (g_{\pi'_{l+1} \bmin t'} - g_{\pi'_{l}})
\\=
f_{\pi'_{l-1}} (g_{\pi'_{l+1} \bmin t'} - g_{\pi'_{l-1}})
+
(f_{\pi'_{l}} - f_{\pi'_{l-1}}) (g_{\pi'_{l+1} \bmin t'} - g_{\pi'_{l}}).
\end{multline}
Now, if $\omega \in \Omega\setminus\Omega_{5}$, then $\pi'_{l-1},\pi'_{l+1} \not\in\pi^{\circ}$ in the situation of \eqref{eq:8}.
Therefore, the first term on the right-hand side of \eqref{eq:8} appears in $\hPi^{\pi}$.
Therefore, for every $t' \leq T$, we have
\begin{align*}
\abs{\hPi^{\pi'}(f,g)_{0,t'} - \hPi^{\pi}(f,g)_{0,t'}}
&=
\abs[\Big]{ \sum_{l : \pi'_{l} \in \pi^{\circ} \setminus \pi \text{ and } \pi'_{l} < t'}
(f_{\pi'_{l}} - f_{\pi'_{l-1}}) (g_{\pi'_{l+1} \bmin t'} - g_{\pi'_{l}}) }
\\ &\leq
\bigl( 2 \sup_{t \leq T} \abs{f_{t}} \bigr) \sum_{l : \pi'_{l} \in \pi^{\circ} \setminus \pi \text{ and } \pi'_{l} < t'}
\abs[\Big]{ g_{\pi'_{l+1} \bmin t'} - g_{\pi'_{l}} }.
\end{align*}
If $\omega \not\in \Omega_{2} \cup \Omega_{3} \cup \Omega_{4}$, then this implies
\begin{align*}
\abs{\hPi^{\pi'}(f,g)_{0,t'} - \hPi^{\pi}(f,g)_{0,t'}}
&\leq
\bigl( 2 A \bigr) \sum_{l : \pi'_{l} \in \pi^{\circ} \setminus \pi \text{ and } \pi'_{l} < t'}
\epsilon/(10AJ)
\\ &\leq
\bigl( 2 A \bigr) J \cdot \epsilon/(10AJ)
\\ &=
\epsilon/5.
\end{align*}
Hence, for every $\omega \in \Omega \setminus (\Omega_{\pi'} \cup \Omega_{2} \cup \Omega_{3} \cup \Omega_{4} \cup \Omega_{5})$, we obtain
\[
\sup_{0 \leq t' \leq T} \abs{\hPi^{\pi}(f,g)_{0,t'}-\hPi(f,g)_{0,t'}} < \epsilon.
\qedhere
\]
\end{proof}

\section{Quadratic covariation of a controlled process and a martingale}
\subsection{Variation norm estimate}
The main difficulty in defining $[Y,g]$ for an $X$-controlled process $Y$ and a martingale $g$ is to handle the contribution of the jumps of $X$.
This is done by the following result.
\begin{theorem}
\label{thm:bracket-contr+mart:bound}
Let $0 < q,q_{1} \leq \infty$, $1 \leq q_{0} < \infty$ with $1/q=1/q_{0}+1/q_{1}$.
Let $(g_{t})_{t \geq 0}$ be a \cadlag{} martingale and $(Y')_{t \geq 0}$ a \cadlag{} adapted process.
Let $I \subset (0,\infty)$ be a countable subset and $(\Delta_{t})_{t\in I}$ a (deterministic) sequence.
Consider the process
\begin{equation}
\label{eq:controlled-jump-sum}
B_{t,t'} := \sum_{j \in I \cap (t,t']} Y'_{j-} \Delta_{j} \delta g_{j-,j}.
\end{equation}
Then, for every $p_{1} \in [2,\infty]$ and $1/r < 1/2 + 1/p_{1}$, with $\MYp = \sup_{t} \abs{Y'_{t}}$,
\begin{equation}
\label{eq:bracket-contr+mart:discrete-bound}
\norm{V^{r} B }_{L^{q}}
\lesssim
\norm{ \MYp }_{L^{q_{1}}}
\bigl(\sum_{j \in I} \abs{\Delta_{j}}^{p_{1}} \bigr)^{1/p_{1}}
\norm{ \bigl( \sum_{j \in I} \abs{\delta g_{j-,j}}^{2} \bigr)^{1/2} }_{L^{q_{0}}}.
\end{equation}
\end{theorem}
\begin{proof}
We will first show that the estimate \eqref{eq:bracket-contr+mart:discrete-bound} holds for finite sets $I$.
This will immediately imply that the series \eqref{eq:controlled-jump-sum} converges unconditionally in $L^{q}(V^{r})$ and that its limit also satisfies the estimate \eqref{eq:bracket-contr+mart:discrete-bound}.

When $I$ is finite, we may assume that we are in discrete time, which corresponds to the case $I = \Set{1,\dotsc,N}$ and $Y',g$ being constant on intervals $[n,n+1)$ for $n\in\N$.
By Corollary~\ref{cor:paraprod-stopping}, it suffices to estimate the $L^{q}$ norm of
\begin{equation}
\label{eq:18}
\norm[\Big]{ \sup_{\tau_{k-1} \leq t < t' \leq \tau_{k}} \abs[\big]{ \sum_{t < j \leq t'} Y'_{j-1} \Delta_{j} d g_{j} } }_{L^{q}(\ell^{\rho}_{k})},
\end{equation}
where $(\tau_{k})_{k}$ in an increasing sequence of stopping times and $1/\rho = 1/2 + 1/p_{1}$.

Now we use that $\Delta_{j}$ is deterministic, so that $Y'_{j-1} \Delta_{j}$ is $\calF_{j-1}$-measurable.
In the case $q \geq 1$, this allows us to directly apply the vector-valued BDG inequality (Lemma~\ref{lem:vv-BDG}) to the martingales $h^{(k)}_{n} = \sum_{j\leq n} \one_{\tau_{k-1}<j\leq \tau_{k}} Y'_{j-1} \Delta_{j} dg_{j}$.

In order to treat general $q$, by the quasi-triangle inequality in $L^{q}$, we split
\begin{align}
\eqref{eq:18}
&\lesssim_{q}
\label{eq:19}
\norm[\Big]{ \sup_{\tau_{k-1} \leq t < t' \leq \tau_{k}} \abs[\big]{ \sum_{t < j \leq t'} Y'_{t} \Delta_{j} d g_{j} } }_{L^{q}(\ell^{\rho}_{k})}
\\ &+
\label{eq:20}
\norm[\Big]{ \sup_{\tau_{k-1} \leq t < t' \leq \tau_{k}} \abs[\big]{ \sum_{t < j \leq t'} (Y'_{j-1}-Y'_{t}) \Delta_{j} d g_{j} } }_{L^{q}(\ell^{\rho}_{k})}.
\end{align}
In the former term, by H\"older's inequality, the vector-valued BDG inequality (Lemma~\ref{lem:vv-BDG}) applied to the martingales $h^{(k)}_{n} = \sum_{j\leq n} \one_{\tau_{k-1}<j\leq \tau_{k}} \Delta_{j} dg_{j}$, the fact that $\rho \leq 2$, and again H\"older's inequality, we have
\begin{align*}
\eqref{eq:19}
&\leq
\norm{\MYp}_{L^{q_{1}}}
\norm[\Big]{ \sup_{\tau_{k-1} \leq t < t' \leq \tau_{k}} \abs[\big]{ \sum_{t < j \leq t'} \Delta_{j} d g_{j} } }_{L^{q_{0}}(\ell^{\rho}_{k})}
\\ &\lesssim
\norm{\MYp}_{L^{q_{1}}}
\norm[\Big]{ \bigl( \sum_{\tau_{k-1} < j \leq \tau_{k}} \abs{ \Delta_{j} d g_{j} }^{2} \bigr)^{1/2} }_{L^{q_{0}}(\ell^{\rho}_{k})}
\\ &\leq
\norm{\MYp}_{L^{q_{1}}}
\norm[\Big]{ \bigl( \sum_{j} \abs{ \Delta_{j} d g_{j} }^{\rho} \bigr)^{1/\rho} }_{L^{q_{0}}}
\\ &\leq
\norm{\MYp}_{L^{q_{1}}}
\bigl(\sum_{j} \abs{\Delta_{j}}^{p_{1}} \bigr)^{1/p_{1}}
\norm[\Big]{ \bigl( \sum_{j} \abs{d g_{j} }^{2} \bigr)^{1/2} }_{L^{q_{0}}}.
\end{align*}
In the latter term, by the vector-valued paraproduct estimate (Proposition~\ref{prop:vv-pprod} with $r_{1}=\infty$ and $r=r_{0}=\rho$), we have
\begin{align*}
\eqref{eq:20}
&\lesssim
\norm[\Big]{ \sup_{k} \sup_{\tau_{k-1} < j \leq \tau_{k}} \abs[\big]{ Y'_{j-1}-Y'_{\tau_{k-1}}} }_{L^{q_{1}}}
\norm[\Big]{ \bigl( \sum_{\tau_{k-1} < j \leq \tau_{k}} \abs{ \Delta_{j} d g_{j} }^{2} \bigr)^{1/2} }_{L^{q_{0}}(\ell^{\rho}_{k})}.
\end{align*}
This can be estimated similarly as \eqref{eq:19}.
\end{proof}

\subsection{Discretization of quadratic covariation}
\label{sec:discretization-covariation}
\begin{definition}
\label{def:discrete-bracket-2}
Let $g=(g_{t})_{t\geq 0}$ be a \cadlag{} local martingale.
For adapted \cadlag{} processes $Y,Z$ and a deterministic partition $\pi$, define
\begin{equation}
\label{eq:[Rough,Mart]-pi}
Z \bullet [Y,g]^{\pi}_{T} := \sum_{\pi_{j} < T} Z_{\pi_{j}} \delta Y_{\pi_{j},\pi_{j+1} \bmin T} \delta g_{\pi_{j},\pi_{j+1} \bmin T}.
\end{equation}
In the case $Z\equiv 1$, we omit ``$Z \bullet$'' from the notation.
\end{definition}
It is well-known that $\lim_{\pi} [Y,g]^{\pi}$ need not make sense for general processes $Y$, but does make sense e.g.\ if $Y$ is also a martingale.
In our case, the process $Y$ will be the first component of a controlled process $\rmY$.
In order to pass to a limit in \eqref{eq:[Rough,Mart]-pi}, we will need a localizing sequence for $\rmY$.
\begin{lemma}
\label{lem:controlled-localizing-sequence}
Let $1 \leq \hat{p}_{1}, p_{1} \leq \infty$.
Let $X \in V^{p_{1}}_{\loc}$ be a deterministic \cadlag{} path.
Let $\rmY = (Y,Y')$ be a \cadlag{} adapted process such that $Y \in V^{p_{1}}_{\loc}$ and $R^{\rmY,X} \in V^{\hat{p}_{1}}_{\loc}$ almost surely and $Y'_{0} \in L^{\infty}$.
Then, there exists a localizing sequence $(\tau_{k})$ such that, for every $k$, the process $\tilde{\rmY} = (\tilde{Y},\tilde{Y}')$, defined by
\[
\tilde{Y}_{t} = Y_{t \bmin \tau_{j}-},
\quad
\tilde{Y}'_{t}
=\begin{cases}
Y'_{t} & \text{if } t<\tau_{j},\\
0 & \text{if } t\geq \tau_{j},
\end{cases},
\]
satisfies $\tilde{Y} \in L^{\infty}(V^{p_{1}})$, $MY' \in L^{\infty}$, and $R^{\tilde{\rmY},\tilde{X}} \in L^{\infty}(V^{\hat{p}_{1}})$, where $\tilde{X}_{t} := X_{t\bmin k}$.
\end{lemma}
\begin{proof}
Without loss of generality, $\abs{Y'_{0}} \leq 1/2$.
Let
\[
\tau_{k} := k \bmin
\min \Set{t \given \max (V^{p_{1}}_{[0,t]}Y, \sup_{s\in [0,t]} \abs{Y'_{s}}, V^{\hat{p}_{1}}_{[0,t]} R^{\rmY,X}) \geq k }.
\]
At this point, we have used the fact that the functions $t \mapsto V^{p_{1}}_{[0,t]} Y$ and $t \mapsto V^{\hat{p}_{1}}_{[0,t]} R^{\rmY,X}$ are right continuous if $X,Y,Y'$ are \cadlag{}, so that the above minimum in fact exists.
For the former function, this is verified e.g.\ in \cite[Lemma 7.1]{MR3770049}; the argument for the latter function is similar.

Then, for any $t\leq t'$, we have
\begin{equation}
R^{\tilde{\rmY},\tilde{X}}_{t,t'} =
\begin{cases}
R^{\rmY,X}_{t,t'} & \text{if } t\leq t' < \tau_{k},\\
0 & \text{if } \tau_{k} \leq t\leq t',\\
\delta Y_{t,\tau_{k}-} - Y'_{t} \delta X_{t, t' \bmin k}, & \text{if } t < \tau_{k} \leq t'.
\end{cases}
\end{equation}
The latter case can only appear once in any $\ell^{\hat{p}_{1}}$ norm in the definition of $V^{\hat{p}_{1}}R^{\tilde{\rmY},\tilde{X}}$.
Therefore,
\[
V^{\hat{p}_{1}}R^{\tilde{\rmY},\tilde{X}}
\leq
V^{\hat{p}_{1}}_{[0,\tau_{k})} R^{\rmY,X} + 2k + k V^{\infty}_{[0,k]}X
\]
is a bounded function.
\end{proof}

\begin{theorem}
\label{thm:[Y,g]-discrete-approx}
Let $\hat{p}_{1} < 2 \leq p_{1}$ and $X \in V^{p_{1}}_{\loc}$ a deterministic \cadlag{} path.
Suppose that $\rmY=(Y,Y')$ and $Z$ are \cadlag{} adapted processes, $g$ a \cadlag{} local martingale, and $R^{\rmY,X} \in V^{\hat{p}_{1}}_{\loc}$ almost surely.
Then
\begin{equation}
\label{eq:[Rough,Mart]-limit}
Z \bullet [\rmY,g]
:=
\ucplim_{\dmesh(\pi) \to 0} Z \bullet [Y,g]^{\pi}
\end{equation}
exists, and we have
\begin{equation}
\label{eq:[Rough,Mart]}
Z \bullet [\rmY,g]_{t} =
\sum_{s \leq t} Z_{s-} \Delta X_{s} Y'_{s-} \Delta g_{s} + \sum_{s \leq t} Z_{s-} \Delta R^{\rmY}_{s} \Delta g_{s},
\end{equation}
where $\Delta g_{s} := \delta g_{s-,s}$ and $\Delta R^{\rmY}_{s} := R^{\rmY}_{s-,s}$.
Moreover, for any $1/r < 1/2 + 1/p_{1}$, we have $Z \bullet [\rmY,g] \in V^{r}_{\loc}$.
\end{theorem}

\begin{remark} \label{rem:[Rough,Mart]}
The case needed for the construction of the square bracket in Theorem~\ref{thm:dX} is $Z \equiv 1$.
General processes $Z$ are needed in the consistency result, Theorem~\ref{thm:controlled-integral=RSM-integral}.
\end{remark}

\begin{proof}
Since \eqref{eq:[Rough,Mart]-pi} and \eqref{eq:[Rough,Mart]} are linear in $\rmY$, we may assume $\abs{Y'_{0}} \leq 1$ upon replacing $\rmY$ by $\rmY/\max(1,\abs{Y'_{0}})$.
Similarly, we may assume $\abs{Z_{0}} \leq 1$.

Using the localizing sequence $\tau_{k} = \min\Set{ t \given \abs{Z_{t}} > k}$ and replacing $Z$ by $(Z_{t\bmin \tau_{k}-})_{t}$, we may assume that $Z$ is uniformly bounded.
Using the localizing sequence given by Lemma~\ref{lem:L1-localizing-sequence}, we may assume $g \in L^{1}(V^{\infty})$.
Using the localizing sequence given by Lemma~\ref{lem:controlled-localizing-sequence}, we may assume that $X \in V^{p_{1}}$, $Y \in L^{\infty}(V^{p_{1}})$, and $R^{\rmY,X} \in L^{\infty}(V^{\hat{p}_{1}})$.
Overall, we may assume
\begin{equation}
\label{eq:bracket-localizing-sequence}
g \in L^{1}V^{\infty},
\quad X\in V^{p_{1}},
\quad M Y', MZ \in L^{\infty},
\quad R^{\rmY,X} \in L^{\infty}(V^{\hat{p_{1}}}).
\end{equation}
Assuming \eqref{eq:bracket-localizing-sequence}, the first sum in \eqref{eq:[Rough,Mart]} now makes sense by Theorem~\ref{thm:bracket-contr+mart:bound} and is in $V^{r}_{\loc}$ for any $1/r < 1/2+1/p_{1}$.
The second sum in \eqref{eq:[Rough,Mart]} almost surely converges absolutely for every $t$, and in particular defines a process with almost surely $V^{1}_{\loc}$ paths.

Now, still assuming \eqref{eq:bracket-localizing-sequence}, we will show that the limit \eqref{eq:[Rough,Mart]-limit} exists and coincides with \eqref{eq:[Rough,Mart]}.

Fix $T>0$.
Let $A \geq 1$ be such that $\sup_{t \leq T} \abs{X_{t}} < A$ and the set
\begin{align*}
\Omega_{1} &:= \Set[\Big]{ \sup_{t \leq T} (\abs{Y_{t}} \bmax \abs{Y'_{t}} \bmax \abs{g_{t}} \bmax \abs{Z_{t}}) < A }
\end{align*}
has probability $\geq 1-\epsilon$.

Let $J_{X} := \Set{ s \given \abs{\Delta X_{s}} > \epsilon/(2A) }$ and $J_{Y}(\omega) := \Set{ s \given \abs{\Delta Y_{s}} > \epsilon/2 }$.
Let $N < \infty$ be such that $\abs{J_{X}} \leq N$ and
\[
\Omega_{4} := \Set{ \abs{J_{Y}} < N }
\]
has probability $\geq 1 - \epsilon$.

Let $\delta$ be such that
\begin{align*}
&\sup_{t \leq t' \leq T : \abs{t'-t} \leq \delta, (t,t'] \cap J_{X} = \emptyset}
\abs{\delta X_{t,t'}} < \epsilon/A,\\
&\sup_{t \in (J_{X} \cup J_{Y}) \cap [0,T]} \sup_{0 < s \leq \delta} \abs{X_{t-}-X_{t-s}} < \epsilon/(10 A N), \\
&\sup_{t \in (J_{X} \cup J_{Y}) \cap [0,T]} \sup_{0 < s \leq \delta} \abs{X_{t+s}-X_{t}} < \epsilon/(10 A N), \\
\end{align*}
and the sets
\begin{align*}
\Omega_{5} &:= \Set[\Big]{ \sup_{t \in (J_{X} \cup J_{Y}) \cap [0,T]} \sup_{0 < s \leq \delta} (\abs{\delta Y_{t-s,t-}} \bmax \abs{\delta Y'_{t-s,t-}} \bmax \abs{\delta g_{t-s,t-}}) < \epsilon/(100 A^{2} N) }, \\
\Omega_{6} &:= \Set[\Big]{ \sup_{t \in (J_{X} \cup J_{Y}) \cap [0,T]} \sup_{0 < s \leq \delta} (\abs{\delta Y_{t,t+s}} \bmax \abs{\delta g_{t,t+s}}) < \epsilon/(100 A^{2} N) }, \\
\Omega_{7} &:= \Set[\Big]{ \inf_{s, t \in (J_{X} \cup J_{Y}) \cap [0,T], s \neq t} \abs{s-t} > \delta },\\
\Omega_{8} &:= \Set[\Big]{ \sup_{t \leq t' \leq T : \abs{t'-t} \leq \delta, (t,t'] \cap J_{Y} = \emptyset}
\abs{\delta Y_{t,t'}} < \epsilon },
\end{align*}
have probability $\geq 1 - \epsilon$.
Let $\pi$ be a deterministic partition with $\mesh(\pi) < \delta$.

The basic idea to handle the main term is the following.
Suppose $\omega \in \Omega_{1}\cap\dotsb\cap\Omega_{8}$ and $s \in J_{X} \cup J_{Y}(\omega)$.
Suppose $\pi_{j} < s \leq \pi_{j+1} \bmin T$.
Then
\begin{align*}
\MoveEqLeft
\abs[\Big]{ Z_{\pi_{j}}\delta Y_{\pi_{j},\pi_{j+1} \bmin T} \delta g_{\pi_{j},\pi_{j+1} \bmin T} - Z_{s-} \Delta Y_{s} \Delta g_{s} }
\\ &\leq
\abs{Z_{\pi_{j}}-Z_{s-}} \cdot \abs{\delta Y_{\pi_{j},\pi_{j+1} \bmin T} \delta g_{\pi_{j},\pi_{j+1} \bmin T}}
+
\abs{Z_{s-}} \cdot \abs{\delta Y_{\pi_{j},\pi_{j+1} \bmin T} - \Delta Y_{s}} \cdot \abs{\delta g_{\pi_{j},\pi_{j+1} \bmin T}}
\\ & \quad+
\abs{Z_{s-} \Delta Y_{s}} \cdot \abs{\delta g_{\pi_{j},\pi_{j+1} \bmin T} - \Delta g_{s}}
\\ & =
\abs{Z_{\pi_{j}}-Z_{s-}} \cdot \abs{\delta Y_{\pi_{j},\pi_{j+1} \bmin T} \delta g_{\pi_{j},\pi_{j+1} \bmin T}}
+
\abs{Z_{s-}} \cdot \abs{\delta Y_{\pi_{j},s-} + \delta Y_{s,\pi_{j+1}\bmin T}} \cdot \abs{\delta g_{\pi_{j},\pi_{j+1} \bmin T}}
\\ & \quad+
\abs{Z_{s-} \Delta Y_{s}} \cdot \abs{\delta g_{\pi_{j},s-} + \delta g_{s,\pi_{j+1} \bmin T}}
\\ &\leq
3 \cdot (2A)^{2} \cdot 2\epsilon/(100 A^{2} N)
\\ &\leq
\epsilon/(4 N).
\end{align*}
In case $s \in J_{Y}(\omega) \setminus J_{X}$, we similarly estimate
\begin{align*}
\MoveEqLeft
\abs{ Z_{s-}Y'_{s-}\Delta X_{s} \Delta g_{s} - Z_{\pi_{j}}Y'_{\pi_{j}}\delta X_{\pi_{j},\pi_{j+1}\bmin T} \delta g_{\pi_{j},\pi_{j+1}\bmin T}}
\\ &\leq
\abs{ Z_{s-}-Z_{\pi_{j}} } \cdot \abs{Y'_{s-}\Delta X_{s} \Delta g_{s}}
+
\abs{Z_{\pi_{j}}} \cdot \abs{Y'_{s-} - Y'_{\pi_{j}}} \cdot \abs{ \Delta X_{s} \Delta g_{s} }
\\ &\quad
+
\abs{Z_{\pi_{j}} Y'_{\pi_{j}}} \cdot \abs{\Delta X_{s} - \delta X_{\pi_{j},\pi_{j+1}\bmin T}} \cdot \abs{ \Delta g_{s} }
+
\abs{ Z_{\pi_{j}}Y'_{\pi_{j}} \delta X_{\pi_{j},\pi_{j+1}\bmin T}} \cdot \abs{ \Delta g_{s} - \delta g_{\pi_{j},\pi_{j+1}\bmin T} }
\\ &\lesssim
\epsilon/N.
\end{align*}
Since $\omega \in \Omega_{4}$, these errors contribute $O(\epsilon)$ to the sum over $j$.
Hence, we obtain
\begin{align*}
\MoveEqLeft
\abs{ \sum_{\pi_{j} < T} Z_{\pi_{j}} \delta Y_{\pi_{j},\pi_{j+1} \bmin T} \delta g_{\pi_{j},\pi_{j+1} \bmin T} - \sum_{s \leq T} Z_{s-}\Delta Y_{s} \Delta g_{s}}
\\ &\leq
\abs[\bigg]{ \sum_{\substack{\pi_{j} < T \\ (\pi_{j}, \pi_{j+1}] \cap J_{X} \neq \emptyset}} Z_{\pi_{j}} \delta Y_{\pi_{j},\pi_{j+1} \bmin T} \delta g_{\pi_{j},\pi_{j+1} \bmin T} - \sum_{s \leq T, s \in J_{X}} Z_{s-} \Delta Y_{s} \Delta g_{s}}
\\ &+
\abs[\bigg]{ \sum_{\substack{\pi_{j} < T \\ (\pi_{j}, \pi_{j+1}] \cap J_{X} = \emptyset}} Z_{\pi_{j}} \delta Y_{\pi_{j},\pi_{j+1} \bmin T} \delta g_{\pi_{j},\pi_{j+1} \bmin T} - \sum_{s \leq T, s \not\in J_{X}} Z_{s-} \Delta Y_{s} \Delta g_{s}}
\\ &\leq
\abs{J_{X}} \epsilon/(10 N)
+ \abs[\bigg]{\sum_{s \leq T, s \not\in J_{X}} Z_{s-}Y'_{s-} \Delta X_{s} \Delta g_{s}}
\\ &+
\abs[\bigg]{ \sum_{\substack{\pi_{j} < T \\ (\pi_{j}, \pi_{j+1}] \cap J_{X} = \emptyset}} Z_{\pi_{j}} \delta Y_{\pi_{j},\pi_{j+1} \bmin T} \delta g_{\pi_{j},\pi_{j+1} \bmin T} - \sum_{s \leq T, s \not\in J_{X}} Z_{s-} \Delta R_{s} \Delta g_{s}}.
\end{align*}
The last line is estimated by
\begin{align*}
&
\sum_{s \leq T, s \not\in (J_{X} \cup J_{Y})} \abs{ Z_{s-} \Delta R_{s} \Delta g_{s}}
+
\Big\lvert \sum_{\substack{\pi_{j} < T \\ \mathclap{(\pi_{j}, \pi_{j+1}] \cap (J_{X} \cup J_{Y}) = \emptyset}}} Z_{\pi_{j}} \delta Y_{\pi_{j},\pi_{j+1} \bmin T} \delta g_{\pi_{j},\pi_{j+1} \bmin T}
\\ &+
\sum_{\substack{\pi_{j} < T \\ \mathclap{(\pi_{j}, \pi_{j+1}] \cap (J_{Y} \setminus J_{X}) \neq \emptyset}}} Z_{\pi_{j}}\delta Y_{\pi_{j},\pi_{j+1} \bmin T} \delta g_{\pi_{j},\pi_{j+1} \bmin T}
- \sum_{s \leq T, s \in (J_{Y} \setminus J_{X})} Z_{s-}\Delta R_{s} \Delta g_{s} \Big\rvert
\\ &=
\sum_{s \leq T, s \not\in (J_{X} \cup J_{Y})} \abs{ Z_{s-}\Delta R_{s} \Delta g_{s}}
+
\Big\lvert \sum_{\substack{\pi_{j} < T \\ \mathclap{(\pi_{j}, \pi_{j+1}] \cap (J_{X} \cup J_{Y}) = \emptyset}}} Z_{\pi_{j}} \delta Y_{\pi_{j},\pi_{j+1} \bmin T} \delta g_{\pi_{j},\pi_{j+1} \bmin T}
\\ &+
\sum_{\substack{\pi_{j} < T \\ (\pi_{j}, \pi_{j+1}] \cap (J_{Y} \setminus J_{X}) \neq \emptyset}} Z_{\pi_{j}}Y'_{\pi_{j}}\delta X_{\pi_{j},\pi_{j+1}\bmin T} \delta g_{\pi_{j},\pi_{j+1} \bmin T}
\Big\rvert
+O(\epsilon)
\\ &=
\sum_{s \leq T, s \not\in (J_{X} \cup J_{Y})} \abs{ Z_{s-} \Delta R_{s} \Delta g_{s}}
+
\Big\lvert \sum_{\substack{\pi_{j} < T \\ \mathclap{(\pi_{j}, \pi_{j+1}] \cap (J_{X} \cup J_{Y}) = \emptyset}}} Z_{\pi_{j}} R_{\pi_{j},\pi_{j+1} \bmin T} \delta g_{\pi_{j},\pi_{j+1} \bmin T}
\\ &+
\sum_{\substack{\pi_{j} < T \\ (\pi_{j}, \pi_{j+1}] \cap J_{X} = \emptyset}} Z_{\pi_{j}} Y'_{\pi_{j}} \delta X_{\pi_{j},\pi_{j+1}\bmin T} \delta g_{\pi_{j},\pi_{j+1} \bmin T}
\Big\rvert
+O(\epsilon)
\\ &\leq
\sum_{s \leq T, s \not\in (J_{X} \cup J_{Y})} \abs{ Z_{s-} \Delta R_{s} \Delta g_{s}}
+
\sum_{\substack{\pi_{j} < T \\ (\pi_{j}, \pi_{j+1}] \cap (J_{X} \cup J_{Y}) = \emptyset}} \abs{ Z_{\pi_{j}} R_{\pi_{j},\pi_{j+1} \bmin T} \delta g_{\pi_{j},\pi_{j+1} \bmin T} }
\\ &+
\abs[\Big]{ \sum_{\substack{\pi_{j} < T \\ (\pi_{j}, \pi_{j+1}] \cap J_{X} = \emptyset}} Z_{\pi_{j}} Y'_{\pi_{j}} \delta X_{\pi_{j},\pi_{j+1}\bmin T} \delta g_{\pi_{j},\pi_{j+1} \bmin T} }
+O(\epsilon).
\end{align*}
These estimates are uniform in $T$, so we obtain
\begin{align*}
\MoveEqLeft
\sup_{T \leq T_{0}} \abs{ \sum_{\pi_{j} < T} Z_{\pi_{j}} \delta Y_{\pi_{j},\pi_{j+1} \bmin T} \delta g_{\pi_{j},\pi_{j+1} \bmin T} - \sum_{s \leq T} Z_{s-} \Delta Y_{s} \Delta g_{s}}
\\ &\leq
\sup_{T \leq T_{0}} \abs[\Big]{\sum_{s \leq T, s \not\in J_{X}} Z_{s-} Y'_{s-} \Delta X_{s} \Delta g_{s}}
+
\sum_{s \leq T_{0}, s \not\in (J_{X} \cup J_{Y})} \abs{ Z_{s-} \Delta R_{s} \Delta g_{s}}
\\ &+
\sum_{\substack{\pi_{j} < T_{0} \\ (\pi_{j}, \pi_{j+1}] \cap (J_{X} \cup J_{Y}) = \emptyset}}
\abs{ Z_{\pi_{j}} R_{\pi_{j},\pi_{j+1} \bmin T} \delta g_{\pi_{j},\pi_{j+1} \bmin T} }
\\ &+
\sup_{T \leq T_{0}} \abs[\Big]{ \sum_{\substack{\pi_{j+1} \leq T \\ (\pi_{j}, \pi_{j+1}] \cap J_{X} = \emptyset}} Z_{\pi_{j}} Y'_{\pi_{j}}\delta X_{\pi_{j},\pi_{j+1}} \delta g_{\pi_{j},\pi_{j+1}} }
\\ &+
\sup_{\substack{j : \pi_{j+1} \leq T_{0},\\ (\pi_{j}, \pi_{j+1}] \cap J_{X} = \emptyset}}
\sup_{T \in (\pi_{j},\pi_{j+1})} \abs{Z_{\pi_{j}} Y'_{\pi_{j}} \delta X_{\pi_{j},T} \delta g_{\pi_{j},T} }
+O(\epsilon).
\end{align*}
The contribution of the sums involving $Y'$ is $O(\epsilon)$ in the space $L^{q}(V^{r})$ for any $r>2$ by Theorem~\ref{thm:bracket-contr+mart:bound}, since $\abs{\Delta X_{s}} = O(\epsilon)$ and $\delta X_{\pi_{j},\pi_{j+1}} = O(\epsilon)$ in all summands.
The contribution of the supremum involving $Y'$ is easy to bound, again because $\delta X = O(\epsilon)$ there.

The contribution of the sums involving $R$ is bounded by
\[
(\sum_{j} \abs{R_{\dots}}^{2} )^{1/2} (\sum_{j} \abs{\delta g_{\dots}}^{2} )^{1/2}
\leq
(\sup_{j} \abs{R_{\dots}})^{1-\hat{p}_{1}/2} ( V^{\hat{p}_{1}} R)^{\hat{p}_{1}/2}
(\sum_{j} \abs{\delta g_{\dots}}^{2} )^{1/2}.
\]
Using that $\abs{R_{\dots}} = O(\epsilon)$ in all these terms and the BDG inequality to estimate the square function of $g$, we see that the contribution of these terms is $O(\epsilon^{1-\hat{p}_{1}/2})$ in $L^{q}$.
\end{proof}

\begin{proof}[Proof of Theorem~\ref{thm:RSMlift}]
By Corollary~\ref{cor:Ito-mesh-convergence}, Theorem ~\ref{thm:[Y,g]-discrete-approx}, and \eqref{eq:int-controlled-controlled-discrete-sums}, we have
\[
\Pi (\mathrm{W},\mathrm{\bar W})_{0,\cdot}
=
\ucplim_{\dmesh(\pi)\to 0} \Bigl( \sum_{\pi_{j} < t} \delta W_{0,\pi_{j}} \delta\bar{W}_{\pi_{j},\pi_{j+1}\bmin t} + Y'_{\pi_{j}} \bar Y'_{\pi_{j}} \bbX_{\pi_{j},\pi_{j+1} \bmin t} \Bigr)_{t}.
\]
This expression does not depend on the decompositions of $W,\bar{W}$, hence $\Pi(W,\bar{W})$ is well-defined.

In order to see that $\Pi(W,\bar{W})$ has locally bounded $p/2$-variation, we localize as in \eqref{eq:bracket-localizing-sequence}, also making sure that $Y',\bar{Y}' \in L^{\infty}(V^{p})$ by a minor variation of that argument.
With these finite moments assumptions, boundedness of the $p/2$-variation is given by Theorems~\ref{thm:main} and~\ref{thm:dX} as well as the sewing lemma.
\end{proof}

\subsection{Integration by parts}
\label{sec:int-by-parts}
The following estimate will be used for boundary terms.
\begin{lemma}
\label{lem:discretized-product}
Let $0 < q_{0},q_{1} \leq \infty$ and $1/q = 1/q_{0} + 1/q_{1}$.
Let $0 < p_{0},p_{1} \leq \infty$ and $1/r < 1/p_{0} + 1/p_{1}$.
Let $f,g$ be \cadlag{} adapted processes.
Then
\begin{align*}
\MoveEqLeft
\norm[\Big]{ V^{r} \bigl( \delta f_{t,t'} \delta g_{t,t'} \bigr) }_{L^{q}}
\\ &\leq
\sup_{\tau} \norm{ \sup_{\tau_{k-1} \leq t < \tau_{k}} \abs{ f_{t} - f_{\tau_{k}} } }_{L^{q_{1}}(\ell^{p_{1}})}
\norm{ \sup_{\tau_{k-1} \leq t < \tau_{k}} \abs{ g_{t} - g_{\tau_{k}} } }_{L^{q_{0}}(\ell^{p_{0}})}.
\end{align*}
where the supremum is taken over adapted partitions $\tau$.
\end{lemma}
\begin{proof}
This is a direct consequence of Corollary~\ref{cor:paraprod-stopping} with $1/r < 1/\rho = 1/p_{0} + 1/p_{1}$ and H\"older's inequality.
\end{proof}

\begin{corollary}
\label{cor:discretized-product}
Let $1 \leq q_{0} < \infty$, $0 < q_{1} \leq \infty$, and $1/q = 1/q_{0} + 1/q_{1}$.
Let $0 < p_{1} \leq \infty$ and $1/r < 1/2 + 1/p_{1}$.
Let $f$ be a \cadlag{} adapted process and $g$ a \cadlag{} martingale.
Then
\begin{equation}
\norm[\Big]{ V^{r} \bigl( \delta f_{t,t'} \delta g_{t,t'} \bigr) }_{L^{q}}
\lesssim
\norm{ V^{p_{1}} f }_{L^{q_{1}}}
\norm{ V^{\infty} g }_{L^{q_{0}}}.
\end{equation}
\end{corollary}
\begin{proof}
We apply Lemma~\ref{lem:discretized-product} with $p_{0}=2$.
The resulting $L^{q_{0}}(\ell^{2})$ norm can be estimated, after discretization, using first the vector-valued and then the scalar-valued BDG inequality.
\end{proof}

\begin{proof}[Proof of Theorem~\ref{thm:dX}]
For any adapted partition $\pi$ and any \cadlag{} processes $f,g$, we have the summation by parts identity
\begin{equation}
\label{eq:summation-by-parts}
\hPi^{\pi}(f,g)_{0,T} + \hPi^{\pi}(g,f)_{0,T} + [f,g]^{\pi}_{T} = (f_{T}-f_{0})(g_{T}-g_{0}).
\end{equation}
Define
\begin{equation}
\label{eq:Pi-g-Y}
\cPi(g,\rmY) := \delta g \delta Y - \hPi(Y,g) - \delta [\rmY,g].
\end{equation}
Convergence \eqref{eq:9} then follows from Corollary~\ref{cor:Ito-mesh-convergence} and Theorem~\ref{thm:[Y,g]-discrete-approx}.

Chen's relation \eqref{eq:Chen-g-dY} follows from Chen's relation \eqref{eq:Chen} for $\hPi(Y,g)$.

The variation norm bound \eqref{eq:controlled-integrator-BDG-limit} follows from Corollary~\ref{cor:discretized-product}, part~\ref{it:m2} of Theorem~\ref{thm:main}, and Theorem~\ref{thm:bracket-contr+mart:bound} applied to the respective terms.
\end{proof}

\subsection{Quadratic covariation of two martingales}

In this section, we recall a few facts about quadratic covariation needed in Section~\ref{sec:consistency} and explain how they fit into the approach to It\^o integration provided by Theorem~\ref{thm:main}.

Let $f,g$ be \cadlag{} martingales.
One way to define the quadratic covariation process of $f,g$ is by
\begin{equation}
\label{eq:martingale-bracket-via-Pi}
[f,g]_{t} := \delta f_{0,t} \delta g_{0,t} - \hPi(f,g)_{0,t} - \hPi(g,f)_{0,t}.
\end{equation}
Corollary~\ref{cor:Ito-mesh-convergence} and the summation by parts identity \eqref{eq:summation-by-parts} then recover the description of the quadratic covariation in terms of discrete brackets:
\begin{equation}
\label{eq:martingale-bracket-discretized}
\begin{split}
[f,g]_{t}
&=
\ucplim_{\mesh(\pi)\to 0} \delta f_{0,t} \delta g_{0,t} - \hPi^{\pi}(f,g)_{0,t} - \hPi^{\pi}(g,f)_{0,t}
\\ &=
\ucplim_{\mesh(\pi)\to 0} [f,g]^{\pi}_{0,t}.
\end{split}
\end{equation}
In particular, in the case $g=f$, the function $t \mapsto [g]_{t} := [g,g]_{t}$ is a.s.\ monotonically increasing and locally bounded.
Passing to the limit in the vector-valued BDG inequality, Lemma~\ref{lem:vv-BDG}, we obtain the estimate
\begin{equation}
\label{eq:cadlag-vv-BDG}
\norm[\big]{ V^{\infty} h^{(k)} }_{L^q(\ell^r_k)}
\lesssim_{q,r}
\norm[\big]{ [h^{(k)}]^{1/2} }_{L^q(\ell^r_k)},
\end{equation}
where $h^{(k)}$ are \cadlag{} martingales, $[h] = [h]_{\infty} = \lim_{t\to\infty} [h,h]_{t}$, and the hypotheses on the exponents $q,r$ are the same as in Lemma~\ref{lem:vv-BDG}.

Finally, we recall the (almost sure, pathwise) It\^o isometry
\begin{equation}
\label{eq:bracket-Pi}
[\Pi(f,g)_{s,\cdot}]_{t}
=
\int_{(s,t]} \abs{f_{u-}-f_{s}}^{2} \dif [g]_{u},
\end{equation}
where the integral is taken in the Riemann--Stieltjes sense.

\section{Consistency of rough and stochastic integration}
\label{sec:consistency}
Let $g$ be a \cadlag{} local martingale and $\rmg=(g,\Pi(g,g))$ the $p$-rough path lift (with $p\in (2,3)$) provided by Theorem~\ref{thm:main} with $F=\delta g$.
It is well-known that, for any $g$-controlled $p$-rough adapted process $\rmA=(A,A')$, the It\^o integral and the rough integral coincide almost surely:
\begin{equation}
\label{eq:Ito=rough}
\int A_{u-} \dif g_{u} = \int \rmA_{u-} \dif \rmg_{u},
\end{equation}
see e.g.\ \cite[Proposition 5.1]{FH2020} for the case of Brownian motion and references given there for historical information.
We begin with a generalization of this fact, in which one of the copies of $g$ is replaced by a further process $Y$ and $Z$ plays the role of $A'$.
\begin{lemma}
\label{lem:consistency-PiYg}
Let $g$ be a \cadlag{} local martingale and $Y,Z$ \cadlag{} adapted processes.
Then, along adapted partitions $\pi$, we have
\begin{equation}
\label{eq:consistency-PiYg}
\ucplim_{\mesh(\pi) \to 0} \Bigl( \sum_{\pi_{j} < T} Z_{\pi_{j}} \Pi(Y,g)_{\pi_{j},\pi_{j+1} \bmin T} \Bigr)_{T} = 0.
\end{equation}
\end{lemma}
Lemma~\ref{lem:consistency-PiYg} generalizes \cite[Lemma 4.35]{MR3909973}, where additional structural hypotheses are made on $Y,Z$.
\begin{remark}
Lemma~\ref{lem:consistency-PiYg} is the main ingredient in showing consistency results such as \eqref{eq:Ito=rough}.
Indeed, the difference between the discrete approximations of the two sides of \eqref{eq:Ito=rough} is precisely the sum in \eqref{eq:consistency-PiYg}.
More generally, one can replace the rough lift $\rmg$ by a rough semimartingale $\rmg+\tilde{g}$, where $\tilde{g}$ is independent from $g$, and the controlled process $\rmA$ by another process that is a $g$-controlled rough semimartingale conditionally on each path of $g$.
\end{remark}
\begin{proof}[Proof of Lemma~\ref{lem:consistency-PiYg}]
Without loss of generality, $Y_{0}=0$.
Multiplying $Z$ by an $\calF_{0}$-measurable time-independent function, we may also assume $\abs{Z_{0}} \leq 1$.
Similarly to \eqref{eq:bracket-localizing-sequence}, we may assume
\[
g \in L^{1}V^{\infty},
\quad
MY,MZ \in L^{\infty}.
\]
By the BDG inequality and It\^o isometry \eqref{eq:bracket-Pi}, we have
\begin{align*}
\bbE \sup_{T} \abs[\Big]{ \sum_{\pi_{j}< T} Z_{\pi_{j}} \Pi(Y,g)_{\pi_{j},\pi_{j+1} \bmin T} }
&\sim
\bbE \Bigl[ \sum_{j} Z_{\pi_{j}} \Pi(Y,g)_{\pi_{j},\pi_{j+1}} \Bigr]^{1/2}
\\ &=
\bbE \Bigl( \sum_{j} \int_{(\pi_{j},\pi_{j+1}]} \abs{Z_{\pi_{j}} \delta Y_{\pi_{j},u-}}^{2} \dif [g]_{u} \Bigr)^{1/2}
\\ &\lesssim
\bbE \Bigl( \int_{(0,T]} \abs{\delta Y_{\floor{u-,\pi},u-}}^{2} \dif [g]_{u} \Bigr)^{1/2}.
\end{align*}
We will use the dominated convergence theorem to show that this converges to $0$.
First, we note that
\[
\int_{(0,T]} \abs{\delta Y_{\floor{u-,\pi},u-}}^{2} \dif [g]_{u}
\leq
(V^{\infty} Y)^{2} \delta [g]_{0,T},
\]
which gives us the integrable pointwise upper bound.
It remains to show that, almost surely,
\begin{equation}
\label{eq:sup-[]-mesh}
\lim_{\delta\to 0} \sup_{\mesh(\pi) \leq \delta} \int_{(0,T]} \abs{\delta Y_{\floor{u-,\pi},u-}}^{2} \dif [g]_{u} = 0.
\end{equation}
The supremum over all partitions with a given bound on mesh is necessary here, since the analogue of the dominated convergence theorem is false for nets.
To see \eqref{eq:sup-[]-mesh}, take $\omega\in\Omega$ such that the function $u \mapsto [g]_{u}$ is monotonically increasing and bounded on $[0,T]$ (this is true a.s.).
Let $\epsilon>0$ be arbitrary.
By the \cadlag{} property of $Y$, there are finitely many points $(s_{k})$ such that $\abs{\Delta Y_{s_{k}}} \geq \epsilon$, and there exists $\delta>0$ such that $V^{\infty}Y|_{(s_{k}-\epsilon,s_{k})} < \epsilon$, $V^{\infty}Y|_{[s_{k},s_{k}+\epsilon]} < \epsilon$, and for every interval $J$ such that $s_{k} \not\in J$ for all $k$ we have $V^{\infty}Y|_{J} < \epsilon$.
It follows that, for every partition $(\pi)$ with $\mesh(\pi) < \delta$, we have
\begin{equation}
\label{eq:16}
\begin{split}
\int_{(0,T]} \abs{\delta Y_{\floor{u-,\pi},u-}}^{2} \dif [g]_{u}
&\lesssim
\epsilon^{2} \int_{(0,T]} \dif [g]_{u}
+
\sum_{k} \abs{\Delta Y_{s_{k}}}^{2} \int_{(s_{k},s_{k}+\delta)} \dif [g]_{u}.
\\ &\leq
\epsilon^{2} \int_{(0,T]} \dif [g]_{u}
+
\sum_{k} \abs{\Delta Y_{s_{k}}}^{2} \delta [g]_{s_{k}+,s_{k}+\delta}.
\end{split}
\end{equation}
The first term is clearly arbitrarily small, and the second term also becomes arbitrarily small as $\delta$ decreases because the sum is finite and $u \mapsto [g]_{u}$ is monotonic.
\end{proof}

\begin{lemma}
\label{lem:consistency-PigY}
Let $\hat{p}_{1} < 2 \leq p_{1}$.
Let $X\in V^{p_{1}}_{\loc}$ be a deterministic \cadlag{} path.
Suppose that $\rmY = (Y,Y')$ is a \cadlag{} adapted process, $Z$ a \cadlag{} adapted process, $g$ a \cadlag{} local martingale, $R^{\rmY,X} \in V^{\hat{p}_{1}}_{\loc}$ a.s..
Then
\[
\ucplim_{\dmesh(\pi) \to 0} \Bigl( \sum_{\pi_{j} < T} Z_{\pi_{j}} \Pi(g,\rmY)_{\pi_{j},\pi_{j+1}\bmin T} \Bigr)_{T} = 0.
\]
\end{lemma}
\begin{remark}
If $Y$ is a \cadlag{} process with a.s.\ locally bounded $1$-variation, then one can take $X=0$, $Y'=0$ in Lemma~\ref{lem:consistency-PigY}.
\end{remark}
\begin{proof}
By definition \eqref{eq:Pi-g-Y}, we have
\begin{multline*}
\sum_{\pi_{j} < T} Z_{\pi_{j}} \Pi(g,\rmY)_{\pi_{j},\pi_{j+1}\bmin T}
=
\sum_{\pi_{j} < T} Z_{\pi_{j}} \delta Y_{\pi_{j},\pi_{j+1}\bmin T} \delta g_{\pi_{j},\pi_{j+1}\bmin T}
\\-
\sum_{\pi_{j} < T} Z_{\pi_{j}} [\rmY,g]_{\pi_{j},\pi_{j+1}\bmin T}
-
\sum_{\pi_{j} < T} Z_{\pi_{j}} \Pi(Y,g)_{\pi_{j},\pi_{j+1}\bmin T}.
\end{multline*}
The last term on the right-hand side converges to zero by Lemma~\ref{lem:consistency-PiYg}.

The first term on the right-hand side is, by Definition~\ref{def:discrete-bracket-2}, equal to $Z \bullet [Y,g]^{\pi}$.
By Theorem~\ref{thm:[Y,g]-discrete-approx}, it converges to $Z \bullet [\rmY,g]$.

The middle term equals $Z^{(\pi)} \bullet [\rmY,g]$.
This also converges to $Z \bullet [\rmY,g]$ as $\mesh(\pi) \to 0$ by an argument similar to \eqref{eq:16}.
\end{proof}

If $(g+Y,Y')$ is an $X$-controlled $p$-RSM, $p\in (2,3)$, then $\rmZ=(Z,Z')$ with
\begin{equation}
\label{eq:RSM=>controlled}
Z = g+Y,
\quad
Z'_{t}(\delta X,\delta g) = Y'_{t} \delta X + \delta g
\end{equation}
is easily seen to be an $(X,g)$-controlled $p$-rough process.
Indeed, $g \in V^{p}_{\loc}$ almost surely by Lemma~\ref{lem:L1-localizing-sequence} and L\'epingle's inequality~\eqref{eq:Lepingle}.
It remains to observe that
\begin{align*}
R^{\rmZ,(X,g)}_{s,t}
&=
\delta Z_{s,t} - Z'_{t}(\delta X_{s,t},\delta g_{s,t})
\\ &=
\delta g_{s,t} + \delta Y_{s,t} - Y'_{t}\delta X_{s,t} - \delta g_{s,t}
\\ &=
R^{\rmY,X}_{s,t}.
\end{align*}
The converse implication is more subtle, because the $g$ component of the Gubinelli derivative of a $(X,g)$-controlled process need not be the identity.

\begin{theorem}
\label{thm:controlled=>RSM}
Let $p\in (2,3)$ and $X\in V^{p}_{\loc}$ be a deterministic \cadlag{} path.
Let $g$ be a \cadlag{} local martingale.
Let $\rmZ = (Z,Z')$ be an adapted \cadlag{} $(X,g)$-controlled $p$-rough process.

Then $(Z,Z'(\cdot,0))$ is an $X$-controlled $p$-RSM:
\[
(Z,Z'(\cdot,0)) = (\tilde{g}+\tilde{Y},\tilde{Y}'),
\]
with the local martingale part given by
\begin{equation}
\label{eq:F(RSM):mart-part}
\tilde{g}_{T} := \Pi(Z'(0,\cdot),g)_{0,T}
\end{equation}
and Gubinelli derivative
\begin{equation}
\label{eq:F(RSM):derivative}
\tilde{Y}'_{T} := Z'_{T}(\cdot,0).
\end{equation}
\end{theorem}

\begin{proof}[Proof of Theorem~\ref{thm:controlled=>RSM}]
With the local martingale component defined by \eqref{eq:F(RSM):mart-part}, the controlled rough component will be defined by
\[
\tilde{Y}_{T} := Z_{T} - \tilde{g}_{T}.
\]
It follows from L\'epingle's inequality~\eqref{eq:Lepingle} and localization, Lemma~\ref{lem:L1-localizing-sequence}, that $\tilde{Y} \in V^{p}_{\loc}$ almost surely.
It remains to show that $R^{\tilde{\rmY},X} \in V^{p/2}_{\loc}$ almost surely.
To this end, with $s<t$, we write
\begin{equation}
\label{eq:13}
\begin{split}
R^{\tilde{\rmY},X}_{s,t}
&=
\tilde{Y}_{t} - \tilde{Y}_{s} - Z'_{s}(X_{t}-X_{s},0)
\\ &=
Z_{t} - Z_{s} - \Pi(Z'(0,\cdot),g)_{0,t} + \Pi(Z'(0,\cdot),g)_{0,s} - Z'_{s}(X_{t}-X_{s},0)
\\ &=
\Bigl(Z_{t} - Z_{s} - Z'_{s} (X_{t}-X_{s},g_{t}-g_{s}) \Bigr)
\\ &\quad
- \Pi(Z'(0,\cdot),g)_{0,t} + \Pi(Z'(0,\cdot),g)_{0,s} + Z'_{s}(0,g_{t}-g_{s})
\\ &=
R^{\rmZ,(X,g)}_{s,t} - \Pi(Z'(0,\cdot),g)_{s,t}.
\end{split}
\end{equation}
The former term is in $V^{p/2}_{\loc}$ by the hypothesis.
The latter term is in $V^{p/2}_{\loc}$ by Theorem~\ref{thm:main} and localization similar to Lemma~\ref{lem:controlled-localizing-sequence}.
\end{proof}

\begin{corollary}\label{cor:F(RSM)}
Let $p \in (2,3)$.
If $(g+Y,Y')$ is an $X$-controlled, $p$-rough semimartingale and $\sigma \in C^{2}$, then $(\sigma(g+Y),D\sigma \circ Y')$ is also an $X$-controlled $p$-rough semimartingale.
\end{corollary}
\begin{proof}
By \eqref{eq:RSM=>controlled}, $g+Y$ can be lifted to an $(X,g)$-controlled $p$-rough process.
The composition of this process with $\sigma$ is again an $(X,g)$-controlled $p$-rough path, see e.g.\ \cite[Remark 4.15]{MR3770049}, to which we can apply Theorem~\ref{thm:controlled=>RSM}.
\end{proof}

\begin{remark}
Theorem~\ref{thm:controlled=>RSM} has an analog for classical semimartingales.
Let $g$ be a \cadlag{} local martingale and $\rmZ = (Z,Z')$ a \cadlag{} adapted process such that $R^{\rmZ,g} \in V^{1}_{\loc}$ and $Z' \in V^{2}_{\loc}$.
Then $Z$ must be a semimartingale. Indeed,
let
\[
\tilde{g}_{T} := \Pi(Z',g)_{T},
\quad
Y_{T} := Z_{T} - \tilde{g}_{T},
\quad
Y'_{T} := 0.
\]
Then, by the same calculation as in \eqref{eq:13}, we have
\[
\delta Y_{s,t} = R^{\rmY,0}_{s,t} = - \Pi(Z',g)_{s,t}.
\]
It follows from the $\ell^{1}$-valued estimate in Corollary~\ref{cor:vv-pprod-delta-F} that $Y \in V^{1}_{\loc}$, so that $Z$ is a semimartingale.
\end{remark}

\begin{theorem}
\label{thm:controlled-integral=RSM-integral}
Let $p\in (2,3)$ and $\rmX = (X,\bbX)$ be a deterministic \cadlag{} $p$-rough path.
Let $g$ be a \cadlag{} local martingale.
Let $\rmZ = (Z,Z')$ be an adapted \cadlag{} $(X,g)$-controlled $p$-rough process.
Then
\[
\int \rmZ \dif\calJ(\rmX,g) = \Pi(Z,(X,g)).
\]
where the left-hand side is the pathwise rough integral and the right-hand side is the RSM integral.
\end{theorem}

\begin{proof}
The right-hand side makes sense by Theorem~\ref{thm:controlled=>RSM}.
Expanding the definitions, we see that the difference between the two sides vanishes by Lemma~\ref{lem:consistency-PiYg} and Lemma~\ref{lem:consistency-PigY}.
\end{proof}

\begin{proof}[Proof of Theorem~\ref{thm:RDE-sol=>RSDE-sol}]
RDE theory yields a solution $(Z,(\sigma,\mu)(Z))$ as $(X,g)$-controlled $p$-rough process. By Theorem~\ref{thm:controlled=>RSM}, we see
that $(Z,\sigma(Z))$ is an $X$-controlled $p$-RSM, as is $(\sigma(Z), D\sigma(Z) \circ \sigma(Z))$ by Corollary~\ref{cor:F(RSM)}. To see the stated decomposition
into local martingale and rough drift part, we write the RDE solution as integral equation, obtained as mesh-limit of local approximations given by
\begin{align*}
\delta Z_{s, t} \cong & f_0 (Z_s) (\delta X)_{s, t} + f_{00} (Z_s) \bbX_{s, t} + f_1 (Z_s) (\delta g)_{s, t} \\
& + f_{01} (Z_s) \Pi (X, g)_{s,
t} + f_{10} (Z_s) \Pi (g, X)_{s, t} + f_{11} (Z_s) \Pi (g, g)_{s, t}
\end{align*}
where $f_0 = \sigma, f_1 = \mu, f_{00}= D\sigma \circ \sigma$ and so on. (Our assumptions on $\sigma,\mu$ imply that all the $f_{ij}$'s are bounded.)
It follows from Lemma~\ref{lem:consistency-PiYg} and~\ref{lem:consistency-PigY} that convergence still takes place when $f_{01},f_{10},f_{11}$ are set to zero, provided we restrict
ourselves to the mesh limit of deterministic partitions.
What remains are It\^o left-point sums, with $f_1$-terms, and u.c.p.\ It\^o limit $M = \int \mu(Z^-) \dif g$.
All these entails
convergence of sum with the remaining terms ($f_0$ and $f_{00}$), as given in the statement.
Alternatively, though equivalently, we can view $\int \sigma (Z^-) \dif \bfX$ as integral of a rough semimartingale against $(0 + X, \mathrm{Id})$, trivially another $X$-controlled rough semimartingale, hence rely on Theorem~\ref{thm:RSMlift}.
\end{proof}

\appendix
\section{H\"older estimates for martingale transforms} \label{sec:Happ}
For a two-parameter process $\Pi=(\Pi_{t,t'})_{0 \leq t < t' \leq T}$
and $\alpha \in [0,\infty)$, we set
\[
H^{\alpha}\Pi := \sup_{0 \leq t < t' \leq T} \frac{\abs{\Pi_{t,t'}}}{\abs{t'-t}^\alpha}.
\]

The following result is a H\"older version of the variational estimates of Theorem~\ref{thm:main}.
It improves upon the estimate given by Kolmogorov's theorem by eliminating the loss of $1/q$ in the Hölder exponent.
\begin{theorem}
\label{thm:H-Ito}
In the situation of Theorem~\ref{thm:main}, part~\ref{it:m2}, suppose that all processes have a.s.\ continuous paths and restrict the time parameter to a finite interval, $t \in [0,1]$.
Let
\[
0 \leq \gamma < \alpha + \beta = \alpha_{i} + \beta_{i}
\]
with $\alpha,\beta,\alpha_{i},\beta_{i} \geq 0$.
Then, we have
\[
\norm[\big]{ H^{\gamma} \Pi (F, g) }_{L^q}
\lesssim
\norm{ H^{\beta} F }_{L^{q_1}}
\norm{ H^{\alpha} (S g) }_{L^{q_0}}
+
\sum_{i=1}^{i_{\max}} \norm[\big]{ H^{\alpha_{i}}F^{i} \cdot H^{\beta_{i}} \Pi(\tilde{F}^{i},g) }_{L^{q}}
\]
\end{theorem}
\begin{proof}
We abbreviate $\bbX := \Pi(F,g)$.

Consider the deterministic partitions $\tau^{(n)} = 2^{-n}\N$, $\tilde{\tau}^{(n)} = \Set{0,1} \cup (2^{-n} \N + 2^{-n-1})$.
Let
\[
\bbK_{n} := \sup_{j \in \N} \sup_{\tau^{(n)}_{j-1} \leq t \leq t' \leq \tau^{(n)}_{j}} \abs{\bbX_{t,t'}},
\]
and define $\tilde{\bbK}_{n}$ analogously with $\tilde{\tau}^{(n)}$ in place of $\tau^{(n)}$.
Then, we have
\[
\sup_{\abs{t-t'} \leq 2^{-n-1}} \abs{\bbX_{t,t'}}
\leq
\bbK_{n} + \tilde{\bbK}_{n},
\quad
\sup_{\abs{t-t'} \leq 1} \abs{\bbX_{t,t'}} \leq \bbK_{0}.
\]
It follows that
\[
\sup_{\abs{t-t'} \leq 2^{-n-1}} \abs{t-t'}^{-\gamma} \abs{\bbX_{t,t'}}
\lesssim
2^{n \gamma} \bbK_{n} + 2^{n\gamma} \tilde{\bbK}_{n}.
\]
Therefore,
\[
H^{\gamma} \bbX
\lesssim
\max_{n\in\N} 2^{\gamma n} (\bbK_{n} + \tilde{\bbK}_{n}).
\]
It follows that
\[
\norm{ H^{\gamma} \bbX }_{L^{q}}^{q}
\lesssim
\sum_{n=0}^{\infty} \bigl( 2^{\gamma n} \norm{ \bbK_{n} }_{L^{q}} \bigr)^{q}
+
\sum_{n=0}^{\infty} \bigl( 2^{\gamma n} \norm{ \tilde{\bbK}_{n} }_{L^{q}} \bigr)^{q}.
\]
The two sums are similar, so we only consider the first one.
Let $1<r<\infty$ be such that $\gamma + 1/r < \alpha + \beta$.
By Theorem~\ref{thm:vv-pprod}, which passes to the continuous time case, we have
\begin{align*}
2^{\gamma n} \norm{ \bbK_{n} }_{L^{q}}
&\leq
2^{\gamma n} \norm{ \ell^{r}_{j} \sup_{\tau^{(n)}_{j-1} \leq t \leq t' \leq \tau^{(n)}_{j}} \abs{\bbX_{t,t'}} }_{L^{q}}
\\ &\lesssim
2^{\gamma n} \sum_{i=1}^{i_{\max}} \norm[\big]{ \ell^{r}_{k} \bigl(\sup_{\tau^{(n)}_{k-1}\leq s < t \leq \tau^{(n)}_{k}} \abs{F^{i}_{\tau^{(n)}_{k-1},s}} \cdot \abs{\Pi(\tilde{F}^{i},g)_{s,t}} \bigr) }_{q}
\\ &+
2^{\gamma n} \norm[\big]{ \ell^{2r}_{k} \sup_{\tau^{(n)}_{k-1} \leq s<t \leq \tau^{(n)}_{k}} \abs{F_{s,t}} }_{q_{1}}
\norm{ \ell^{2r} Sg_{\tau^{(n)}_{k-1},\tau^{(n)}_{k}} }_{q_{0}}
\\ &\leq
2^{\gamma n} \sum_{i=1}^{i_{\max}} \norm[\big]{ H^{\beta_{i}}F^{i} \cdot H^{\alpha_{i}} \Pi(\tilde{F}^{i},g) \cdot \ell^{r}_{k} \abs{\tau^{(n)}_{k-1}-\tau^{(n)}_{k}}^{\alpha_{i}+\beta_{i}} }_{q}
\\ &+
2^{\gamma n} \norm[\big]{ H^{\beta}F \cdot \ell^{2r}_{k} \abs{\tau^{(n)}_{k-1} - \tau^{(n)}_{k}}^{\beta} }_{q_{1}}
\norm{ H^{\alpha}(Sg) \ell^{2r} \abs{\tau^{(n)}_{k-1}-\tau^{(n)}_{k}}^{\alpha} }_{q_{0}}
\\ &\lesssim
\sum_{i=1}^{i_{\max}} 2^{(\gamma+1/r-\alpha_{i}-\beta_{i}) n} \norm[\big]{ H^{\beta_{i}}F^{i} \cdot H^{\alpha_{i}} \Pi(\tilde{F}^{i},g) }_{q}
\\ &+
2^{(\gamma+1/r-\alpha-\beta) n} \norm[\big]{ H^{\beta}F }_{q_{1}} \norm{ H^{\alpha}(Sg) }_{q_{0}}.
\end{align*}
By the choice of $r$, this is summable in $n$.
\end{proof}

\medskip

\textbf{Acknowledgement:}
PKF has received funding from the European Research Council (ERC) under the European Union’s Horizon 2020 research and innovation programme (grant agreement No. 683164) and the DFG Research Unit FOR 2402.
PZ was partially funded by the Deutsche Forschungsgemeinschaft (DFG, German Research Foundation) under Germany's Excellence Strategy -- EXC-2047/1 -- 390685813.
We thank the anonymous referees for their detailed reports that helped to improve this article.

\printbibliography
\end{document}